\documentclass{article}

\RequirePackage[OT1]{fontenc}
\RequirePackage{amsthm,amsmath}
\RequirePackage[numbers]{natbib}
\RequirePackage[colorlinks,citecolor=blue,urlcolor=blue]{hyperref}

\usepackage{enumitem}
\usepackage{amsfonts,graphicx,bbm}
\usepackage{array,booktabs}
\usepackage[margin=1.25in]{geometry}

\def\P{{\mathbb P}}
\def\E{{\mathbb E}}

\newcommand{\var}{{\rm {Var}}}

\def\N{\mathbb N}
\def\R{\mathbb R}

\usepackage{bm}

\newcommand\comment[1]{{}}

\DeclareMathOperator*{\argmin}{arg\,min}

\numberwithin{equation}{section}
\theoremstyle{plain}
\newtheorem{theorem}{Theorem}[section]
\newtheorem{proposition}{Proposition}[section]
\newtheorem{lemma}{Lemma}[section]
\newtheorem{corollary}{Corollary}[section]

\theoremstyle{definition}
\newtheorem{definition}{Definition}[section]
\theoremstyle{remark}
\newtheorem{remark}{Remark}[section]
\newtheorem{example}{Example}[section]
\begin{document}

\title{{Phase Transitions in Genome-wide Association Studies and Categorical Variable Screenings}} %\thanksref{T1}} 

\author{Zheng Gao\thanks{Department of Statistics, University of Michigan, Ann Arbor, USA. This research is supported by NSF Grant DMS-1830293, Algorithms for Threat Detection.}\\{gaozheng@umich.edu}}

% \author{Chenlin\thanks{notes on author} \\ l.wind.wang@gmail.com\thanks{This is a test!}}
\date{\vspace{-5ex}}
% \title{A Experiment on \LaTeX \\ The Maketitle Experiement\thanks{notes on title}}

\maketitle

\begin{abstract}
Motivated by genome-wide association screening studies  (GWAS), we study high-dimensional marginal screenings of categorical variables 
% and multiple testing problems 
where test statistics have approximate chi-square distributions.
We characterize four new phase transitions in high-dimensional chi-square models, and derive the signal sizes necessary and sufficient for statistical procedures to simultaneously control false discovery (in terms of family-wise error rate or false discovery rate) and missed detection (in terms of family-wise non-discovery rate or false non-discovery rate) in large dimensions.
% 1) approximately, and 2) exactly select the set of relevant features in large dimensions.
Remarkably, degrees of freedom in the chi-square distributions do not affect the boundaries in all four phase transitions.
Several well-known procedures are shown to attain these boundaries.
Two new phase transitions are also identified in the Gaussian location model under one-sided alternatives.

We then elucidate on the nature of signal sizes in association tests by characterizing its relationship with marginal frequencies, odds ratio, and sample sizes in $2\times2$ contingency tables. % in common association tests for binary traits.
% Specifically, detection of rare genetic variants associated with the outcomes requires more samples compared to common variants. 
% Specifically, the difficulty in detecting associations between genetic variants and phenotype traits is quantified in terms of the signal sizes.
% Specifically, the amount of signal per sample is inherently smaller when rare variants are present. 
% In other words, 
% Further, case prevalence in the recruited subjects plays a role, while case prevalence in the population is irrelevant.
% We provide explicit formulas for power calculations, as well as software tools for finding the optimal study designs.
This allows us to illustrate an interesting manifestation of the phase transition phenomena in GWAS.
We also show, perhaps surprisingly, that given total sample sizes, balanced designs in such association studies rarely deliver optimal power. % for detecting rare genetic variants with high odds ratios.

% The phase transition phenomena are demonstrated with numerical simulations, and with empirical evidence from GWAS.
% data from the NHGRI-EBI GWAS catalog.
\end{abstract}

% \begin{keyword}[class=MSC]
% %\kwd[Primary ]{62G10, 62G20}
% \kwd{62G10, 62G20}
% %\kwd[; secondary ]{60K35}
% \end{keyword}

% \begin{keyword}
% \kwd{Multiple testing}
% \kwd{asymptotic optimality}
% \kwd{family-wise error rate (FWER) control}
% \kwd{false discovery rate (FDR) control}
% \kwd{categorical data analysis}
% \kwd{association tests}
% \end{keyword}
% \tableofcontents
% \end{frontmatter}

\section{Introduction}
\label{sec:intro}

% in the literature of high-dimensional statistics.

% This paper seeks to establish some new asymptotic optimality results for common multiple testing procedures, including the Bonferroni procedure and the Benjamini-Hochberg procedure, in high-dimensional chi-square models.

% \subsection{Genome-wide association studies and the chi-square model}
% \label{subsec:motivation-chisq}

In this paper, we study large-scale categorical variable screening problems, typified by genome-wide association studies (GWAS) where millions of categorical genetic factors are examined for their potential influence on phenotypic traits simultaneously.
% We briefly introduce the scientific background next, to provide some context.

% introduced through the language of GWAS next.
% These categorical covariate screening problems naturally give rise to the high-dimensional chi-square models, introduced through the language of GWAS next.
% 
Broadly speaking, GWAS aim to discover genetic variations that are linked to traits or diseases of interested, by testing for associations between the subjects' genetic compositions and their phenotypes \cite{bush2012genome}. 
In a typical GWAS with a case-control design, a total of $n$ subjects are recruited,  consisting of $n_1$ subjects possessing a defined trait, and $n_2$ subjects without the trait serving as controls.
The genetic compositions of the subjects are then examined for variations known as single-nucleotide polymorphisms (SNPs) at an array of $p$ genomic locations, and compared between the case and the control group.

Focusing on one specific genomic location, the counts of observed genotypes, if two variants are present, can be tabulated as follows.
% \begin{table}[ht]
% \centering
\begin{center}
    \begin{tabular}{cccc}
    \hline
    & \multicolumn{2}{c}{Genotype} & \\
    \cline{2-3}
    \# Observations & Variant 1 & Variant 2 & Total by phenotype \\
    \hline
    Cases & $O_{11}$ & $O_{12}$ & $n_1$ \\
    Controls & $O_{21}$ & $O_{22}$ & $n_2$ \\
    \hline
    \end{tabular}
    % \caption{Tabulated counts of genotype-phenotype combinations in a genetic association test.}
\end{center}
% \end{table}
Researchers test for associations between the genotypes and phenotypes using, for example, the Pearson chi-square test with statistic
\begin{equation} \label{eq:chisq-statistic}
    x = \sum_{j=1}^2 \sum_{k=1}^2 \frac{(O_{jk} - E_{jk})^2}{E_{jk}},
\end{equation}
where ${E}_{jk} = (O_{j1}+O_{j2})(O_{1k}+O_{2k})/n$.
% are the expected number of observations under the null.
%E_{jk} = \Big(\sum_{l}O_{jl}\Big)\Big(\sum_{l}O_{lk}\Big)\Big/n.

Under the mild assumption that the counts $O_{jk}$'s 
% follow a multinomial distribution (or a product-binomial distribution, if we decide to condition on one of the marginals), 
the statistic $x$ in \eqref{eq:chisq-statistic} can be shown to have an approximate $\chi^2(\lambda)$ distribution with $\nu=1$ degree of freedom at large sample sizes \citep{agresti2018introduction}. 
Independence between the genotypes and phenotypes would imply a non-centrality parameter $\lambda$ value of zero; if dependence exists, we would have a non-zero $\lambda$ where its value depends on the underlying multinomial probabilities.
More generally, if we have a $J$ phenotypes and $K$ genetic variants, assuming a $J\times K$ multinomial distribution, the statistic will follow approximately a $\chi^2_{\nu}(\lambda)$ distribution with $\nu = (J-1)(K-1)$ degrees of freedom, when sample sizes are large.

% The same asymptotic distributional approximations also apply to the likelihood ratio statistic, and many other statistics under slightly different modeling assumptions \cite{gao2019upass}.
These association tests are performed at each of the $p$ locations of interest throughout the whole genome, yielding $p$ statistics having approximately (non-)central chi-square distributions, $\chi_{\nu(i)}^2\left(\lambda(i)\right)$, for $i=1,\ldots,p$,
% \begin{equation} \label{eq:model-chisquare-approx}
%     x(i) \mathrel{\dot\sim} \chi_{\nu(i)}^2\left(\lambda(i)\right), \quad i=1,\ldots,p,
% \end{equation}
where $\lambda = (\lambda(i))_{i=1}^p$ is the $p$-dimensional non-centrality parameter.
%, with $\lambda(i)=0$ indicating independence of the $i$-th SNP with the outcomes, and $\lambda(i)\neq0$ indicating associations.

In the following, we will study the statistical limits of multiple testing in an idealized model where the statistics follow independent chi-square distributions,
\begin{equation} \label{eq:model-chisq}
    %x(i) \distras{\mathrm{ind.}} \chi_\nu^2\left(\lambda(i)\right), \quad i=1,\ldots,p.
    x(i) \sim \chi_\nu^2\left(\lambda(i)\right), \quad i=1,\ldots,p,
\end{equation}
where $\chi_\nu^2\left(\lambda(i)\right)$ is a chi-square distributed random variable with $\nu$ degrees of freedom and non-centrality parameter $\lambda(i)$.
% We will also look for statistical procedures that attain the performance limits, as soon as the problems become theoretically feasible.

Although the number of tested genomic locations $p$ can sometimes exceed $10^5$ or even $10^6$, it is often believed that only a small set of genetic locations have tangible influences on the outcome of the disease or the trait of interest.
Under the stylized assumption of sparsity, $\lambda$ is assumed to have $s$ non-zero components, with $s$ being much smaller than the problem dimension $p$. 
The goal of researchers is two-fold: 1) to test if $\lambda(i)=0$ for all $i$, and 2) to estimate the set $S=\{i:\lambda(i)\neq 0\}$.
In other words, we look to first determine if there are \emph{any} genetic variations associated with the disease; and if there are associations, we want to locate them.
The latter, referred to as the \emph{support recovery problem}, is the focus of this work.

\subsection{Relationship with the additive error model}
\label{subsec:motivation-additive}

Multiple testing problems are a staple of modern data-driven studies, where a large number of hypotheses are formulated and screened for their plausibility simultaneously.
This multiplicity of tests brings along a multitude of challenges, and has been a subject of extensive study in recent years.
In particular, numerous papers on the linear regression setting have been written \citep{genovese2012comparison, ji2012ups, ke2014covariance}, and an extensive but incomplete list can be found in the review article by \citet{jin2016rare}.
On the other hand, the simpler additive error model 
\begin{equation} \label{eq:model-additive}
    x(i) = \mu(i) + \epsilon(i), \quad i=1,\ldots,p,
\end{equation}
still stimulates new results \cite{arias2017distribution, butucea2018variable, gao2020fundamental}.
% Closer to the spirit of the current paper 
% sparse Poisson model \cite{arias2015sparse}
% [references, lots of them]

The chi-square model \eqref{eq:model-chisq} plays an important role in analyzing variable screening problems under omnidirectional alternatives.
For example, under two-sided alternatives in the additive error model \eqref{eq:model-additive} and Gaussianity, unbiased test procedures call for rejecting the hypothesis $\mu(i)=0$ at locations where observations have large absolute values, or equivalently, large squared values.
Squaring on both sides of \eqref{eq:model-additive}, we arrive at Model \eqref{eq:model-chisq} with non-centrality parameters $\lambda(i) = \mu^2(i)$ and degree-of-freedom parameter $\nu =1$.
In this case, the support recovery problem is equivalent to locating the set of observations with mean shifts, $S=\{i:\mu(i)\neq 0\}$, where the mean shifts could take place in both directions.

Therefore, a theory for the chi-square model \eqref{eq:model-chisq} naturally lends itself to the study of two-sided alternatives in the Gaussian additive error model \eqref{eq:model-additive}.
In comparing results in the special case of $\nu=1$ with existing theory on one-sided alternatives \cite{arias2017distribution, gao2020fundamental}, we will be able to quantify if, and how much of a price has to be paid for the additional uncertainty when we have no prior knowledge on the direction of the signals.

% In many applications, of course, restricting ourselves to one-sided tests is unrealistic.
% For example, in fMRI studies, the interest is in \emph{both} regions where average brain activities {increase} and where they {decrease}, when comparing the case group to the controls \citep{narayan2015two}. 
% In the challenging application of anomaly detection on Internet traffic streams, millions of IP addresses need to be scanned in real time to identify \emph{both} volumetric attacks and blackouts \citep{kallitsis2016amon}.
% Indeed, omnidirectional tests are the more natural choice in so-called discovery sciences where little to no prior knowledge is available.

\subsection{Asymmetric statistical risks}
\label{subsec:asymmetric-risk}

% An issue of practical importance that was somewhat overlooked by the literature is the choice of statistical risks.
Another consideration that motivates the current study is the choice of practically relevant statistical risks.
To the best of our knowledge, so far, the literature on multiple testing have been concerned with criteria roughly \emph{symmetric} in terms of false discovery and false non-discovery control.
For example, \citet{arias2017distribution} studied conditions under which \emph{fractions} of false discovery and \emph{fractions} of non-discovery can vanish; \citet{gao2020fundamental} investigated \emph{family-wise error control} for both types of errors; meanwhile, \citet{butucea2018variable} analyzed the problem under the Hamming loss, which penalizes false discoveries and missed signals {equally}.

In applications, however, attitudes towards type I and type II errors are often different.
In the example of GWAS, where the number of candidate locations $p$ could be in the millions, researchers are typically interested in the marginal (location-wise) power of discovery, while exercising stringent (family-wise) false discovery control --- a situation not covered by existing theory.
The observation on this asymmetry leads us to study, and discover, two new phase transitions, both in the additive error model \eqref{eq:model-additive} under one-sided alternatives, and in the chi-square model \eqref{eq:model-chisq}.
The latter, as discussed in Section \ref{subsec:motivation-additive}, entails the additive error model \eqref{eq:model-additive} under two-sided alternatives.
% Further details can be found in Section \ref{subsec:risks} below, after 
% We summarize the main messages of this paper next.

\subsection{Contributions}

Our contribution is three-fold: (1) a thorough study of the theoretical limits of thresholding procedures in the chi-square model \eqref{eq:model-chisq} and an enumeration of practical procedures that attain these limits, 
(2) an illustration of the implications of the phase transition phenomena in genetic association studies, and
(3) two new optimality results in the Gaussian additive error model \eqref{eq:model-additive}, as well as direct comparisons of the one-sided and two-sided alternatives. 

In detail, we show that several commonly used family-wise error rate-control procedures --- including Bonferroni's procedure \cite{dunn1961multiple} --- are asymptotically optimal for the \emph{exact}, and \emph{exact-approximate} support recovery problems (defined in Section \ref{subsec:risks} below) in the chi-square model \eqref{eq:model-chisq}.
We further show that the Benjamini-Hochberg (BH) procedure \cite{benjamini1995controlling} is asymptotically optimal for the \emph{approximate}, and \emph{approximate-exact} support recovery problems (defined, too, in Section \ref{subsec:risks}).
These four results are made precise in our main theorems in Section \ref{sec:chisq-boundaries}.
% \ref{thm:chi-squared-exact-boundary}, \ref{thm:chi-squared-exact-approx-boundary}, \ref{thm:chi-squared-approx-boundary}, and \ref{thm:chi-squared-approx-exact-boundary}.
Under appropriate parametrizations of the signal sizes and sparsity, they establish the phase transitions of support recovery problems in the chi-square model, previewed in Figure \ref{fig:phase-chi-squared}.
Remarkably, the degree-of-freedom parameter does not affect the asymptotic boundaries in any of the four support recovery problems.

We then return to association screenings of categorical variables in Section \ref{sec:signal-size-odds-ratio}, and 
present the consequences of the phase transition in the exact-approximate problem in large-scale genetic association studies.
% present empirical evidence for the phase transition in the exact-approximate problem using real data from large-scale association studies on breast cancer obtained from the NHGRI-EBI GWAS Catalog \cite{macarthur2016new}.
% demystify the notion of signal size $\lambda$ in this context.
% which is perhaps less transparent than in additive error models.
We do so by characterizing the relationship between the signal size $\lambda$ and the marginal frequencies, odds ratio, and sample sizes for association tests on 2-by-2 contingency tables.
% Specifically, the amount of signal, when rare variants are present, is weaker compared to the signal when the marginal distributions are balanced.
% In other words, reliable detection of the effects by rare variants would require more samples compared to common variants, even at the same odds ratio.
This result, establishing the relationship between sample sizes and signal sizes, is made precise in Proposition \ref{prop:signal-size-odds-ratio}.
A corollary of the proposition, which may be of independent interest, enables us to determine the optimal design for association studies under a fixed budget, and reveals that balanced designs with equal number of cases and controls are often statistically inefficient.
We illustrate practical use of these results in power analysis with numerical examples. 

Finally, we study the Gaussian additive error model \eqref{eq:model-additive} under \emph{one-sided alternatives} in Section \ref{sec:additive-error-model-boundaries}. 
% While phase transitions for the approximate and exact support recovery problems were discovered in \cite{arias2017distribution} and \cite{gao2018fundamental}, 
Two new results on the phase transition of support recovery problems are established.
All phase transition boundaries (under suitable parametrizations) coincide with those in the chi-square models.
% and Figure \ref{fig:phase-chi-squared} continues to apply.
Following the discussion in Section \ref{subsec:motivation-additive}, this indicates vanishing differences between the difficulties of the one-sided and two-sided alternatives in the Gaussian additive error model \eqref{eq:model-additive}.

The rest of this paper is organized as follows. 
Section \ref{sec:setup} describes the formal set-up for the main results summarized above; these main results are detailed in Sections \ref{sec:chisq-boundaries}, \ref{sec:signal-size-odds-ratio}, and \ref{sec:additive-error-model-boundaries}. 
The phase transitions are demonstrated with numerical simulations in Section \ref{sec:numerical}.
%The many interesting results obtained under one-sided alternatives cannot be easily adapted to the two-sided cases where the locations and the directions of the signals are both unknown.
More related literature, comments, and discussions are collected in Section \ref{sec:discussions}.
Appendix \ref{sec:procedures} contains brief review of commonly used procedures in multiple testing.
Proofs are presented in Appendix \ref{sec:proofs}.

% Practical issues are also addressed to make for simple and effective power analysis.

\section{Statistical risks, procedures, and the asymptotic regime}
\label{sec:setup}

We begin by establishing the necessary notations and definitions.

\subsection{Statistical risks in support recovery problems}
\label{subsec:risks}

\begin{figure}
      \centering
      \includegraphics[width=0.6\textwidth]{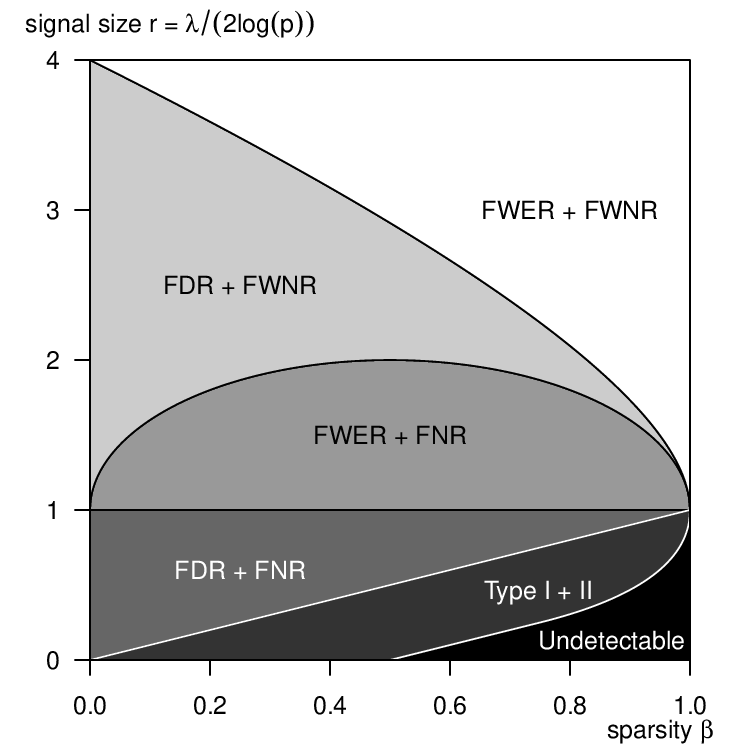}
      \caption{The phase diagram for the high-dimensional chi-square model \eqref{eq:model-chisq}, illustrating the boundaries of the exact support recovery (FWER + FWNR; top curve; Theorem \ref{thm:chi-squared-exact-boundary}),
      the approximate-exact support recovery (FDR + FWNR; second curve from top; Theorem \ref{thm:chi-squared-approx-exact-boundary}),
      the exact-approximate support recovery (FWER + FNR; horizontal line $r=1$; Theorem \ref{thm:chi-squared-exact-approx-boundary}),
      and the approximate support recovery problems (FDR + FNR; tilted line $r=\beta$; Theorem \ref{thm:chi-squared-approx-boundary}).
      The signal detection problem (type I + type II errors of the global test; lower curve) was studied in \citep{donoho2004higher}. 
      In each region of the diagram and above, the annotated statistical risk can be made to vanish, as dimension $p$ diverges. 
      Conversely, the risks has liminf at least one.
      All boundaries are unaffected by the degree-of-freedom.
      All boundaries are identical to those in the Gaussian additive error model \eqref{eq:model-additive} under one-side alternatives; see Section \ref{sec:additive-error-model-boundaries}.
    %   , setting $\lambda=\mu^2$.
      } 
      \label{fig:phase-chi-squared}
\end{figure}

Recall that in support recovery problems, our goal is to come up with a procedure, denoted $\mathcal R$, that produces a set estimate $\widehat{S}$ of the true index set of relevant variables  $S=\{i:\lambda(i)\neq 0\}$.
Formally, one should write $\widehat{S}(\mathcal{R}(x))$ to reflect the dependence of the set estimate on the procedure $\mathcal{R}$ and on the test statistics $x$; 
for notational convenience, we suppress this dependence and simply write $\widehat{S}$.
% in place of $\widehat{S}(\mathcal{R}(x))$ 

For a given procedure $\mathcal{R}$, the \emph{false discovery rate} (FDR) 
and the \emph{false non-discovery rate} (FNR) of a procedure is defined as
% of the procedure is defined to be the expected fraction of false findings not in the true index set, among the reported discoveries \cite{benjamini1995controlling}. 
% Its counterpart, \emph{false non-discovery rate} (FNR), measuring the power of the procedure, is defined as the expected fraction of missed detection. 
% Mathematically, we define
\begin{equation} \label{eq:FDR-FNR}
    \mathrm{FDR}(\mathcal{R}) = \E\left[\frac{|\widehat{S}\setminus S|}{\max\{|\widehat{S}|,1\}}\right],
    \quad \text{and} \quad
    \mathrm{FNR}(\mathcal{R}) = \E\left[\frac{|S\setminus \widehat{S}|}{\max\{|{S}|,1\}}\right],
\end{equation}
where the maxima in the denominators resolve the otherwise division by 0 problem. 
Roughly speaking, FDR measures the expected fraction of false findings, while FNR describes the proportion of type II errors among the true signals, and reflects the average marginal power of the procedure.

A more stringent criterion for false discovery is the family-wise error rate (FWER),
% defined to be the probability of reporting at least one finding not contained in the true index set.
and correspondingly, a more stringent criteria for false non-discovery is the family-wise non-discovery rate (FWNR), i.e.,
% the probability of missing at least one signal in the true index set. That is,
\begin{equation} \label{eq:FWER-FWNR}
    \mathrm{FWER}(\mathcal{R}) = 1 - \P[\widehat{S} \subseteq S], 
    \quad \text{and} \quad
    \mathrm{FWNR}(\mathcal{R}) = 1 - \P[S \subseteq \widehat{S}].
\end{equation}

We introduce five different ways to quantified statistical risks in support recovery problems, leading to different asymptotic limits. 
Following \cite{arias2017distribution}, we define the risk for \emph{approximate} support recovery as
\begin{equation} \label{eq:risk-approximate}
    \mathrm{risk}^{\mathrm{A}}(\mathcal{R}) = \mathrm{FDR}(\mathcal{R}) + \mathrm{FNR}(\mathcal{R}).
\end{equation}
Analogously, we define the risk for \emph{exact} support recovery as
\begin{equation} \label{eq:risk-exact}
    \mathrm{risk}^{\mathrm{E}}(\mathcal{R}) = \mathrm{FWER}(\mathcal{R}) + \mathrm{FWNR}(\mathcal{R}).
\end{equation}
An intimately related measure of success in the exact support recovery risk is the probability of exact recovery, 
\begin{equation} \label{eq:risk-prob}
    \P[\widehat{S} = S] = 1 - \P[\widehat{S} \neq S].
\end{equation}
The relationship between $\P[\widehat{S} = S]$ and $\mathrm{risk}^{\mathrm{E}}$ will be analyzed in Section \ref{subsec:asymptotics}.

The discussion in Section \ref{subsec:asymmetric-risk} --- and in particular, the GWAS application --- prompts us to consider risks that weigh both the family-wise error rate and the marginal power of discovery.
One such risk metric is what we refer to as the \emph{exact-approximate} support recovery risk
\begin{equation} \label{eq:risk-exact-approx}
    \mathrm{risk}^{\mathrm{EA}}(\mathcal{R}) = \mathrm{FWER}(\mathcal{R}) + \mathrm{FNR}(\mathcal{R}).
\end{equation}
The somewhat strange name is chosen to reflect ``exact false discovery control, and approximate false non-discovery control'', should the risk metric \eqref{eq:risk-exact-approx} vanish asymptotically.

Analogously, we consider the \emph{approximate-exact} support recovery risk
\begin{equation} \label{eq:risk-approx-exact}
    \mathrm{risk}^{\mathrm{AE}}(\mathcal{R}) = \mathrm{FDR}(\mathcal{R}) + \mathrm{FWNR}(\mathcal{R}),
\end{equation}
which places more emphasis on non-discovery control.

% These two risks differ in their stringency in controlling false discovery and false non-discovery.
Theoretical limits and performance of procedures in support recovery problems will be studied in terms of the five risk metrics \eqref{eq:risk-approximate}, \eqref{eq:risk-exact}, \eqref{eq:risk-prob}, \eqref{eq:risk-exact-approx} and \eqref{eq:risk-approx-exact}, defined above.

\subsection{Thresholding procedures}
\label{subsec:thresholding-procedures}

We shall study the performance of five procedures,
all of which belong to the broad class of thresholding procedures.
\begin{definition}[Thresholding procedures]
A thresholding procedure for estimating the support 
$S:=\{i\, :\, \lambda(i)\neq0\}$ is one that takes on the form
\begin{equation} \label{eq:thresholding-procedure}
    \widehat{S} = \left\{i\,|\,x(i) \ge t(x)\right\},
\end{equation}
where the threshold $t(x)$ may depend on the data $x$.
\end{definition}
Examples of thresholding procedures include ones that aim to control FWER --- Bonferroni's \cite{dunn1961multiple}, Sid\'ak's \citep{vsidak1967rectangular}, Holm's \citep{holm1979simple}, and Hochberg's procedure \citep{hochberg1988sharper} --- as well as procedures that target FDR, such as the Benjamini-Hochberg \cite{benjamini1995controlling} and the Cand\'es-Barber procedure \cite{barber2015controlling}.
Indeed, thresholding procedures \eqref{eq:thresholding-procedure} is such a general class that it contains most (but not all) of the statistical procedures in the multiple testing literature.
% \cite{roquain2011type}.

We shall restrict our attention to the class of thresholding procedures.
Specifically, the lower bounds that we develop in Theorems \ref{thm:chi-squared-exact-boundary} through \ref{thm:chi-squared-approx-exact-boundary}, and in Theorems \ref{thm:additive-error-exact-approx-boundary} and \ref{thm:additive-error-approx-exact-boundary} below, are only meant to apply to such procedures. 
The optimality of thresholding procedures and the consequences of this restriction will be briefly discussed in Section \ref{sec:discussions}.

\subsection{Asymptotic success and failure of support recovery}
\label{subsec:asymptotics}

We shall work under the asymptotic regime where the problem dimension $p$ diverges.
Specifically, we will work with the triangular array of chi-square models \eqref{eq:model-chisq} indexed by $p$.
Let the non-centrality parameter vectors $\lambda = \lambda_p$ have 
\begin{equation} \label{eq:signal-sparsity}
    |S_p| = \left\lfloor p^{1-\beta} \right\rfloor, \quad \beta\in(0,1)
\end{equation}
non-zero entries, where $\beta$ parametrizes the problem sparsity.
The closer $\beta$ is to 1, the sparser the support $S$; conversely, when $\beta$ is close to 0, the support is dense with many non-null signals.

We further parametrize the range of the non-zero (and perhaps unequal) signals with
\begin{equation} \label{eq:signal-size}
    \underline{\Delta} = 2\underline{r}\log{p}
    \le \lambda(i) \le
    \overline{\Delta} = 2\overline{r}\log{p}, \quad \text{for all}\;\;i\in S_p,
\end{equation}
for some constants $0<\underline{r}\le\overline{r}\le+\infty$.

\begin{remark}
\label{rmk:global-test-boundary}
The parametrization of signal sparsity \eqref{eq:signal-sparsity} and signal sizes  \eqref{eq:signal-size} in the chi-square model seem to be first introduced by \citet{donoho2004higher}, where the signal sizes were assumed equal with magnitude $(2{r}\log{p})$.
It was shown in \cite{donoho2004higher} that a phase transition in the $r$-$\beta$ plane exists for the signal detection problem. 
That is, if $r$ is above a so-called detection boundary, then the global null hypothesis $\lambda(i)=0$ for all $i=1,\ldots,p$ can be told apart from the alternative as $p\to\infty$ with vanishing type I and type II errors; 
otherwise, below the boundary, no test can do better than a random guess.
We shall see that the scaling of sparsity \eqref{eq:signal-sparsity} and signal size \eqref{eq:signal-size} is also suitable for studying the phase transitions of the support recovery problem.
\end{remark}

The criteria for success and failure in support recovery problems under this asymptotic regime are defined as follows.
\begin{definition} \label{def:exact-recovery-success-failure}
We say a sequence of procedures $\mathcal{R} = \mathcal{R}_p$ succeeds asymptotically in the exact (and respectively, exact-approximate, approximate-exact, and approximate) support recovery problem if 
\begin{equation} \label{eq:support-recovery-success}
    \mathrm{risk}^{\mathrm{P}}(\mathcal{R}) \to 0, \quad \text{as}\quad p\to\infty,
\end{equation}
where $\mathrm{P}=\mathrm{E}$ (respectively, $\mathrm{EA}$, $\mathrm{AE}$, $\mathrm{A}$).

Conversely, we say the exact support recovery fails asymptotically in the exact (and respectively, exact-approximate, approximate-exact, and approximate) support recovery problem if 
\begin{equation} \label{eq:support-recovery-faliure}
    \liminf\mathrm{risk}^{\mathrm{P}}(\mathcal{R}) \ge 1, \quad \text{as}\quad p\to\infty,
\end{equation}
where $\mathrm{P}=\mathrm{E}$ (respectively, $\mathrm{EA}$, $\mathrm{AE}$, $\mathrm{A}$).
\end{definition}
% Similarly, we define the criteria for asymptotic success and failure for approximate support recovery as follows.
% \begin{definition} \label{def:approx-recovery-success-failure}
% We say a sequence of procedures $\mathcal{R} = \mathcal{R}_p$ succeeds asymptotically in the approximate support recovery problem if 
% \begin{equation} \label{eq:approx-recovery-success}
%     \mathrm{risk}^{\mathrm{A}}(\mathcal{R}) \to 0, \quad \text{as}\quad p\to\infty.
% \end{equation}
% We say the approximate support recovery fails asymptotically if 
% \begin{equation} \label{eq:approx-recovery-failure}
%     \liminf\mathrm{risk}^{\mathrm{A}}(\mathcal{R}) \ge 1, \quad \text{as}\quad p\to\infty.
% \end{equation}
% \end{definition}

% The performance of procedures in terms of the criteria in Definition \ref{def:exact-recovery-success-failure} 
% % and \ref{def:approx-recovery-success-failure} 
% will be analyzed in Sections \ref{subsec:exact-support-recovery-boundary} and \ref{subsec:approx-support-recovery-boundary}.

We now elaborate on the relationship between the probability of exact recovery and risk of exact support recovery, as promised in Section \ref{subsec:risks}.
\begin{lemma} \label{lemma:risk-exact-recovery-probability}
Let $\mathcal{R} = \mathcal{R}_p$ be the sequence of procedures for support recovery under the chi-square model \eqref{eq:model-chisq}. 
%The probability of exact recovery $\P[\widehat{S} = S]$, and risk of exact support recovery $\mathrm{risk}^{\mathrm{E}}$, defined in \eqref{eq:risk-exact}, are related as follows,
In this case, as $p\to\infty$, we have
\begin{equation} \label{eq:exact-recovery-implies-risk-0}
    \P[\widehat{S} = S] \to 1 \iff \mathrm{risk}^{\mathrm{E}}\to0,
\end{equation}
and
\begin{equation} \label{eq:failure-recovery-implies-risk-1}
    \P[\widehat{S} = S] \to 0 \implies \liminf\mathrm{risk}^{\mathrm{E}}\ge1,
\end{equation}
where the dependence on $p$ was suppressed for notational convenience.
\end{lemma}

By virtue of Lemma \ref{lemma:risk-exact-recovery-probability}, it is sufficient to study the probability of exact support recovery $\P[\widehat{S}=S]$ in place of $\mathrm{risk}^{\mathrm{E}}$, if we are interested in the asymptotic properties of the risk in the sense of \eqref{eq:support-recovery-success} and \eqref{eq:support-recovery-faliure}.
% (The converse, discussed in Section \ref{sec:discussions} below, is not true.)

\section{Phase transitions in the chi-square models}
\label{sec:chisq-boundaries}

% We review four commonly used procedures in multiple testing. 
% Their optimality properties in the chi-square model \eqref{eq:model-chisq} are established in this Section.

A final ingredient we need, before stating our first main result, is a rate at which the nominal levels of FWER or FDR go to zero.
\begin{definition} \label{def:slowly-vanishing}
We say the nominal level of errors $\alpha = \alpha_p$ vanishes slowly, if
\begin{equation} \label{eq:slowly-vanishing-error}
    \alpha\to 0,\quad \text{and} \quad \alpha p^\delta\to\infty \text{  for any } \delta>0.
\end{equation}
\end{definition}
As an example, the sequence of nominal levels $\alpha_p = 1/\log{(p)}$ is slowly vanishing, while the sequence $\alpha_p = 1/\sqrt{p}$ is not.

\subsection{The exact support recovery problem}
\label{subsec:exact-support-recovery-boundary}

The first main result characterizes the phase-transition phenomenon in the exact support recovery problem under the chi-square model.

\begin{theorem} \label{thm:chi-squared-exact-boundary}
Consider the high-dimensional chi-squared model \eqref{eq:model-chisq} with signal sparsity and size as described in \eqref{eq:signal-sparsity} and \eqref{eq:signal-size}.
The function 
\begin{equation} \label{eq:exact-boundary-chisquared}
    g(\beta) = \left(1 + \sqrt{1-\beta}\right)^2
\end{equation}
characterizes the phase transition of exact support recovery problem.
Specifically, if $\underline{r} > {{g}}(\beta)$, then Bonferroni's, Sid\'ak's, Holm's, and Hochberg's procedures with slowly vanishing (see Definition \ref{def:slowly-vanishing}) nominal FWER levels all achieve asymptotically exact support recovery in the sense of \eqref{eq:support-recovery-success}. 

Conversely, if $\overline{r} < {{g}}(\beta)$, then for any thresholding procedure $\widehat{S}$, we have $\P[\widehat{S}=S]\to0$.
Therefore, in view of Lemma \ref{lemma:risk-exact-recovery-probability}, exact support recovery asymptotically fails for all thresholding procedures in the sense of \eqref{eq:support-recovery-faliure}.
\end{theorem}

The procedures mentioned in Theorem \ref{thm:chi-squared-exact-boundary} are reviewed in Appendix \ref{sec:procedures}. 
Proof of the theorem is found in Appendix \ref{subsec:proof-chi-squared-exact-boundary}. 
% The boundary \eqref{eq:exact-boundary-chisquared} is plotted in Figure \ref{fig:phase-chi-squared}.
Comparisons with parallel results in the Gaussian additive error model \eqref{eq:model-additive} will be drawn in Section \ref{sec:additive-error-model-boundaries}.

\begin{remark} \label{rmk:strong-classification-boundary-2}
Theorem \ref{thm:chi-squared-exact-boundary} predicts that the asymptotic boundaries are the same for all values of the parameter $\nu$.
In simulations (see Section \ref{sec:numerical}), we find this asymptotic prediction to be quite accurate for $\nu\le3$ even in moderate dimensions ($p=100$). 
For $\nu>3$, the phase transitions take place somewhat above the boundary ${g}$.
The behavior is qualitatively similar for the other three phase transitions (see Theorems \ref{thm:chi-squared-exact-approx-boundary}, \ref{thm:chi-squared-approx-boundary}, and \ref{thm:chi-squared-approx-exact-boundary} below).
\end{remark}

\subsection{The exact-approximate support recovery problem}
\label{subsec:exact-approx-support-recovery-boundary}

The next theorem describes the phase transition in the exact-approximate support recovery problem.

\begin{theorem} \label{thm:chi-squared-exact-approx-boundary}
In the context of Theorem \ref{thm:chi-squared-exact-boundary}, 
the function 
\begin{equation} \label{eq:exact-approx-boundary-chisquared}
    \widetilde{g}(\beta) = 1
\end{equation}
characterizes the phase transition of exact-approximate support recovery problem.
Specifically, if $\underline{r} > \widetilde{g}(\beta)$, then the procedures listed in Theorem \ref{thm:chi-squared-exact-boundary} with slowly vanishing nominal FWER levels achieve asymptotically exact-approximate support recovery in the sense of \eqref{eq:support-recovery-success}. 

Conversely, if $\overline{r} < \widetilde{g}(\beta)$, then for any thresholding procedure $\widehat{S}$, the exact-approximate support recovery fails in the sense of \eqref{eq:support-recovery-faliure}.
\end{theorem}

Theorem \ref{thm:chi-squared-exact-approx-boundary} is proved in Appendix \ref{subsec:proof-chi-squared-mix-boundaries}. 

\begin{remark}
The boundary \eqref{eq:exact-approx-boundary-chisquared} was briefly suggested by \citet{arias2017distribution}, who focused exclusively on Gaussian additive error models \eqref{eq:model-additive}.
Unfortunately, it was falsely claimed that the boundary characterized the phase transition of the \emph{exact} support recovery problem, and the alleged proof was left as an ``exercise to the reader''.
This exercise was completed in \cite{gao2020fundamental}, where the correct boundary \eqref{eq:exact-boundary-chisquared} was identified. 

Theorem \ref{thm:chi-squared-exact-approx-boundary} here shows that the boundary \eqref{eq:exact-approx-boundary-chisquared} \emph{does} exist, though for the slightly different \emph{exact-approximate} support recovery problem.
As we will see in Section \ref{sec:additive-error-model-boundaries}, the boundary \eqref{eq:exact-approx-boundary-chisquared} also applies to the exact-approximate support recovery problem in the Gaussian additive error model \eqref{eq:model-additive}.
\end{remark}

\subsection{The approximate support recovery problem}
\label{subsec:approx-support-recovery-boundary}

Our third main result characterizes the phase-transition phenomenon in the approximate support recovery problem in the chi-square model.

\begin{theorem} \label{thm:chi-squared-approx-boundary}
Consider the high-dimensional chi-squared model \eqref{eq:model-chisq} with signal sparsity and size as described in \eqref{eq:signal-sparsity} and \eqref{eq:signal-size}.
The function 
\begin{equation} \label{eq:approx-boundary-chisquared}
    h(\beta) = \beta
\end{equation}
characterizes the phase transition of approximate support recovery problem.
Specifically, if $\underline{r} > {h}(\beta)$, then the Benjamini-Hochberg procedure $\widehat{S}_p$ (defined in Appendix \ref{sec:procedures}) with slowly vanishing (see Definition \ref{def:slowly-vanishing}) nominal FDR levels achieves asymptotically approximate support recovery in the sense of \eqref{eq:support-recovery-success}. 

Conversely, if $\overline{r} < {h}(\beta)$, then approximate support recovery asymptotically fails in the sense of \eqref{eq:support-recovery-faliure} for all thresholding procedures.
\end{theorem}

Theorem \ref{thm:chi-squared-approx-boundary} is proved in Appendix \ref{subsec:proof-chi-squared-mix-boundaries}.

\subsection{The approximate-exact support recovery problem}
\label{subsec:aprox-exact-support-recovery-boundary}

The last phase transition is in terms of the approximate-exact support recovery risk
\eqref{eq:risk-approx-exact}.

\begin{theorem} \label{thm:chi-squared-approx-exact-boundary}
In the context of Theorem \ref{thm:chi-squared-approx-boundary}, the function 
\begin{equation} \label{eq:approx-exact-boundary-chisquared}
    \widetilde{h}(\beta) = \left(\sqrt{\beta} + \sqrt{1-\beta}\right)^2
\end{equation}
characterizes the phase transition of approximate-exact support recovery problem.
Specifically, if $\underline{r} > \widetilde{h}(\beta)$, then the Benjamini-Hochberg procedure with slowly vanishing nominal FDR levels achieves asymptotically approximate-exact support recovery in the sense of \eqref{eq:support-recovery-success}. 

Conversely, if $\overline{r} < \widetilde{h}(\beta)$, then for any thresholding procedure $\widehat{S}$, the approximate-exact support recovery fails in the sense of \eqref{eq:support-recovery-faliure}.
\end{theorem}

Theorem \ref{thm:chi-squared-approx-exact-boundary} is proved in Appendix \ref{subsec:proof-chi-squared-exact-boundary}. 

\begin{remark}
% This is reflected by the fact that $g(\beta) > \widetilde{g}(\beta) > h(\beta)$, and $g(\beta) > \widetilde{h}(\beta) > h(\beta)$ for all $\beta\in(0,1)$.
% The difficulty of the exact-approximate and approximate-exact problems are a little more difficult to guess
Theorems \ref{thm:chi-squared-exact-boundary} through \ref{thm:chi-squared-approx-exact-boundary} allow us to compare, quantitatively, the required signals sizes in support recovery problems, as well as in the global hypothesis testing problem in the chi-square model \eqref{eq:model-chisq}.
As mentioned in Remark \ref{rmk:global-test-boundary}, there exists a phase transition in the global hypothesis testing problem characterized by the boundary
\begin{equation} \label{eq:detection-boundary-chisquare}
    f(\beta) = 
    \begin{cases}
    \left(1-\sqrt{1-\beta}\right)^2, &\beta>3/4 \\
    \max\{0, \beta-1/2\}, &\beta\le3/4,
    \end{cases}
\end{equation}
which was identified in \citet{donoho2004higher}.
Results in this section indicate that at all sparsity levels $\beta\in(0,1)$, the difficulties of the problems in terms of the required signal sizes have the following ordering
$$
f(\beta) < h(\beta) < \widetilde{g}(\beta) < \widetilde{h}(\beta) < g(\beta),
$$
as previewed in Figure \ref{fig:phase-chi-squared}.
The ordering aligns with our intuition that the required signal sizes increase as we move from detection to support recovery problems.
Similarly, more stringent criteria for error control (e.g., FWER compared to FDR) require larger signals.
We can now also compare $\widetilde{g}(\beta)$ and $\widetilde{h}(\beta)$, whose ordering may not be clear from this line of reasoning.
\end{remark}

\section{The phase transition phenomena in GWAS}
\label{sec:signal-size-odds-ratio}

We return to the application of association screenings for categorical variables, and put the results in the previous section to use.
In particular, we focus on the exact-approximate support recovery problem, and demonstrate the consequences of its phase transition (Theorem \ref{thm:chi-squared-exact-approx-boundary}) in genetic association studies.

In order to do so, we must first connect the concept of ``statistical signal size'' $\lambda$ with some key quantities in association tests.
While ``signal size'' likely sounds foreign to most practitioners, it is intimately linked with the concept of ``effect sizes'' --- or odds ratios --- in association studies, which are frequently estimated and reported in GWAS catalogs.
We characterize the relationship between the two quantities in the special, but fairly common case of association tests on 2-by-2 contingency tables in Section \ref{subsec:odds-and-power}.

\subsection{Odds ratios and statistical power}
\label{subsec:odds-and-power}

% Unlike in additive models where the parameter $\mu$ has the interpretation of signal-to-noise ratios, the meaning of the signal sizes $\lambda$ in chi-square model is perhaps not as transparent.

Consider a 2-by-2 multinomial distribution with marginal probabilities of phenotypes $(\phi_1, \phi_2)$ and genotypes $(\theta_1, \theta_2)$.
The \emph{probability} table (as opposed to the table of multinomial \emph{counts} in the introduction) is as follows.
\begin{center}
    \begin{tabular}{cccc}
    \hline
    & \multicolumn{2}{c}{Genotype} \\
    \cline{2-3}
    Probabilities & Variant 1 & Variant 2 & Total by phenotype \\
    \hline
    Cases & $\mu_{11}$ & $\mu_{12}$ & $\phi_1$ \\
    Controls & $\mu_{21}$ & $\mu_{22}$ & $\phi_2$ \\
    Total by genotype & $\theta_1$ & $\theta_2$ & 1 \\
    \hline
    \end{tabular}
\end{center}
The odds ratio (i.e., ``effect size'') is defined as the ratio of the phenotype frequencies between the two genotype variants,
\begin{equation} \label{eq:odds-ratio}
    \text{R} := \frac{\mu_{11}}{\mu_{21}}\Big/\frac{\mu_{12}}{\mu_{22}}
    = \frac{\mu_{11}\mu_{22}}{\mu_{12}\mu_{21}}.
\end{equation}
The multinomial distribution is fully parametrized by the trio $(\theta_1, \phi_1, R)$.
Odds ratios further away from 1 indicate greater contrasts between the probability of outcomes.
Independence between the genotypes and phenotyes would imply an odds ratio of one, and hence $\mu_{jk} = \phi_j\theta_k$, for all $j,k \in\{1,2\}$.

% When data are sampled from the multinomial distribution, the chi-square test defined in \eqref{eq:chisq-statistic} is asymptotically equivalent to tests including, e.g., the likelihood ratio test and Welch's t-test, both in terms of level and power \cite{ferguson2017course,gao2019upass}.
For a sequence of local alternatives $\mu^{(1)}, \mu^{(2)}, \ldots$, such that $\sqrt{n}(\mu^{(n)}_{jk} - \phi_j\theta_k)$ converges to a constant table $\delta = (\delta_{jk})$, the chi-square test statistics converge in distribution to the non-central chi-squared distribution with non-centrality parameter 
$\lambda = \sum_{j=1}^2 \sum_{k=1}^2 {\delta_{jk}^2}/{(\phi_j\theta_k)}$; see, e.g.,\cite{ferguson2017course}.
Hence, for large samples from a fixed distribution $(\mu_{ij})$, the statistic is well approximated by a $\chi^2_1(\lambda)$ distribution, where
\begin{equation} 
\lambda = n\sum_{j=1}^2 \sum_{k=1}^2 \frac{(\mu_{jk} - \phi_j\theta_k)^2}{\phi_j\theta_k}.
\end{equation}
%Since $\lambda$ is linear in the number of samples $n$, 
% Power of association tests at $\alpha$ level is approximately $\P[\chi^2_{\nu}(\lambda)>\chi^2_{\nu,\alpha}]$, where $\chi^2_{\nu,\alpha}$ is the upper $\alpha$-quantile of a central Chi-squared distribution.
Power calculations therefore only depend on the $\mu_{jk}$'s through $\lambda=nw^2$, where we define 
\begin{equation} \label{eq:signal-size-chisq}
    w^2:=\lambda/n
\end{equation} 
to be the \emph{signal size per sample}. 
Statistical power would be increasing in $w^2$ for fixed sample sizes.

The next proposition states that the statistical signal size per sample can be parametrized by the odds ratio and the marginals in the probability table.

\begin{proposition} \label{prop:signal-size-odds-ratio}
Consider a 2-by-2 multinomial distribution with marginal distributions $(\phi_1, \phi_2 = 1-\phi_2)$ and $(\theta_1, \theta_2=1-\theta_1)$.
Let signal size $w^2$ be defined as in \eqref{eq:signal-size-chisq}, and odds ratio $\text{R}$ be defined as in \eqref{eq:odds-ratio}. 
If $R=1$, we have $w^2 = 0$; if $R\in(0,1)\cup(1,+\infty)$, then we have
\begin{equation} \label{eq:signal-size-odds-ratio}
    w^2(\text{R}) =
    \frac{1}{4A(\text{R}-1)^2}\left(B+CR-\sqrt{(B+CR)^2-4A(R-1)^2}\right)^2,
\end{equation}
where $A = \phi_1\theta_1\phi_2\theta_2$, $B = \phi_1\theta_1+\phi_2\theta_2$, and $C = \phi_1\theta_2+\phi_2\theta_1$.
\end{proposition}

Proposition \ref{prop:signal-size-odds-ratio} is derived in Appendix \ref{subsec:proof-signal-size-odds-ratio}. 

To understand Proposition \ref{prop:signal-size-odds-ratio}, we illustrate Relation \eqref{eq:signal-size-odds-ratio} for selected values of marginals $\theta_1$ and $\phi_1$ in Figure \ref{fig:signal-vs-odds}.
Observe in the figure that an odds ratio further away from one corresponds to stronger statistical signal per sample, ceteris paribus.
However, this ``valley'' pattern is in general not symmetric around 1, except for balanced marginal distributions ($\phi_1=1/2$ or $\theta_1=1/2$).
While the odds ratio $R$ can be arbitrarily close to 0 or diverge to $+\infty$ for any marginal distribution, the signal sizes $w^2$ are bounded from above by constants that depend only on the marginals.
% This is quantified in the next corollary.

\begin{figure}
      \centering
      \includegraphics[width=0.49\textwidth]{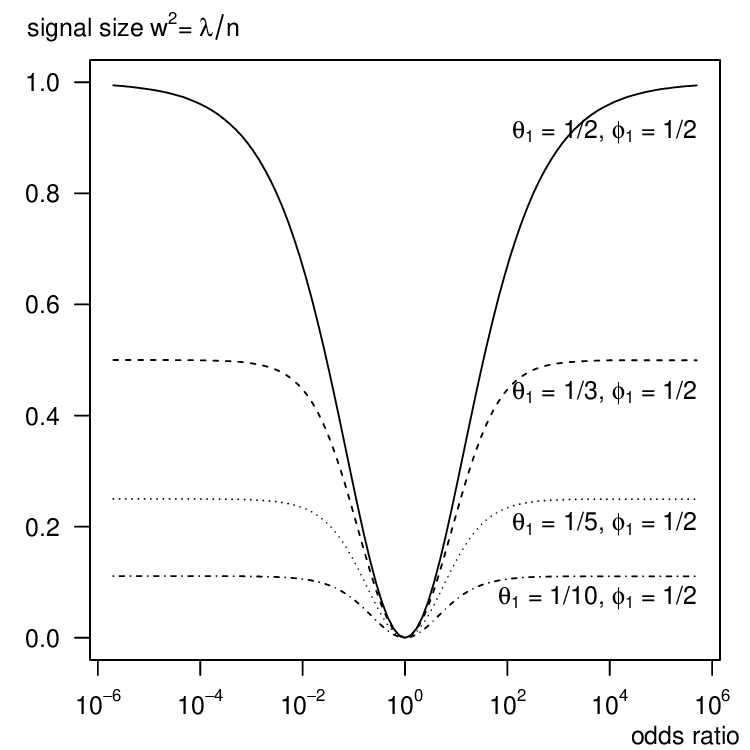}
      \includegraphics[width=0.49\textwidth]{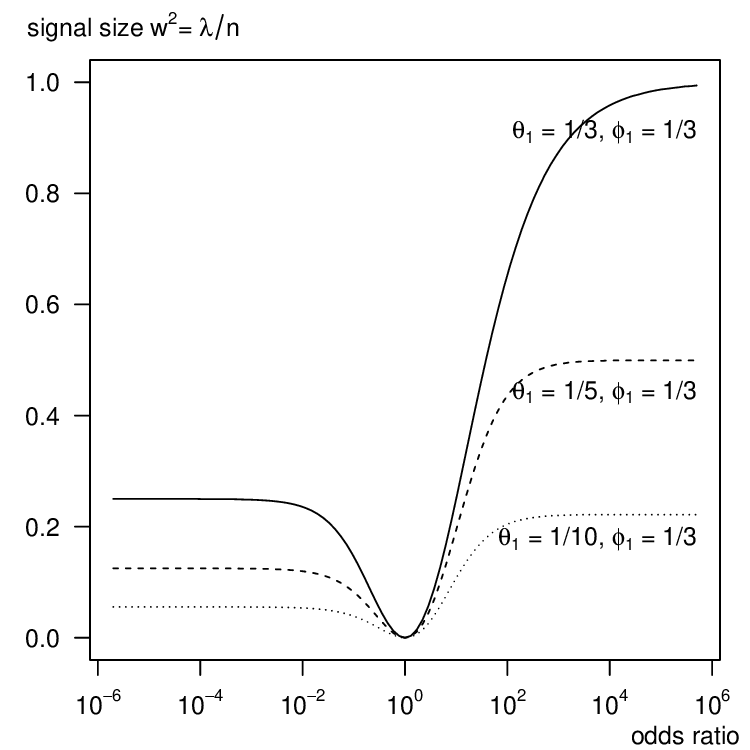}            
      \caption{Signal sizes per sample $w^2$ as functions of odds ratios in 2-by-2 multinomial distributions for selected genotype marginals in balanced (left) and unbalanced (right) designs; see Relation \eqref{eq:signal-size-odds-ratio} in Proposition \ref{prop:signal-size-odds-ratio}.
      For given marginal distributions, extreme odds ratios imply stronger statistical signals at a given sample size.
      However, the signal sizes are bounded above by constants that depend on the marginal distributions; see Relations \eqref{eq:signal-size-upper-bound-1} and \eqref{eq:signal-size-upper-bound-2}.
      % Unbalanced marginal distributions -- or rare variants -- lead to smaller signal sizes at a given odds ratio.
      } 
      \label{fig:signal-vs-odds}
\end{figure}

\begin{corollary} \label{cor:signal-limits-OR}
The signal size as a function of the odds ratio $w^2(R)$ is decreasing on $(0,1)$ and increasing on $(1,\infty)$, with limits
\begin{equation} \label{eq:signal-size-upper-bound-1}
    \lim_{\text{R}\to0_+} w^2(\text{R}) = \min\left\{\frac{\phi_1\theta_1}{\phi_2\theta_2}, \frac{\phi_2\theta_2}{\phi_1\theta_1}\right\},
\end{equation}
and
\begin{equation} \label{eq:signal-size-upper-bound-2}
    \lim_{\text{R}\to+\infty} w^2(\text{R}) = \min\left\{\frac{\phi_1\theta_2}{\phi_2\theta_1}, \frac{\phi_2\theta_1}{\phi_1\theta_2}\right\}.
\end{equation}
\end{corollary}
% Proof of Corollary \ref{cor:signal-limits-OR} is found in Appendix \ref{subsec:proof-signal-size-odds-ratio}. 

Corollary \ref{cor:signal-limits-OR} immediately implies that balanced designs with roughly equal number of cases and controls are not necessarily the most informative.

For example, in a study where a third of the recruited subjects carry the genetic variant positively correlated with the trait (i.e., $\theta_1=1/3$), an unbalanced design with $\phi_1=1/3$ would maximize $w^2$ at large odds ratios.
This unbalanced design is much more efficient compared to, say, a balanced design with $\phi_1=1/2$.
In the first case, we have $w^2\to1$ as $R\to\infty$; whereas in the second design, $w^2<1/2$ no matter how large $R$ is.
This difference can also be read by comparing the dashed curve ($\theta_1=1/3$, $\phi_1=1/2$) in the left panel of Figure \ref{fig:signal-vs-odds}, with the solid curve ($\theta_1=1/3$, $\phi_1=1/3$) in the right panel of Figure \ref{fig:signal-vs-odds}.

\subsection{Optimal study designs and rare variants}
\label{subsec:optimal-design} 

For a study with a fixed budget, i.e., a fixed total number of subjects $n$, the researcher is free to choose the fraction of cases $\phi_1$ to be included in the study.
A natural question is how this budget should be allocated to maximize the statistical power of discovery, or equivalently, the signal sizes $\lambda=nw^2$.

In principal, Relation \eqref{eq:signal-size-odds-ratio} can be optimized with respect to the fraction of cases $\phi_1$ in order to find optimal designs, if $\theta_1$ is known and held constant.
In practice, this is not the case.
While the fraction of cases can be controlled, the distributions of genotypes \emph{in the study} are often unknown prior to data collection, and can change with the case-to-control ratio.

Fortunately, the conditional distributions of genotypes in the healthy control groups are often estimated by existing studies, and are made available by consortiums such as the NHGRI-EBI GWAS catalog \cite{macarthur2016new}.
% Assume (after appropriate relabelling, hence without loss of generality) that the first variant is associated with an increased risk of disease, and is henceforth referred to as the risk variant.
We denote the conditional frequency of the first genetic variant in the control group as $(f, 1-f)$, where
$$
f := \mu_{21} / \phi_2.
$$
The multinomial probability is fully parametrized by the new trio: $(f, \phi_1, R)$.
\begin{center}
    \begin{tabular}{cccc}
    \hline
    & \multicolumn{2}{c}{Genotype} \\
    \cline{2-3}
    Probabilities & Variant 1 & Variant 2 & Total by phenotype \\
    \hline
    Cases & $\frac{\phi_1fR}{fR+1-f}$ & $\frac{\phi_1(1-f)}{fR+1-f}$ & $\phi_1$ \\
    Controls & $f(1-\phi_1)$ & $(1-f)(1-\phi_1)$ & $1-\phi_1$ \\
    \hline
    \end{tabular}
\end{center}
Proposition \ref{prop:signal-size-odds-ratio} may also be re-stated in terms of the new parametrization.

% Note that all these quantities refer to what is in the study, and differ from their counterparts in the general population.

\begin{corollary} \label{cor:signal-size-odds-ratio-conditional-frequency}
In the 2-by-2 multinomial distribution with marginals $(\phi_1, \phi_2 = 1-\phi_1)$, and conditional distribution of the variants in the control group $(f, 1-f)$,
Relation \eqref{eq:signal-size-odds-ratio} holds with $\theta_1 = {\phi_1fR}/{(fR+1-f)} + f(1-\phi_1)$ and $\theta_2 = 1-\theta_1$.
\end{corollary} 

The choice of $\phi_1$ now has a practical solution.

\begin{corollary} \label{cor:optimal-design}
In the context of Corollary \ref{cor:signal-size-odds-ratio-conditional-frequency},
the optimal design $(\phi^*_1, \phi^*_2)$ that maximizes the signal size per sample $w^2$ is prescribed by
\begin{equation} \label{eq:optimal-design}
    \phi_1^* = \frac{fR+1-f}{fR+1-f+\sqrt{R}}, \quad\text{and}\quad 
    \phi_2^* = 1-\phi_1^*.
\end{equation}
% when the denominator in \eqref{eq:optimal-design} is non-zero; otherwise, $\phi_1^*=\phi_2^*=1/2$.
\end{corollary} 

Proof of Corollary \ref{cor:optimal-design} is found in Appendix \ref{subsec:proof-signal-size-odds-ratio}. 

Of particular interest in the genetics literature are genetic variants with very low allele frequencies in the control group (i.e., $f\approx 0$), known as rare variants.
In such cases, Equation \eqref{eq:optimal-design} can be approximated using the Taylor expansion,
\begin{equation} \label{eq:optimal-design-approx}
    \phi_1^* = \frac{1}{1 + \sqrt{R}} + \frac{(R-\sqrt{R})f}{1+\sqrt{R}} + O(f^2).
\end{equation}
To illustrate, for rare and adversarial factors ($f\approx0$ and $R>1$), the optimal $\phi_1^*$ is less than $1/2$.
Therefore, for studies under a fixed budget, controls should constitute the majority of the subjects, in order to maximize power.
On the other hand, for rare and protective factors ($f\approx0$ and $R<1$), the optimal $\phi_1^*$ is greater than $1/2$, and cases should be the majority.

\subsection{Phase transitions in large-scale association screening studies}

% Specifically, we develop recipes to find suitable designs of association studies such that combination of the dimensionality $p$, sparsity $\beta$, and signal sizes $r$ of the problem lands in the desired region of risk control, as predicted by the results in Section \ref{sec:chisq-boundaries}.

% Of course, in applications, not all three of the parameters $(p, \beta, r)$ can be altered as we wish.
% In particular, the problem dimensions and sparsity levels are usually determined by the underlying physical processes.
% In the GWAS example, the number of genomic marker locations is determined by the chip used for gene sequencing, while the number of relevant genomic locations is a consequence of the biological process.
% Therefore, in order to achieve a desired level of error control, we can often only hope to influence the statistical signal sizes.

Returning to the problem of \emph{high-dimensional} marginal screenings for categorical covariates, we explore the manifestation of the phase transition in the exact-approximate support recovery problem in the genetic context.

Recall Theorem \ref{thm:chi-squared-exact-approx-boundary} predicts that FWER and FNR can be simultaneously controlled in large dimensions if and only if 
\begin{equation}
    r = \frac{\lambda}{2\log{p}} = \frac{w^2n}{2\log{p}} > 1.
\end{equation}
Therefore, if we were to apply FWER-controlling procedures at low nominal levels (say, $5\%$), then the FNR would experience a phase transition in the sense that, if
\begin{equation} \label{eq:power-1-region}
    r>1 \iff w^2 > \frac{2\log{p}}{n},
\end{equation}
then the FNR can be close to 0; otherwise, FNR must be close to 1.

% Translating this result into the language of association tests, 
Using the parametric relationship described in Corollary \ref{cor:signal-size-odds-ratio-conditional-frequency} (and Proposition \ref{prop:signal-size-odds-ratio}), 
the inequalities in \eqref{eq:power-1-region} implicitly define regions of $(f, R)$ where associations are discoverable with high power, for a given $\phi_1$.
Further, the boundary of such discoverable regions sharpens as dimensionality diverges. 
We illustrate this phase transition through a numerical example next.

\begin{figure}
    \centering
    \includegraphics[width=0.49\textwidth]{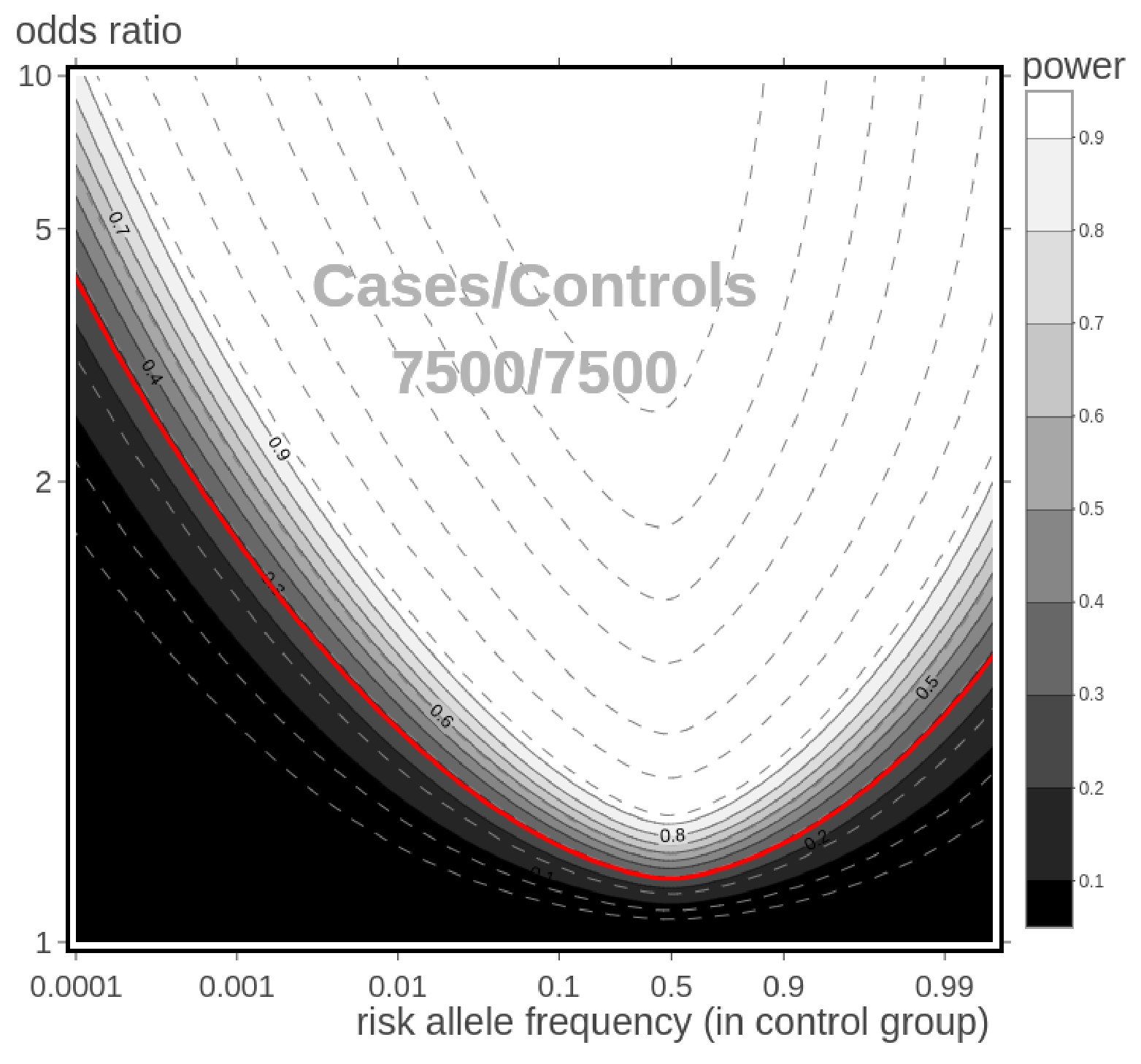}
    \includegraphics[width=0.49\textwidth]{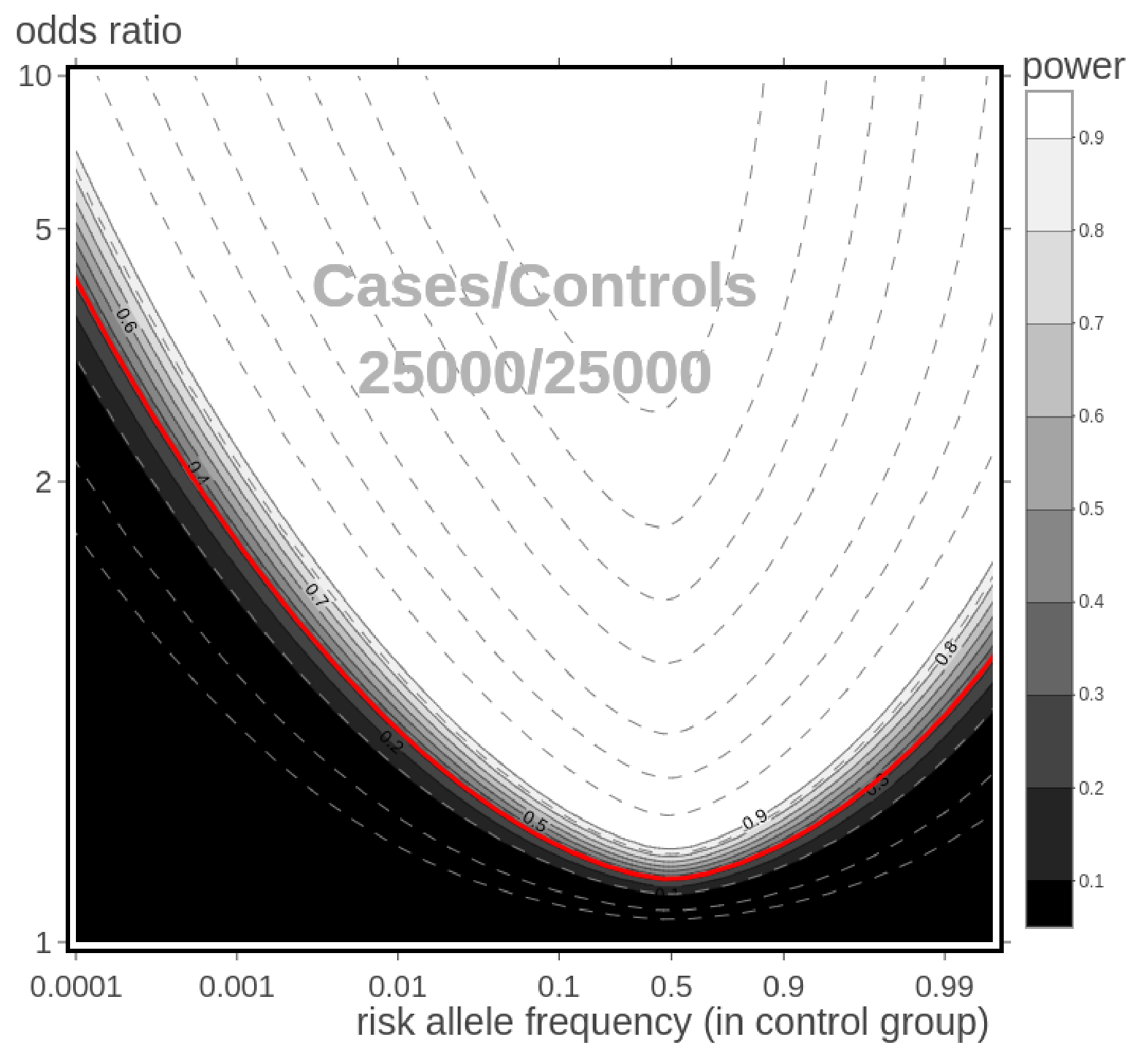}
    \includegraphics[width=0.49\textwidth]{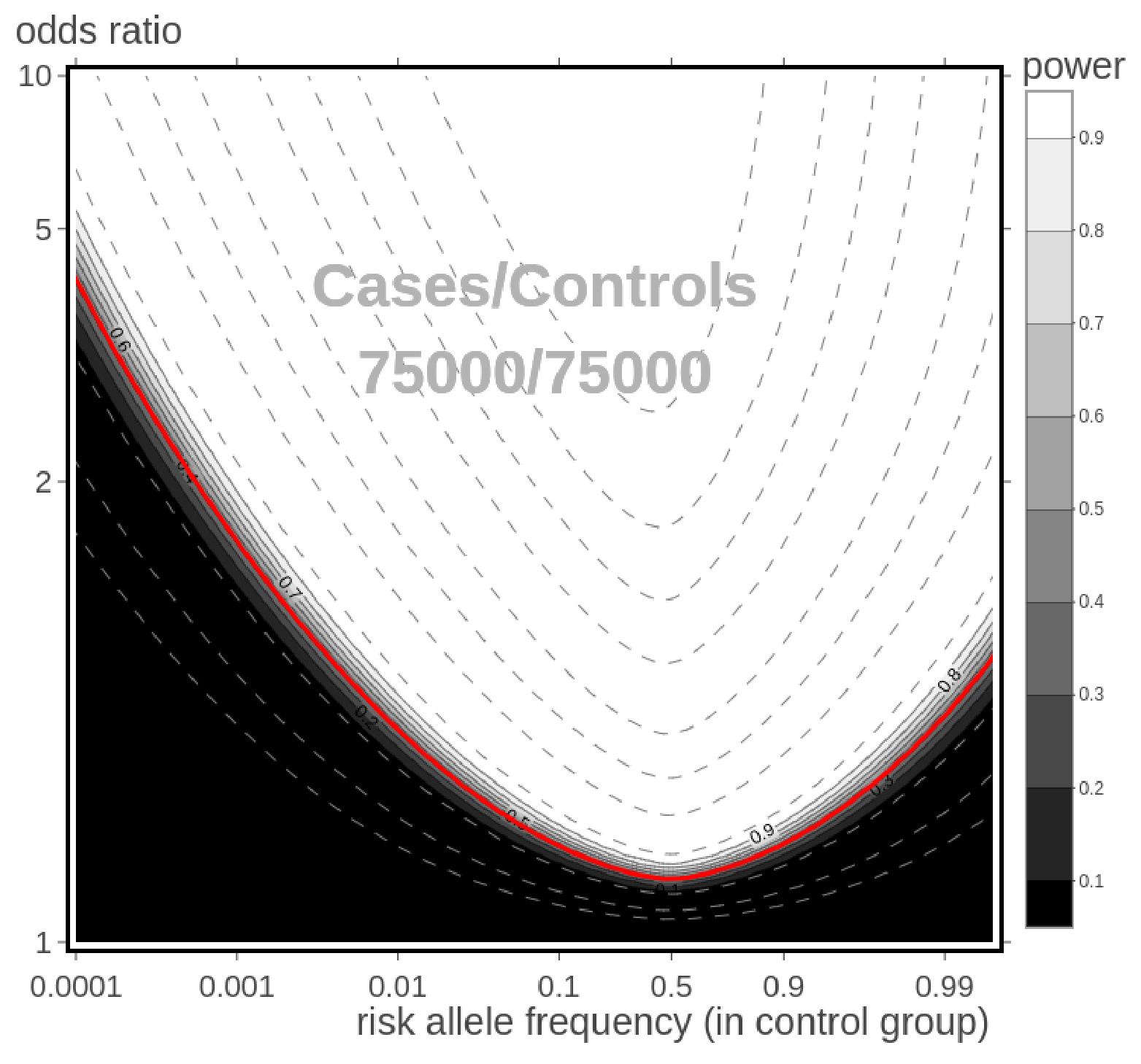}
    \includegraphics[width=0.49\textwidth]{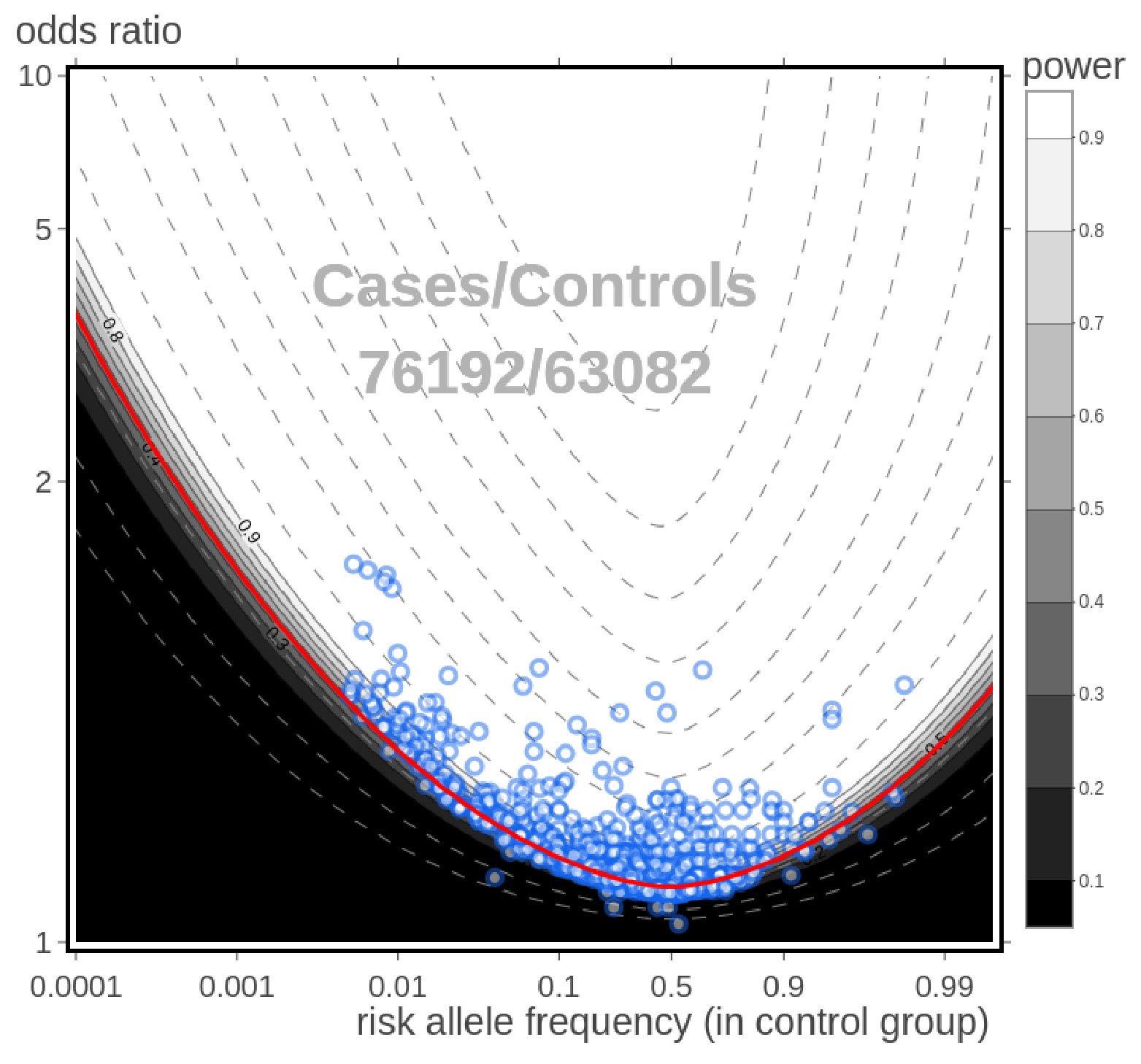}
    \caption{The OR-RAF diagram visualizing the marginal power of discovery in genetic association studies, after applying Bonferroni's procedure with nominal FWER at $5\%$ level. 
    Sample sizes are marked in each panel, and the problem dimensions are, respectively, $p=4$ (upper-left), $p=10^2$ (upper-right), and $p=10^6$ (lower-left), so that $n/\log{p}$ are roughly constant.
    Red curves mark the boundaries ($r=1$) of the phase transition for the exact-approximate support recovery problem; dashed curves are the equi-signal (equi-power) curves.
    The phase-transition in signal sizes $\lambda$ translates into the phase transition in terms of $(f,R)$, and sharpens as $p\to\infty$; see Example \ref{exmp:OR-RAF_phase_transition}.
    In the lower-right panel, we visualize discovered associations (blue circles) in a recent GWA study \cite{michailidou2017association}; the estimated odds ratios and risk allele frequencies are subject to survival bias and should not be taken at their face values; see Remark \ref{rmk:OR-RAF_false_evidence}.} 
    \label{fig:OR-RAF_GWAS}
\end{figure}

\begin{example}
\label{exmp:OR-RAF_phase_transition}
Consider association tests on $2\times2$ contingency tables at $p$ locations as introduced in Section \ref{sec:intro}, where the counts follow 
a multinomial distribution
% independent binomial distributions 
% $$
% O_{11} \sim \mathrm{Binom}(n\phi_1, fR/(fR+1-f)),\quad 
% O_{21} \sim \mathrm{Binom}(n(1-\phi_1), fR/(fR+1-f)),
% $$
parametrized by $(f, R, \phi_1)$ as in Section \ref{subsec:optimal-design}.
Assume that the phenotype marginals are fixed at $\phi_1 = \phi_2 = 1/2$.
% --- as is the case in genetic association studies --- 
Applying Bonferroni's procedure with nominal FWER at $\alpha=5\%$ level, we can approximate the marginal power of association tests by
\begin{equation} \label{eq:power-approximation}
    \P[\chi^2_{1}(\lambda)>\chi^2_{1,\alpha/p}],
\end{equation}
where $\chi^2_{1,\alpha/p}$ is the upper $(\alpha/p)$-quantile of a central chi-squared distribution with 1 degree of freedom.
We calculate this marginal power as a function of the parameters $(f,R)$ in three scenarios:
\begin{itemize}
    \item $p=4$, $n=3\times10^4$ 
    \item $p=10^2$, $n=1\times10^5$
    \item $p=10^6$, $n=3\times10^6$
\end{itemize}
and visualize the results as heatmaps\footnote{Since genetic variants can always be relabelled such that Variant 1 is positively associated with Cases, we only produce part of the diagram where $R>1$.
Sample sizes marked in the figure are adjusted by a factor of $1/2$, to reflect the genetic context where a pair of alleles are measured for every individual at every genomic location.} (referred to as OR-RAF diagrams) in Figure \ref{fig:OR-RAF_GWAS}.
These parameter values are chosen so that $\log{p}/n$ are roughly constant (around $4.6\times10^{-5}$).

We also overlay ``equi-signal'' curves, i.e., functions implicitly defined by the equations $r=c$ for a range of $c$ (dashed curves), and highlight the predicted boundary of phase transition for the exact-approximate support recovery problem $r=1$ (red curves).
The change in marginal power clearly sharpens around the predicted boundary $r=1$ as dimensionality diverges.
\end{example}

% --- or equivalently, the marginal power ---

\begin{remark}
\label{rmk:OR-RAF_false_evidence}
In an attempt to find empirical evidence of our theoretical predictions, we chart the genetic variants associated with breast cancer, discovered in a 2017 study by \citet{michailidou2017association} in an OR-RAF diagram. 
The estimated risk allele frequencies ($f$) and odds ratios ($R$) are taken from the NHGRI-EBI GWAS catalog \cite{macarthur2016new}, and plotted against a power heatmap calculated according to the reported sample sizes. 
See lower-right panel of Figure \ref{fig:OR-RAF_GWAS}.

It is tempting to believe, on careless inspection, that roughly \emph{all} discovered associations fall inside the high power region of the diagram, therefore demonstrating the phase transition in statistical power.
Unfortunately, the estimates here are subject to survival {bias} --- the study in fact uses the {same} dataset for \emph{both} support estimation and parameter estimation, without adjusting the latter for the selection process.
The seemingly striking agreement between the power calculations and the estimated effects of reported associations \emph{should not} be taken as evidence for the validity of our theory.
Nevertheless, we conjecture, as the theory predicts, that accurate and unbiased parameter estimates from an independent replication will still place the associations in the high power region of the diagram. 
\end{remark}

Finally, we demonstrate with an example how results in Sections \ref{sec:chisq-boundaries} and \ref{sec:signal-size-odds-ratio} may be used for planning prospective association studies.

\begin{example}
In a GWAS with $p = 10^6$ genomic marker locations, researchers wish to locate genetic associations with the trait of interest.
Specifically, they wish to maximize power in the region where genetic variants have risk allele frequencies of $0.01$ and odds ratios of $1.2$.
By Corollary \ref{cor:optimal-design}, the optimal design has a fraction of cases $\phi^* = 0.478$, yielding the statistical signal size per sample $w^2\approx9.00\times10^{-5}$ according to Corollary \ref{cor:signal-size-odds-ratio-conditional-frequency}.

If we wish to achieve exact-approximate support recovery in the sense of \eqref{eq:support-recovery-success}, Theorem \ref{thm:chi-squared-exact-approx-boundary} predicts that the signal size parameter $r$ has to be at least $\widetilde{g}(\beta)= 1$.
This signal size calls for a sample size of $n = \lambda / w^2 = 2r\log(p)/w^2 \approx 307,011$.
In a typically GWAS, a pair of alleles are sequenced for every marker location, bringing the required number of subjects in the study to $n/2 \approx 153,509$.
\end{example}

In comparison, a more accurate power calculation directly using \eqref{eq:power-approximation} predicts that $n / 2 = 165,035$ subjects are needed, under the set of parameters ($p=10^6$, $f=0.01$, $R=1.2$) and $\mathrm{FWER}=0.05$, $\mathrm{FNR}=0.5$; this is $7\%$ higher than our crude asymptotic approximation.
% The accuracy of the asymptotic approximations, by nature of the statements in Theorem \ref{thm:chi-squared-exact-boundary} and \ref{thm:chi-squared-approx-boundary}, depends on how close the error metrics are to zero.
% For example, the number of subjects needed for $\mathrm{FWER}=\mathrm{FWNR}=0.01$ is $499,598$, an $8\%$ increase over the asymptotic prediction; at $\mathrm{FWER}=\mathrm{FWNR}=0.1$, this number becomes $398,996$, some $14\%$ lower than the asymptotic result.
In general, we recommend using the more precise calculations over the back-of-the-envelope asymptotics for planning prospective studies and performing systematic reviews;
a user-friendly web application implementing the more precise approximations is provided in \cite{gao2019upass}.
Nevertheless, the phase transition results generate simple but powerful insights that cannot be easily supplanted.

\section{Phase transitions in the Gaussian additive error model}
\label{sec:additive-error-model-boundaries}

% $\mathrm{risk}^{\mathrm{EA}}$ and $\mathrm{risk}^{\mathrm{AE}}$.
As alluded to in the introduction, we draw explicit comparisons between the one-sided and two-sided alternatives in Gaussian additive error models \eqref{eq:model-additive}.
% The exact, and the approximate support recovery problems in the additive error model \eqref{eq:model-additive} under standard Gaussian errors have been studied in \cite{gao2018fundamental} and \cite{arias2017distribution}, respectively. 

\begin{remark} \label{rmk:strong-classification-boundary-1}
The exact support recovery problem under one-sided alternatives in the dependent Gaussian additive error model \eqref{eq:model-additive} was studied in \cite{gao2020fundamental}. 
The parametrization of sparsity was identical to \eqref{eq:signal-sparsity}, while the range of the non-zero (and perhaps unequal) mean shifts $\mu(i)$ was parametrized as
\begin{equation} \label{eq:signal-size-additive}
    % \underline{\Delta} = 
    \sqrt{2\underline{r}\log{p}}
    \le \mu(i) \le
    % \overline{\Delta} = 
    \sqrt{2\overline{r}\log{p}}, \quad \text{for all}\;\;i\in S_p,
\end{equation}
with constants $\underline{r}$ and $\overline{r}$, where $0<\underline{r}\le\overline{r}\le+\infty$.
Under this one-sided alternative, a phase transition in the $r$-$\beta$ plane was described, and the boundary was found to be identical to \eqref{eq:exact-boundary-chisquared} in Theorem \ref{thm:chi-squared-exact-boundary}. The latter, as discussed in Section \ref{subsec:motivation-additive}, covers two-sided alternatives in the additive model \eqref{eq:model-additive}.

In other words, comparing the two-sided alternative versus its one-sided counterpart, there is asymptotically no difference in terms of the signal sizes needed to achieve exact support recovery.
As we shall see in numerical experiments (in Section \ref{sec:numerical} below), the difference is not very pronounced even in moderate dimensions, and vanishes as $p\to\infty$, in accordance with Theorem \ref{thm:chi-squared-exact-boundary}.
\end{remark}

A similar comparison can be drawn in the approximate support recovery problem between the two types of alternatives.

\begin{remark} \label{rmk:weak-classification-boundary}
The approximate support recovery problem in the Gaussian additive error model \eqref{eq:model-additive} under one-sided alternatives was studied in \cite{arias2017distribution}, 
where the phase transition phenomenon was characterized by a boundary that coincides with \eqref{eq:approx-boundary-chisquared} in Theorem \ref{thm:chi-squared-approx-boundary}.
Similar to the exact support recovery problem, this indicates vanishing difference in the difficulties of the two alternatives.
\end{remark}

Finally, we derive two new asymptotic results under the \emph{asymmetric} statistical risks, \eqref{eq:risk-exact-approx} and \eqref{eq:risk-approx-exact}, under one-sided alternatives.
First, a counterpart of Theorem \ref{thm:chi-squared-exact-approx-boundary} describes the phase transition in the exact-approximate support recovery problem.

\begin{theorem} \label{thm:additive-error-exact-approx-boundary}
Consider the high-dimensional additive error model \eqref{eq:model-additive} under independent standard Gaussian errors, with signal sparsity and size as described in \eqref{eq:signal-sparsity} and \eqref{eq:signal-size-additive}.
The function $\widetilde{g}(\beta)$ in \eqref{eq:exact-approx-boundary-chisquared} characterizes the phase transition of exact-approximate support recovery problem.

Specifically, if $\underline{r} > \widetilde{g}(\beta)$, then the procedures listed in Theorem \ref{thm:chi-squared-exact-boundary} with slowly vanishing nominal FWER levels achieve asymptotically exact-approximate support recovery in the sense of \eqref{eq:support-recovery-success}. 
Conversely, if $\overline{r} < \widetilde{g}(\beta)$, then for any thresholding procedure $\widehat{S}$, the exact-approximate support recovery fails in the sense of \eqref{eq:support-recovery-faliure}.
\end{theorem}

A counterpart of Theorem \ref{thm:chi-squared-approx-exact-boundary} also holds under one-sided alternatives.

\begin{theorem} \label{thm:additive-error-approx-exact-boundary}
In the context of Theorem \ref{thm:additive-error-exact-approx-boundary}, the function $\widetilde{h}(\beta)$ in \eqref{eq:approx-exact-boundary-chisquared}
characterizes the phase transition of approximate-exact support recovery problem.

Specifically, if $\underline{r} > \widetilde{h}(\beta)$, then the Benjamini-Hochberg procedure with slowly vanishing nominal FDR levels achieves asymptotically approximate-exact support recovery in the sense of \eqref{eq:support-recovery-success}. 
Conversely, if $\overline{r} < \widetilde{h}(\beta)$, then for any thresholding procedure $\widehat{S}$, the approximate-exact support recovery fails in the sense of \eqref{eq:support-recovery-faliure}.
\end{theorem}

Theorems \ref{thm:additive-error-exact-approx-boundary} and \ref{thm:additive-error-approx-exact-boundary} are proved in Appendix \ref{subsec:proof-additive-error-mix-boundaries}. 

\begin{remark}
Comparing Theorems \ref{thm:chi-squared-exact-approx-boundary} to \ref{thm:additive-error-exact-approx-boundary} and Theorems \ref{thm:chi-squared-approx-exact-boundary} to \ref{thm:additive-error-approx-exact-boundary}, we see that the phase transition boundaries under the two types of alternatives are identical in the exact-approximate and approximate-exact support recovery problems.
As pointed out in Remarks \ref{rmk:strong-classification-boundary-1} and \ref{rmk:weak-classification-boundary}, the additional uncertainty in the two-sided alternatives do not call for larger signal sizes asymptotically.

To complete the comparisons, we point out that the phase transition boundaries for the sparse signal \emph{detection} problem in the two types of alternatives are both identical to \eqref{eq:detection-boundary-chisquare}. This was analyzed in \cite{donoho2004higher}.
\end{remark}

\section{Numerical illustrations}
\label{sec:numerical}

We illustrate with simulations the phase transition phenomena in the chi-square model, and compare numerically the required signal sizes in support recovery problems between the two types of alternatives in the additive error model.
% We also demonstrate the fundamental trade-off between odds ratios and relative frequencies in association studies, as outlined in Proposition \ref{prop:signal-size-odds-ratio} and Corollary \ref{cor:signal-size-odds-ratio-conditional-frequency}, using evidence from large-scale genetic studies.

\subsection{The exact support recovery problem}

The sparsity of the signal vectors in the experiments are parametrized as in \eqref{eq:signal-sparsity}. 
Signal sizes are assumed equal with magnitude $\lambda(i)=2r\log{p}$ for $i\in S$.
We estimate the support set $S$ using Bonferroni's procedure with nominal FWER level set at $1/(5{\log{p}})$.
The nominal FWER levels vanishes slowly, in line with the assumptions in Theorem \ref{thm:chi-squared-exact-boundary}.
Experiments were repeated 1000 times at each of the 400 sparsity-signal-size combinations, for dimensions $p=10^2, 10^3$, and $10^4$.

The empirical probabilities of exact support recovery under Bonferroni's procedure are shown in Figure \ref{fig:phase-simulated-chi-squared}.
The numerical results suggest not only good accuracy of the predicted boundaries in high-dimensions ($p=10^4$, right panels of Figure \ref{fig:phase-simulated-chi-squared}), but also practical relevance of the theoretical predictions in moderate dimensions ($p=100$, left panels of Figure \ref{fig:phase-simulated-chi-squared}).

\begin{figure}
      \centering
      \includegraphics[width=0.32\textwidth]{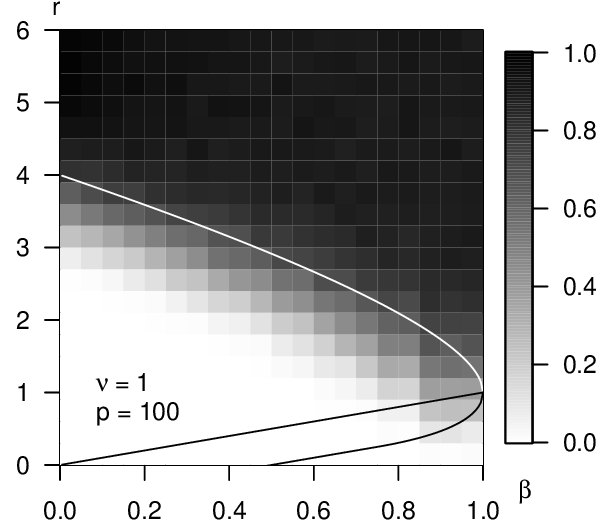}
      \includegraphics[width=0.32\textwidth]{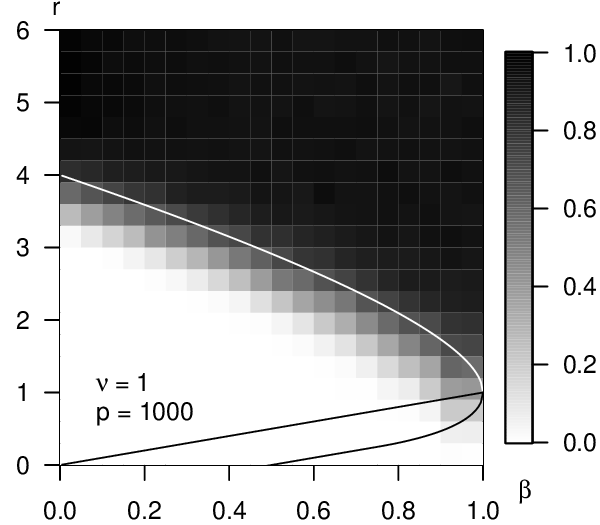}
      \includegraphics[width=0.32\textwidth]{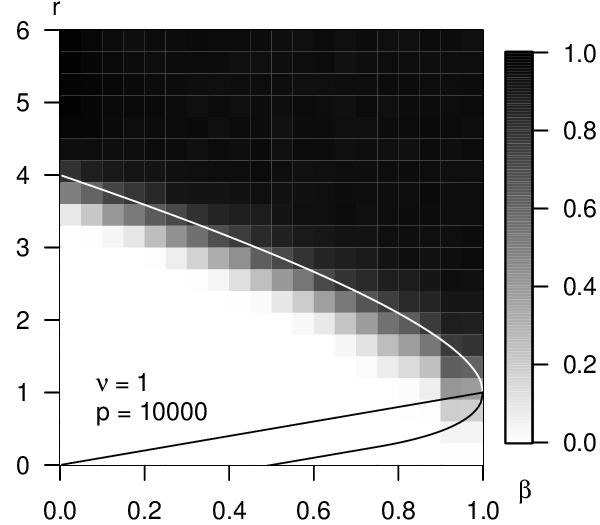}
      \includegraphics[width=0.32\textwidth]{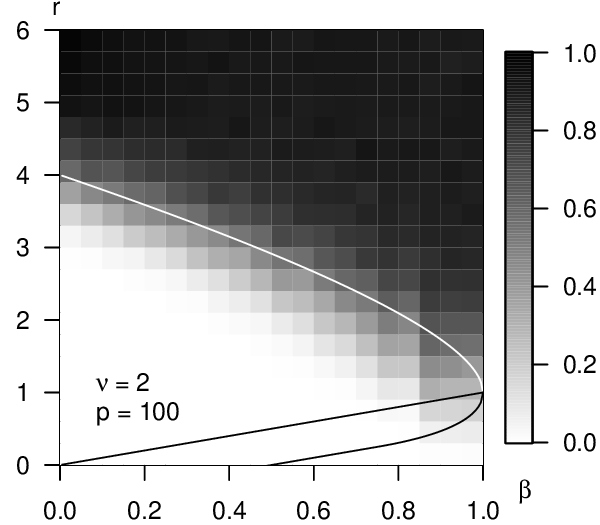}
      \includegraphics[width=0.32\textwidth]{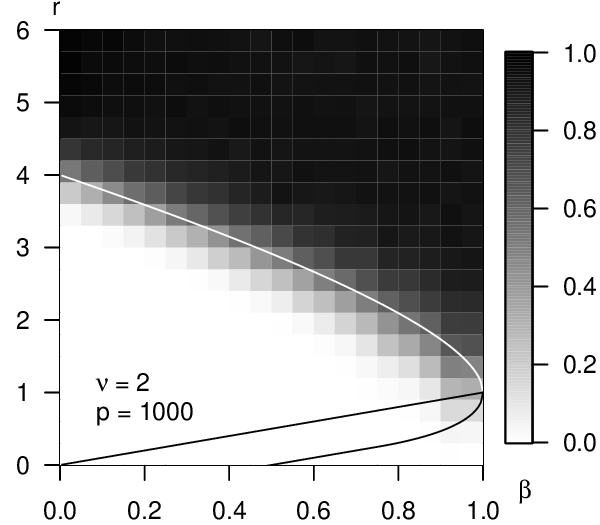}
      \includegraphics[width=0.32\textwidth]{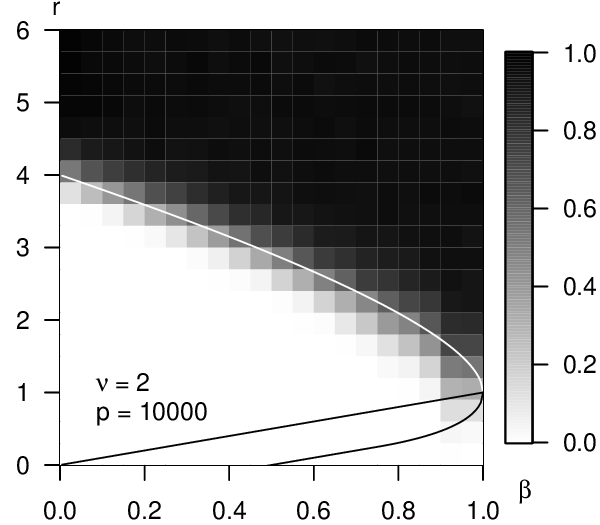}
      \includegraphics[width=0.32\textwidth]{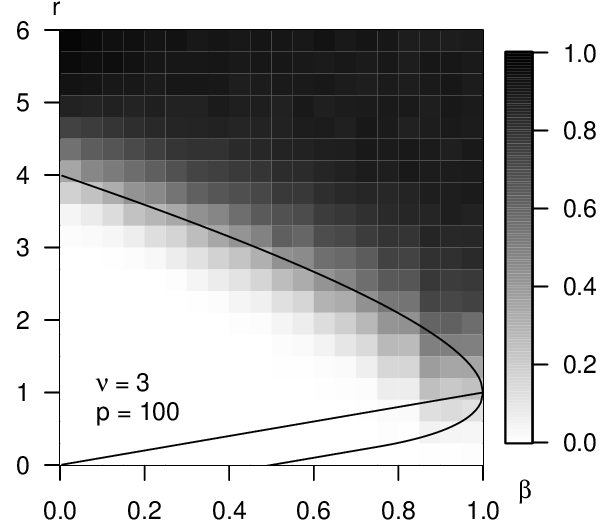}
      \includegraphics[width=0.32\textwidth]{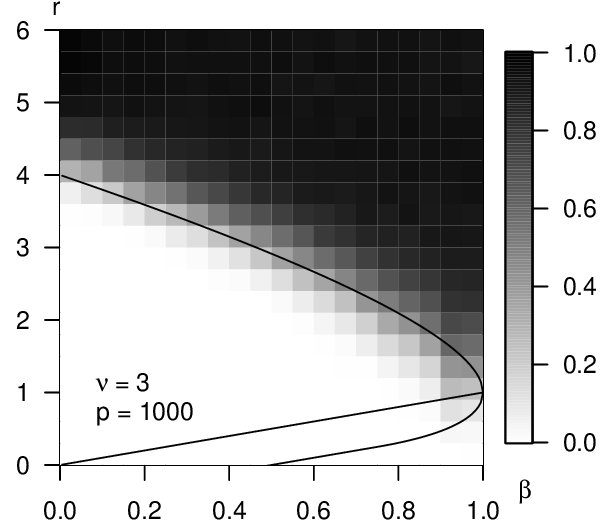}
      \includegraphics[width=0.32\textwidth]{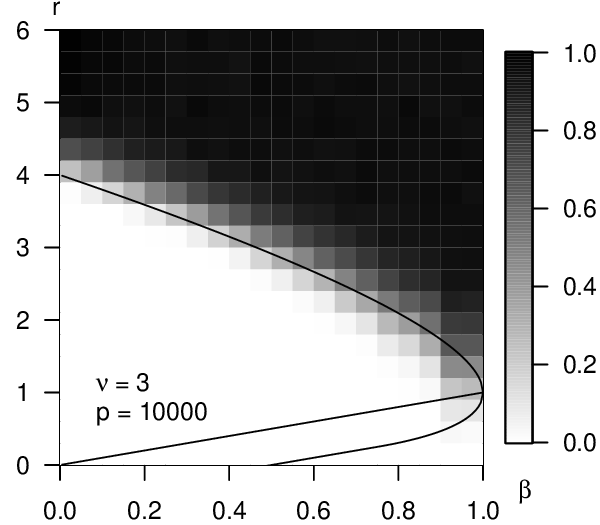}
      \includegraphics[width=0.32\textwidth]{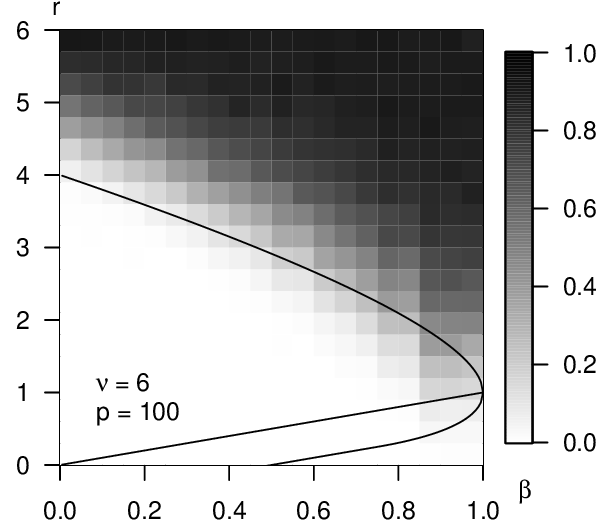}
      \includegraphics[width=0.32\textwidth]{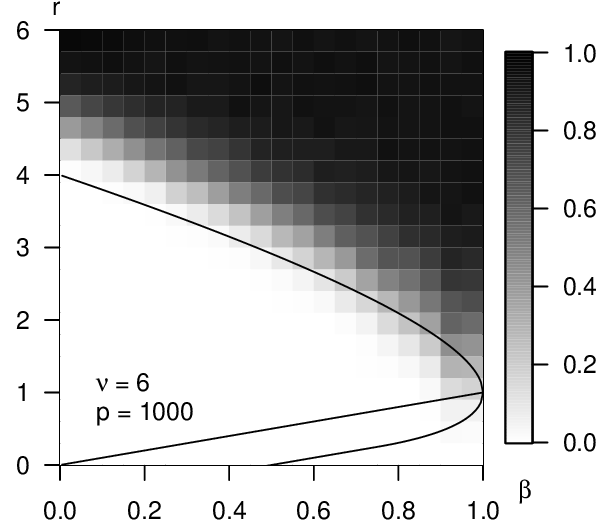}
      \includegraphics[width=0.32\textwidth]{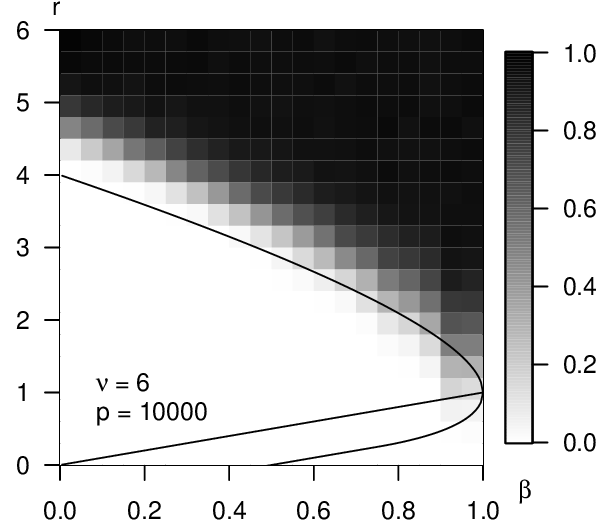}
      \caption{The empirical probability of exact support recovery of Bonferroni's procedure in the chi-squared model \eqref{eq:model-chisq}. 
      We simulate $\nu=1, 2, 3, 6$ (first to last row), at dimensions $p=10^2, 10^3, 10^4$ (left to right column), for a grid of sparsity levels $\beta$ and signal sizes $r$.
      The experiments were repeated 1000 times for each sparsity-signal size combination; darker color indicates higher probability of exact support recovery.  
      Numerical results are in general agreement with the boundaries described in Theorem \ref{thm:chi-squared-exact-boundary}; for large $\nu$'s, the phase transitions take place somewhat above the predicted boundaries.
      The boundary for the approximate support recovery (Theorem \ref{thm:chi-squared-approx-boundary}) and the detection boundary (see \citep{donoho2004higher}) are plotted for comparison.} 
      \label{fig:phase-simulated-chi-squared}
\end{figure}

We conduct further experiments to examine the optimality claims in Theorem \ref{thm:chi-squared-exact-boundary} by comparing with the oracle procedure with thresholds $t_p=\min_{i\in S}x(i)$.
We also examine the claims in Remark \ref{rmk:strong-classification-boundary-1}, and compare the one-sided alternatives in Gaussian additive models with the two-sided alternatives (or equivalently, the chi-square model with $\nu=1$).
We apply Bonferroni's procedure and the oracle thresholding procedure in both settings.

Experiments were repeated 1000 times for a grid of signal size values ranging from $r=0$ to $6$, and for dimensions $10^2, 10^3$, and $10^5$.
Results of the experiments, shown in Figure \ref{fig:one-vs-two-sided-exact_support_recovery}, suggest vanishing difference between difficulties of two-sided vs one-sided alternatives in the additive error models, as well as vanishing difference between the powers of Bonferroni's procedures and the oracle procedures as $p\to\infty$.

\begin{figure}
      \centering
      \includegraphics[width=0.32\textwidth]{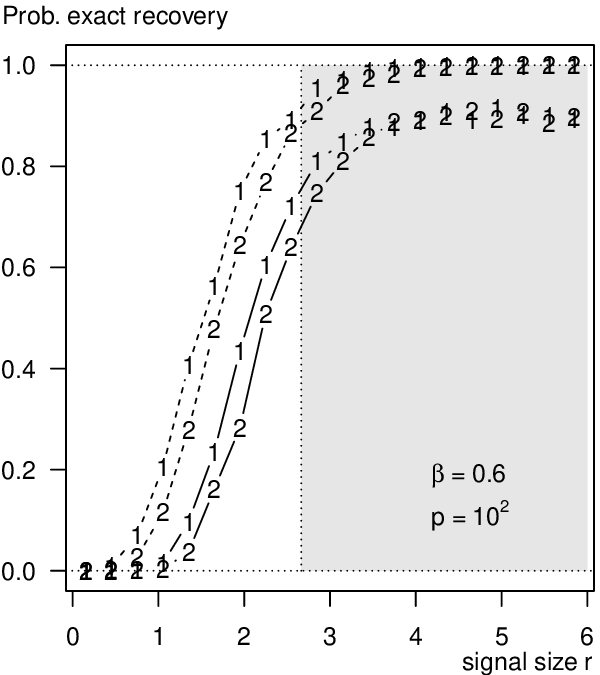}
      \includegraphics[width=0.32\textwidth]{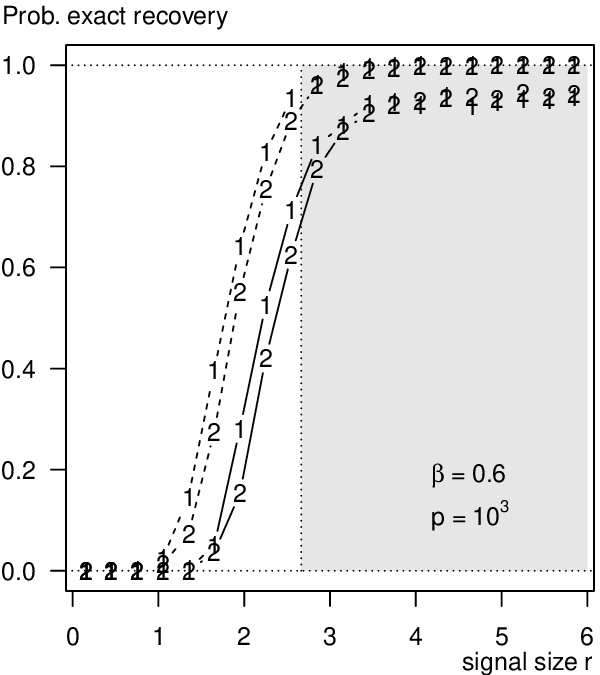}
      \includegraphics[width=0.32\textwidth]{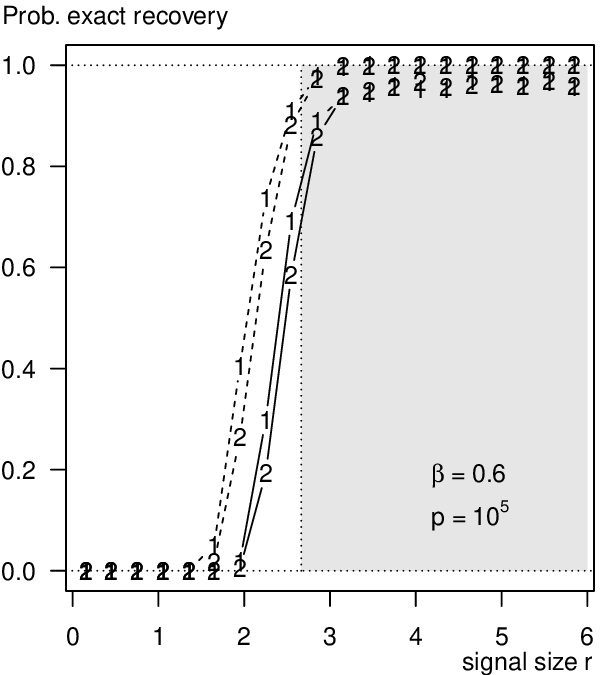}
      \caption{The empirical probability of exact support recovery of Bonferroni's procedure (solid curves) and the oracle procedure (dashed curves) in the chi-squared model with one degree of freedom (marked `2') in the additive Gaussian error model and under one-sided alternatives (marked `1'). 
      We simulate at dimensions $p=10^2, 10^3, 10^5$ (left to right) for a grid of signal sizes $r$ and sparsity level $\beta=0.6$.
      The experiments were repeated 1000 times for each method-model-signal-size combination. 
      Numerical results show evidence of convergence to the 0-1 law as predicted by Theorem \ref{thm:chi-squared-exact-boundary}; regions where asymptotically exact support recovery can be achieved are shaded in grey.
      The difference in power between Bonferroni's procedure and the oracle procedure, as well as in the two types of alternatives both decrease as dimensionality increases.} 
      \label{fig:one-vs-two-sided-exact_support_recovery}
\end{figure}

\subsection{The approximate, and approximate-exact support recovery problem}

Similar experiments are conducted to examine the optimality claims in Theorem \ref{thm:chi-squared-approx-boundary}, and in Remark \ref{rmk:weak-classification-boundary}.
We define an oracle thresholding procedure for approximate support recovery, where the threshold is chosen to minimize the empirical risk.
That is,
$$
t_p(x, S) \in \argmin_{t\in\R} \frac{|\widehat{S}(t)\setminus S|}{\max\{|\widehat{S}(t)|,1\}} + \frac{|S\setminus \widehat{S}(t)|}{\max\{|{S}|,1\}},
%\mathcal{R^\mathrm{oracle}} \in \argmin_{\widehat{S}(\mathcal{R})\in\mathcal{S}} \mathrm{risk}^{\mathrm{A}}(\mathcal{R}),
$$
where $\widehat{S}(t) = \{i\;|\;x(i)\ge t\}$;
in implementation, we only need to scan the values of observations $t\in\{x(1), \ldots, x(p)\}$. 
The nominal FDR level for the BH procedure is set at $1/(5{\log{p}})$, therefore slowly vanishing, in line with the assumptions in Theorem \ref{thm:chi-squared-approx-boundary}; all other parameters are identical to that in the experiments for exact support recovery.
Results of the experiments are shown in Figure \ref{fig:one-vs-two-sided-approx_support_recovery} and Figure \ref{fig:phase-simulated-chi-squared-approx-boundary}.

We also examine the boundary described in Theorem \ref{thm:chi-squared-exact-approx-boundary}.
Experimental settings are identical to that in the experiments for approximate support recovery.
% Results for the BH procedure are in general close to that of the oracle procedure,
We compare the performance of the BH procedure with an oracle procedure with threshold
$$
t_p(x, S) \in \min_{i\in S} x(i),
$$
and visualize results of the experiments in Figure \ref{fig:phase-simulated-chi-squared-approx-exact-boundary}.
Notice that the BH procedure sets its threshold somewhat higher than the oracle, especially for small $\beta$'s. 
The empirical risk of the oracle procedure (not shown here in the interest of space) follows much more closely the predicted boundary \eqref{eq:approx-exact-boundary-chisquared}.

\begin{figure}
      \centering
      \includegraphics[width=0.32\textwidth]{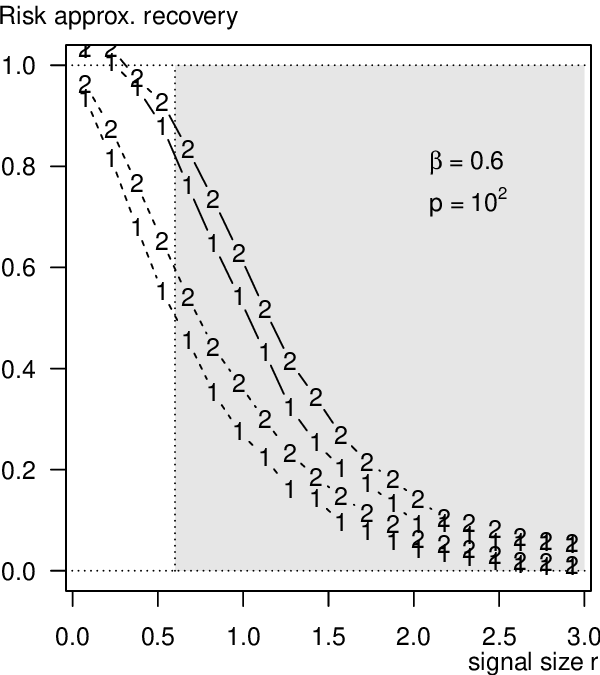}
      \includegraphics[width=0.32\textwidth]{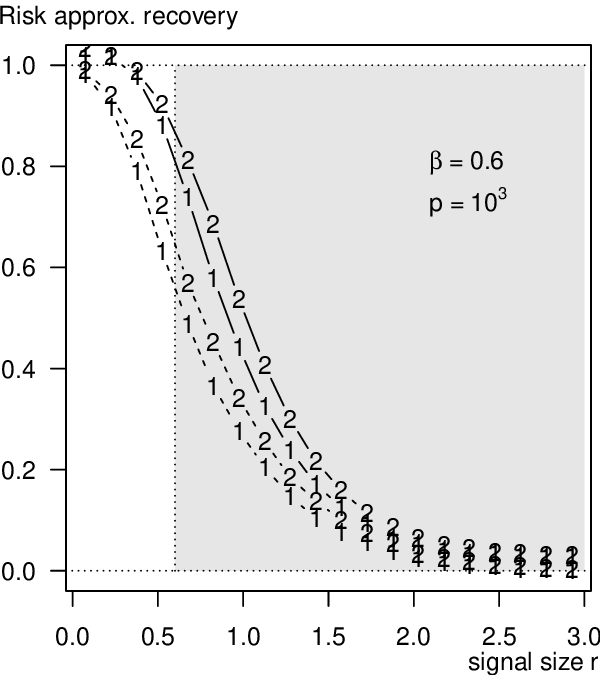}
      \includegraphics[width=0.32\textwidth]{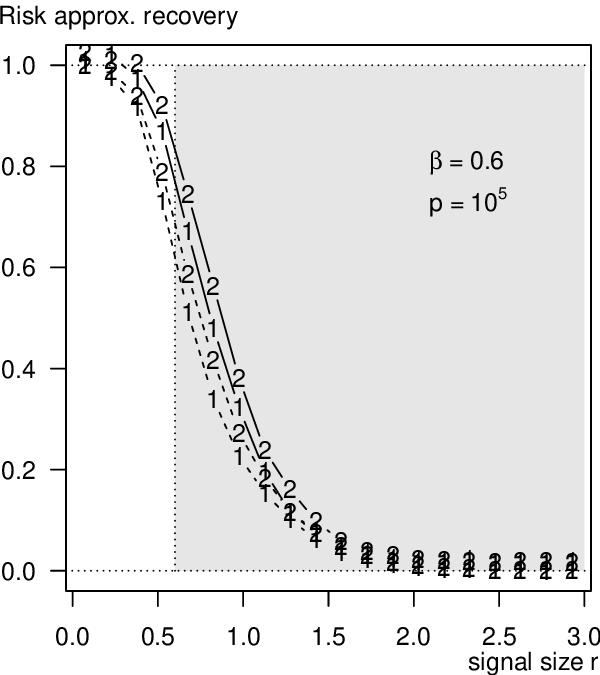}
      \caption{The empirical risk of approximate support recovery of Benjamini-Hochberg's procedure (solid curves) and the oracle procedure (dashed curves) in the chi-squared model with one degree of freedom (marked `2') and in the additive Gaussian error model under one-sided alternatives (marked `1'). 
      We simulate at dimensions $p=10^2, 10^3, 10^5$ (left to right) for a grid of signal sizes $r$ and sparsity level $\beta=0.6$.
      The experiments were repeated 1000 times for each method-model-signal-size combination. 
      Numerical results show evidence of convergence to the 0-1 law as predicted by Theorem \ref{thm:chi-squared-approx-boundary}; regions where asymptotically approximate support recovery can be achieved are shaded in grey.
      The difference in risks between Benjamini-Hochberg's procedure and the oracle procedure, as well as in the two types of alternatives, both decrease as dimensionality increases.} 
      \label{fig:one-vs-two-sided-approx_support_recovery}
\end{figure}

\begin{figure}
      \centering
      \includegraphics[width=0.32\textwidth]{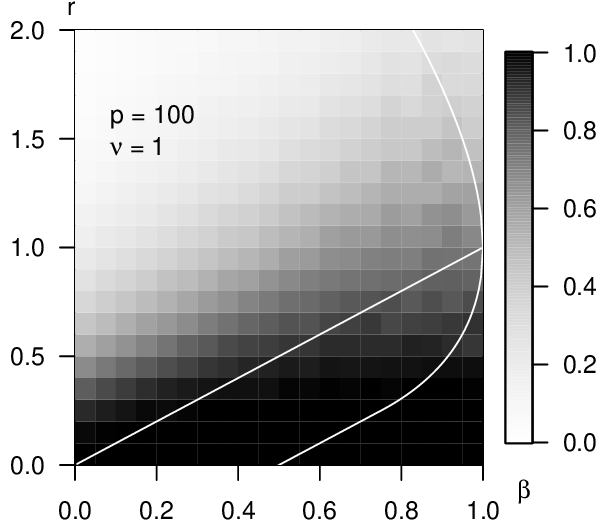}
      \includegraphics[width=0.32\textwidth]{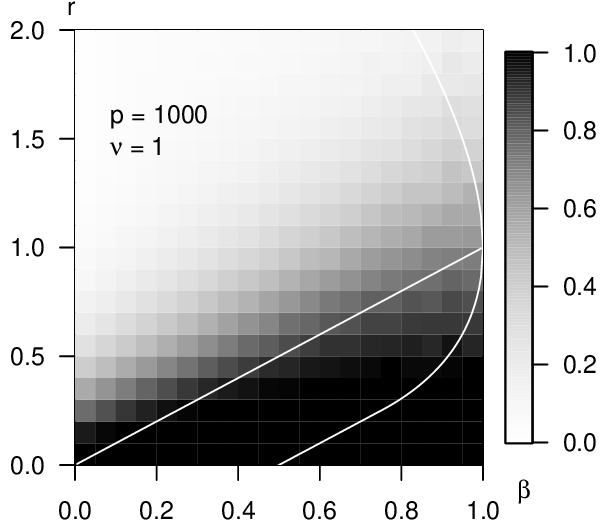}
      \includegraphics[width=0.32\textwidth]{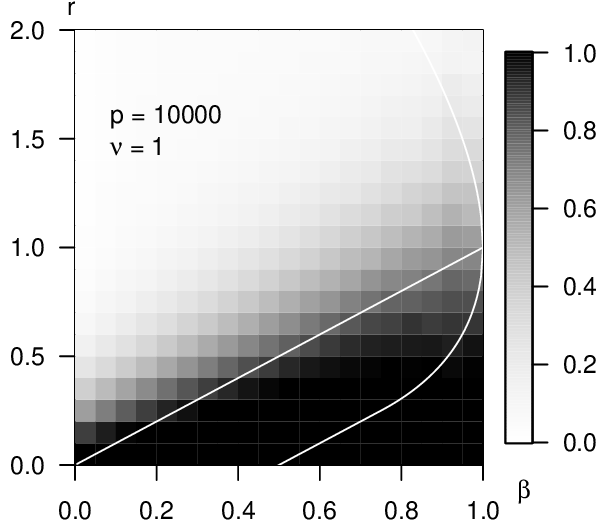}
      \includegraphics[width=0.32\textwidth]{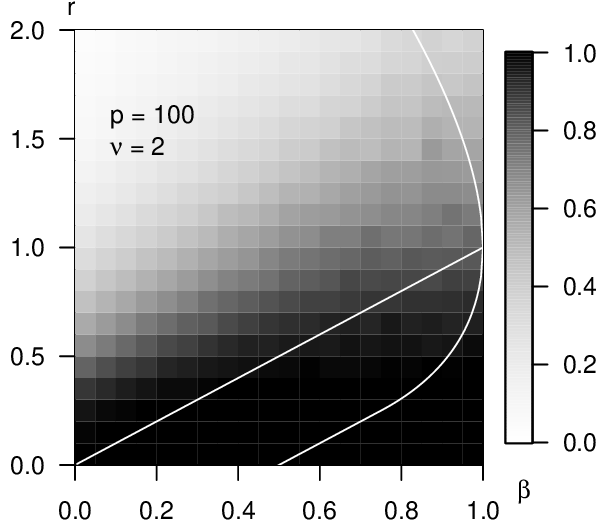}
      \includegraphics[width=0.32\textwidth]{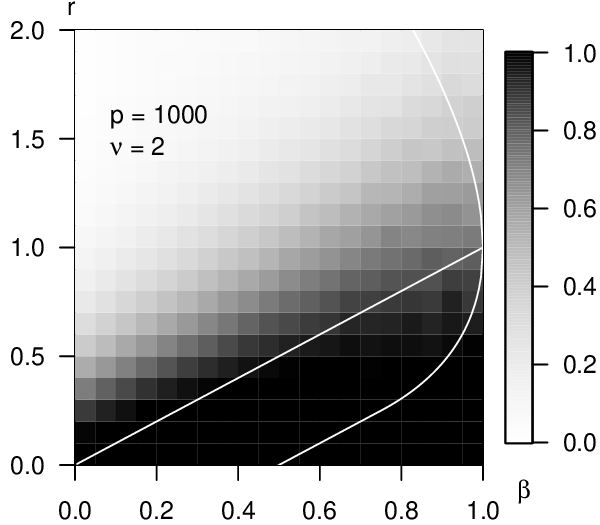}
      \includegraphics[width=0.32\textwidth]{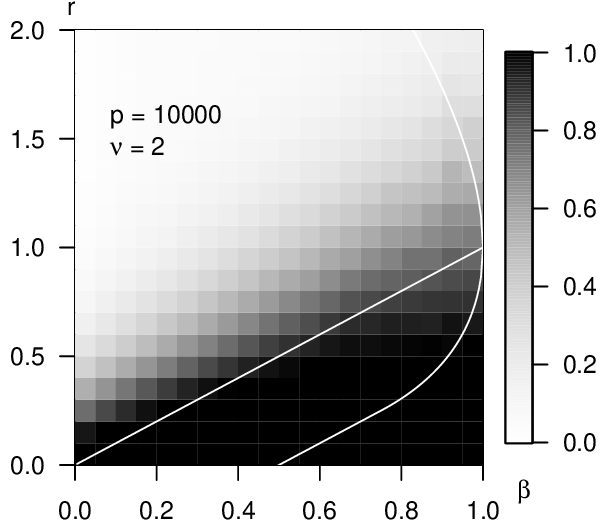}
      \includegraphics[width=0.32\textwidth]{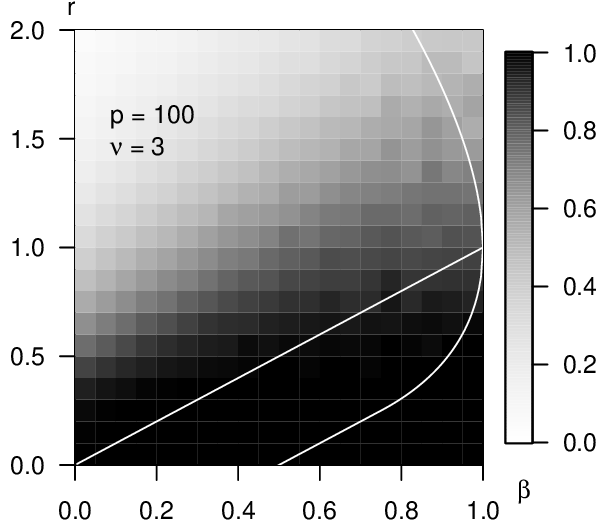}
      \includegraphics[width=0.32\textwidth]{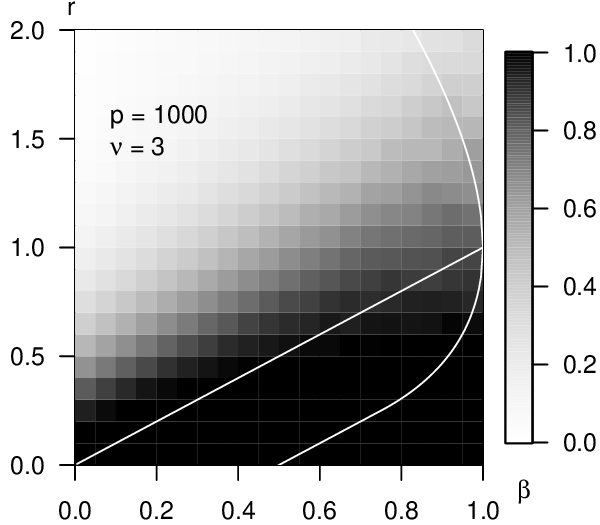}
      \includegraphics[width=0.32\textwidth]{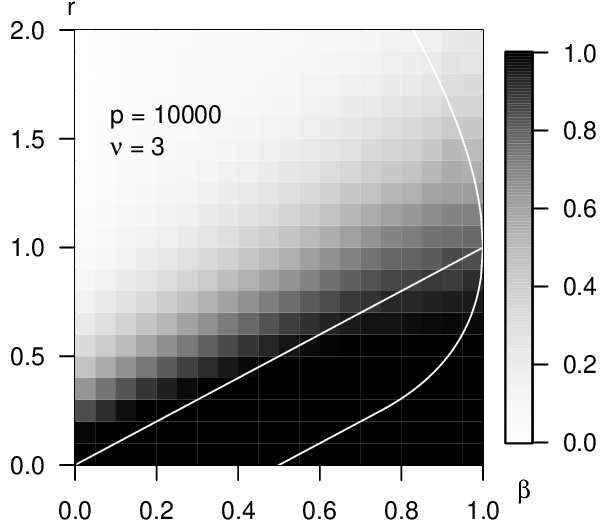}
      \includegraphics[width=0.32\textwidth]{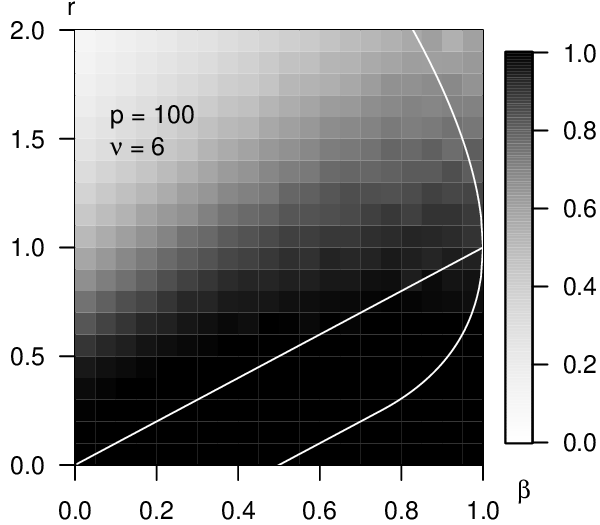}
      \includegraphics[width=0.32\textwidth]{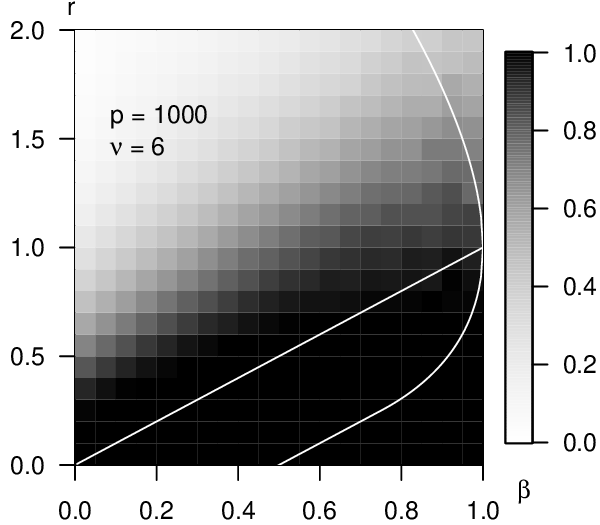}
      \includegraphics[width=0.32\textwidth]{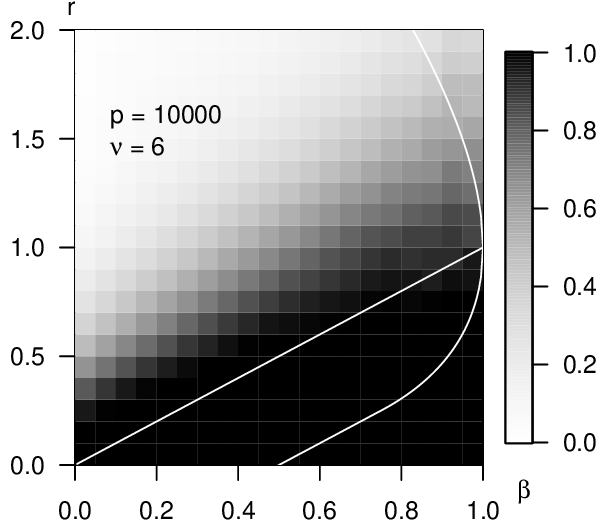}
      \caption{The estimated risk of approximate support recovery $\mathrm{risk}^{\mathrm{A}}$ (see \eqref{eq:risk-approximate}) of the Benjamini-Hochberg procedure in the chi-squared model \eqref{eq:model-chisq}. 
      We simulate $\nu=1, 2, 3, 6$ (first to last row), at dimensions $p=10^2, 10^3, 10^4$ (left to right column), for a grid of sparsity levels $\beta$ and signal sizes $r$.
      The experiments were repeated 1000 times for each sparsity-signal size combination; darker color indicates higher larger $\mathrm{risk}^{\mathrm{A}}$. 
      Numerical results are generally in agreement with the boundaries described in Theorem \ref{thm:chi-squared-approx-boundary}; for large $\nu$'s, the phase transitions take place somewhat above the predicted boundaries.
      The boundary for the exact support recovery problem (Theorem \ref{thm:chi-squared-exact-boundary}) and the detection boundary (see \citep{donoho2004higher}) are plotted for comparison.} 
      \label{fig:phase-simulated-chi-squared-approx-boundary}
\end{figure}

\begin{figure}
      \centering
      \includegraphics[width=0.32\textwidth]{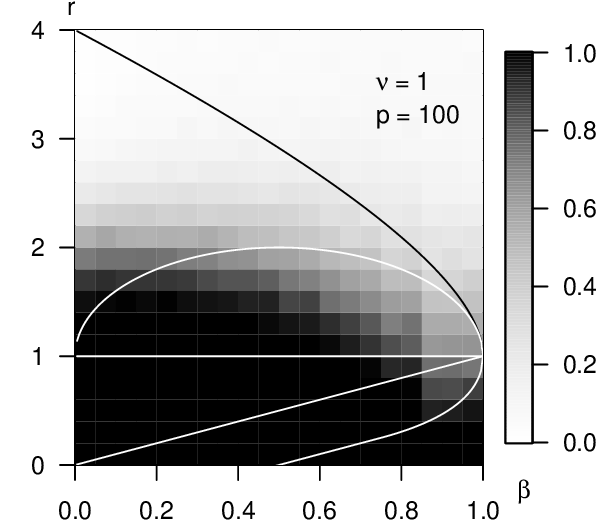}
      \includegraphics[width=0.32\textwidth]{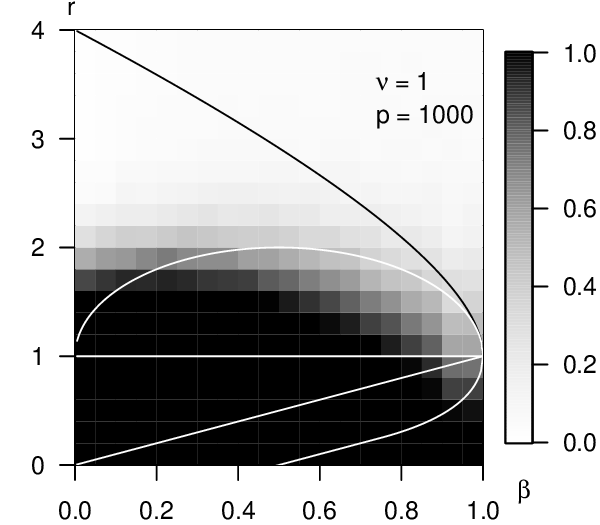}
      \includegraphics[width=0.32\textwidth]{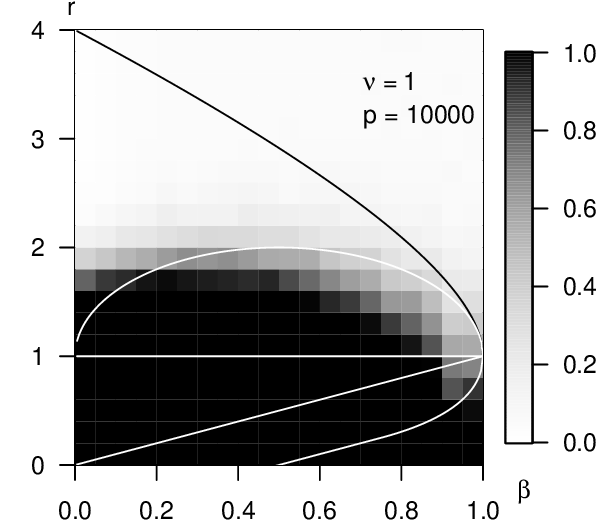}
      \includegraphics[width=0.32\textwidth]{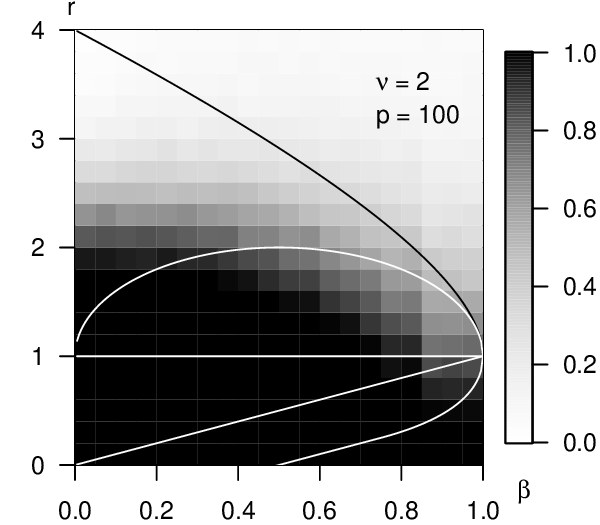}
      \includegraphics[width=0.32\textwidth]{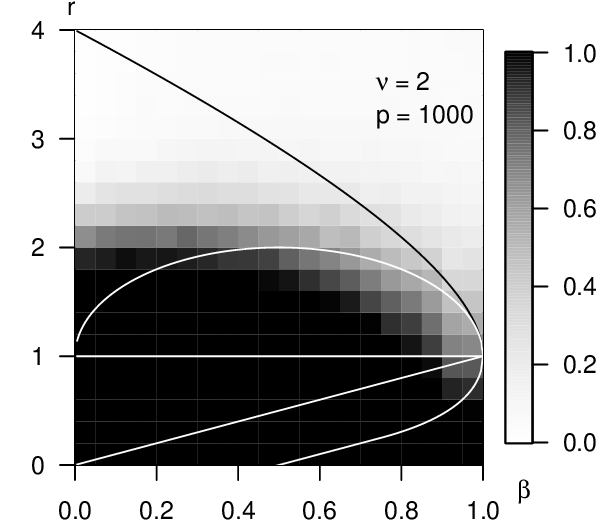}
      \includegraphics[width=0.32\textwidth]{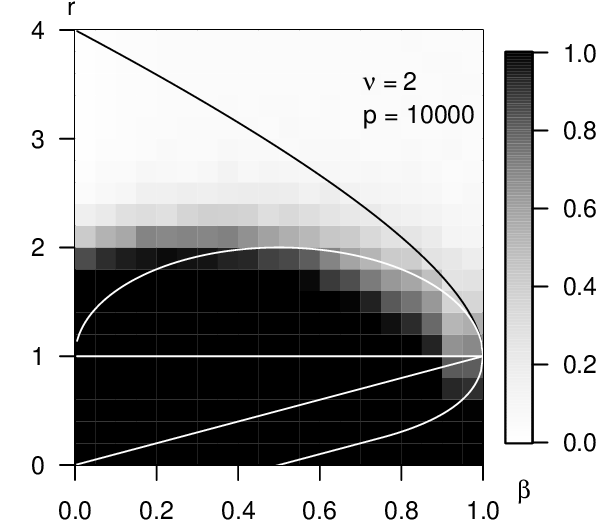}
      \includegraphics[width=0.32\textwidth]{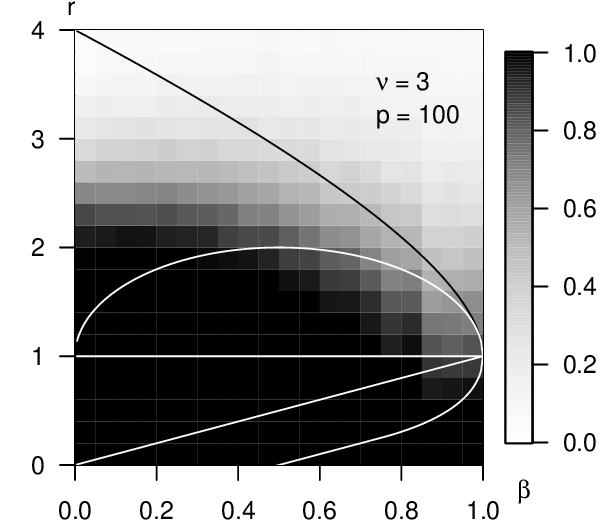}
      \includegraphics[width=0.32\textwidth]{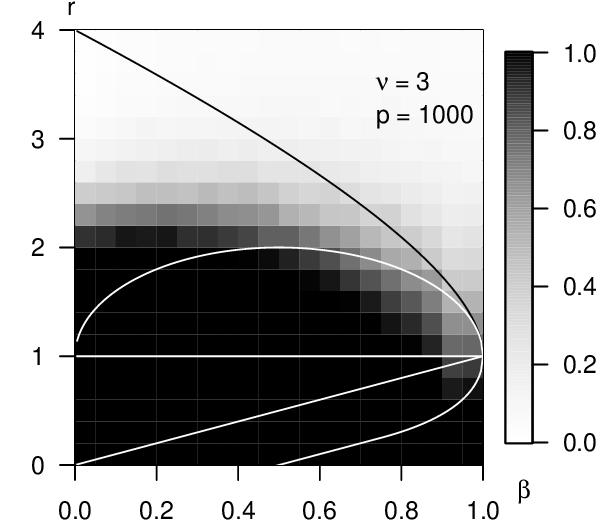}
      \includegraphics[width=0.32\textwidth]{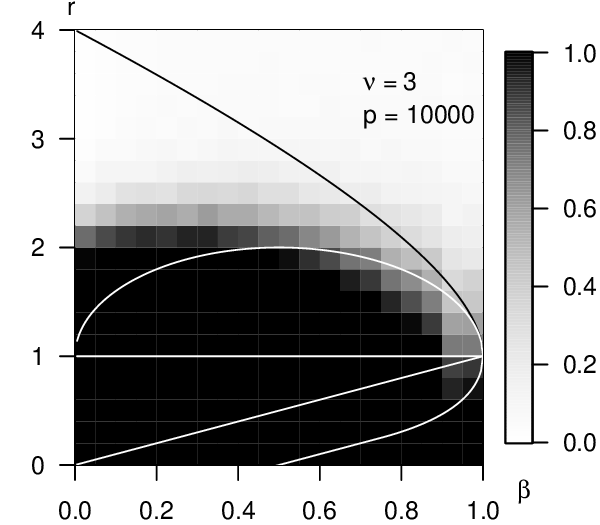}
      \includegraphics[width=0.32\textwidth]{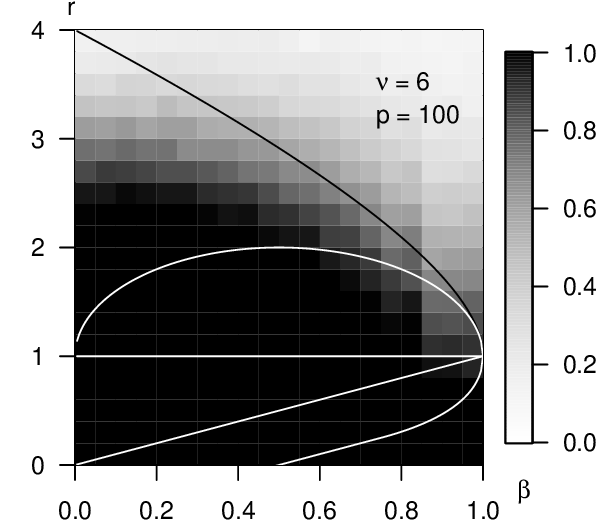}
      \includegraphics[width=0.32\textwidth]{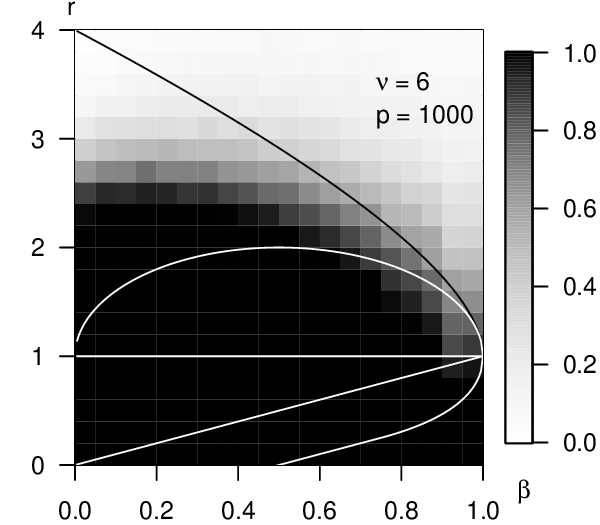}
      \includegraphics[width=0.32\textwidth]{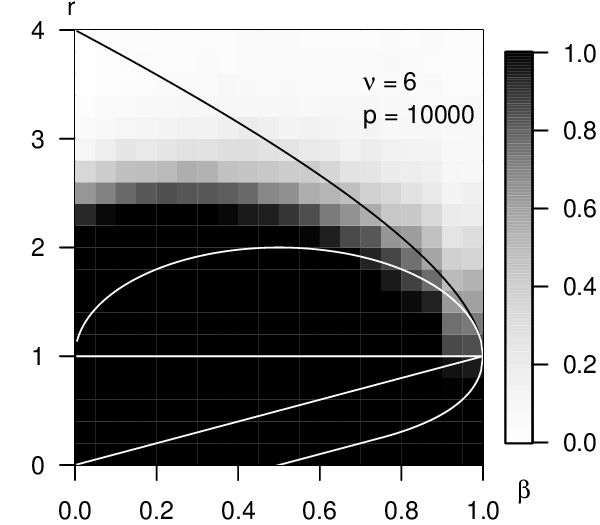}
      \caption{The estimated risk of approximate-exact support recovery $\mathrm{risk}^{\mathrm{EA}}$ (see \eqref{eq:risk-approx-exact}) of the Benjamini-Hochberg procedure in the chi-squared model \eqref{eq:model-chisq}. 
      We simulate $\nu=1, 2, 3, 6$ (first to last row), at dimensions $p=10^2, 10^3, 10^4$ (left to right column), for a grid of sparsity levels $\beta$ and signal sizes $r$.
      The experiments were repeated 1000 times for each sparsity-signal size combination; darker color indicates higher larger $\mathrm{risk}^{\mathrm{EA}}$. 
      Numerical results are generally in agreement with the boundaries described in Theorem \ref{thm:chi-squared-approx-exact-boundary}; for small $\beta$'s and large $\nu$'s, the phase transitions take place somewhat above the predicted boundaries.
      Other boundaries in the support recovery and the detection problems are plotted for comparison.} 
      \label{fig:phase-simulated-chi-squared-approx-exact-boundary}
\end{figure}

\section{Discussions}
\label{sec:discussions}

We would be remiss not to mention some closely related studies on multiple testing in high-dimensions, and the interesting theoretical results derived therein.

Performance limits of statistical procedures in the additive error model \eqref{eq:model-additive} have been actively studied in recent years.
Specifically, when the approximate support recovery risk \eqref{eq:risk-approximate} is used as the criterion, the asymptotic optimality of the Benjamini-Hochberg procedure \cite{benjamini1995controlling}, and the Cand\'es-Barber procedure \cite{barber2015controlling} was established by \citet*{arias2017distribution} under independent additive errors and one-sided alternatives.
% \citet{rabinovich2017optimal} further established the rate-optimality of both procedures under the same regime.

In terms of the more stringent exact support recovery probability \eqref{eq:risk-prob}, several well-known FWER-controlling procedures --- including Bonferroni's procedure --- have been shown to be optimal in the additive error model under one-sided alternatives \cite{gao2020fundamental}.
Optimality results were also obtained for a specific procedure under the expected Hamming loss ($\E[\widehat{S}\triangle S]$) in \cite{butucea2018variable},
% Theoretical limits in the Gaussian additive error model \eqref{eq:model-additive} have been recently studied in \cite{butucea2018variable} under the Hamming loss, $\E[\widehat{S}\triangle S]$.
where the asymptotic analyses focused exclusively on the two-sided alternatives.
% , while the former focused exclusively on the one-sided alternatives. 
% falling just short of an explicit comparison between the two.
% , due in part to slightly different goals of the projects. 
% Our discussions above also applied to the Hamming loss.

There is a wealth of literature on the so-called sparsistency (i.e., $\P[\widehat{S} = S]\to 1$ as $p\to\infty$) problem in the regression context. 
We refer readers to the recent book by \citet{wainwright2019high} (and in particular, the bibliographical sections of Chapters 7 and 15) for a comprehensive review.
We only mention two additional papers by \citet{ji2012ups} and \citet{jin2014optimality} which derived minimax optimality results under the Hamming loss in this context.

The asymptotic optimality of the Bonferroni and Benjamini-Hochberg procedures 
% in the Gaussian scale mixture model 
was analyzed under decision theoretic frameworks in \cite{genovese2002operating, bogdan2011asymptotic, neuvial2012false}, with main focus on location/scale models. 
In particular, these papers show that the statistical risks of the practical procedures come close to that of the oracle procedures under suitable asymptotic regimes.
Strategies for dealing with multiple testing under general distributional assumptions can be found in, e.g., \cite{efron2004large, storey2007optimal, sun2007oracle}, where the non-nulls are assumed to follow a common distribution; the two-sided alternative in the additive error model was featured as the primary example in \cite{sun2007oracle}.
Although weighted sums of false discovery and non-discovery have been studied in the literature, asymmetric statistical risks such as \eqref{eq:risk-exact-approx} and \eqref{eq:risk-approx-exact} are new.
The high-dimensional chi-square model \eqref{eq:model-chisq} also seemed to have received little attention.
While the sparse signal detection problem in the chi-square model has been studied \cite{donoho2004higher}, support recovery problems, to the best of our knowledge, remain unexplored until now.
% In particular, \citet{sun2007oracle} studied likelihood thresholding procedure % ; their proposed likelihood ratio thresholding procedure requires consistent estimation of the non-null distributions and the mixture proportions.

Finally, asymptotic power approximations of association tests on contingency tables can be found in, e.g., \citet{ferguson2017course} or other texts on asymptotic statistics.
A companion paper to the current work analyzes asymptotic equivalences of several additional common association tests, and implements second order power approximations in a software tool \cite{gao2019upass}. 
The software also facilitates visualization and forensics of reported findings in genetic association studies.

\medskip
% \subsection{Optimality of thresholding procedures}

Although it is intuitively appealing to consider only data-thresholding procedures in multiple testing problems, such procedures are not always optimal.
Recently, \citet{chen2018scan} showed that thresholding procedures are in fact sub-optimal in the additive models \eqref{eq:model-additive} when errors have heavy (regularly-varying) tails. 
\citet{gao2020fundamental} characterized the conditions under which thresholding procedures are optimal in the exact support recovery problem.
The optimality of data-thresholding procedures in terms of other statistical risks is an open problem that invites a dedicated investigation in a future study. 
% We content ourselves with making optimality statements concerning only thresholding procedures in the current work.

\medskip

Keen readers must have noticed the curious asymmetry in Relation \eqref{eq:failure-recovery-implies-risk-1} when we discussed the relationship between the exact support recovery risk \eqref{eq:risk-exact} and the probability of exact support recovery \eqref{eq:risk-prob}.

%The converse of \eqref{eq:failure-recovery-implies-risk-1} is not true.
While a trivial procedure that never rejects and a procedure that always rejects both have $\mathrm{risk}^{\mathrm{E}}$ equal to 1, the converse is not true.
For example, it is possible that a procedure selects the true index set $S$ with probability $1/2$, but makes one false inclusion \emph{and} one false omission simultaneously the other half of the time. 
In this case the procedure will have 
$$\mathrm{risk}^{\mathrm{E}} = 1, \quad \text{and} \quad \P[\widehat{S}=S] = 1/2,$$
showing that the converse of Relation \eqref{eq:failure-recovery-implies-risk-1} is in fact false.

The same argument applies to $\mathrm{risk}^{\mathrm{A}}$:
a procedure may select the true index set $S$ with probability $1/2$, but makes enough false inclusions and omissions other half of the time, so that $\mathrm{risk}^{\mathrm{A}}$ is equal to one.
Recently, \citet{arias2017distribution} commented (in their Remark 2) that a procedure with approximate support recovery risk equal to or exceeding 1 is useless --- an opinion with which we beg to differ, in light of this artificial but legitimate example.
% --- although the class of methods with risks equal to or exceeding 1 certainly contains the trivial procedures that we mentioned.

\medskip

The current work focused only on the idealized models \eqref{eq:model-chisq} and \eqref{eq:model-additive} where statistics are \emph{independent}.
In practice, independence is the exception rather than the rule. 
Support recovery problems under dependent observations should be vigorously explored.
There have been some efforts in this direction.
In particular, the boundary for the exact support recovery problem in the additive error model \eqref{eq:model-additive} was shown to hold even under severe dependence and general distributional assumptions \cite{gao2020fundamental}.
% \cite{ji2012ups} 
We conjecture that other phase transitions also continue to hold, under classes of dependence structures that are ``not too different from independence''.
As an example, in the GWAS application, dependence among the genetic markers at different locations (known as linkage disequilibrium) decay as a function of their physical distances on the genome \cite{bush2012genome}, resulting in locally dependent test statistics.
It would be of great interest to extend the current theory to cover important dependence structures that arise in such applications.

\appendix

\section{A review of common procedures in multiple testing}
\label{sec:procedures}

% \subsection{FWER-controlling procedures}
% \label{subsec:FWER-controlling-procedures}

We recall five thresholding procedures, starting with the well-known Bonferroni's procedure which aims at controlling family-wise error rates.
\begin{definition}[Bonferroni's procedure]
Suppose the errors $\epsilon(j)$'s have a common marginal distribution $F$, Bonferroni's procedure with level $\alpha$ is the thresholding procedure that uses the threshold
\begin{equation} \label{eq:Bonferroni-procedure}
    t_p = F^{\leftarrow}(1 - \alpha/p).
\end{equation}
%where  $F^{\leftarrow}(u)=\inf{\left\{x:F(x)\ge u\right\}}$ is the generalized inverse function.
\end{definition}
% It is easy to see that the family-wise error rate (FWER) is controlled at level $\alpha$ by applying the union bound, regardless of the error-dependence structure (see e.g.\ Relation \eqref{eq:Bonferroni-FWER-control}, below).
The Bonferroni procedure is deterministic (i.e., non data-dependent), and only depends on the dimension of the problem and the null distribution.
A closely related procedure is Sid\'ak's procedure \citep{vsidak1967rectangular},
which is a more aggressive (and also deterministic) thresholding procedure that uses the 
threshold
\begin{equation} \label{eq:Sidak-procedure}
    t_p = F^{\leftarrow}((1 - \alpha)^{1/p}).
\end{equation}
% can be shown to control FWER in the case independent errors.

The third procedure, strictly more powerful than Bonferroni's, is the so-called Holm's procedure \citep{holm1979simple}.
On observing the data $x$, its coordinates can be ordered from largest to smallest
$x(i_1) \ge x(i_2)  \ge \ldots \ge x(i_p)$,
where $(i_1, \ldots, i_p)$ is a permutation of $\{1, \ldots, p\}$. 
Denote the order statistics as $x_{[1]}, x_{[2]}, \ldots, x_{[p]}$.
\begin{definition}[Holm's procedure]
Let $i^*$ be the largest index such that
$$
\overline{F}(x_{[i]})) \le \alpha / (p-i+1),\quad \text{for all }\;i\le i^*.
$$
Holm's procedure with level $\alpha$ is the thresholding procedure with threshold
\begin{equation} \label{eq:Holm-procedure}
    t_p(x) = x_{[i^*]}.
\end{equation}
\end{definition}
In contrast to the Bonferroni procedure, Holm's procedure is data-dependent.
% It can be shown that Holm's procedure also controls FWER at $\alpha$ level, regardless of dependence in the data.
A closely related, more aggressive (data-dependent) thresholding procedure is Hochberg's procedure \citep{hochberg1988sharper}
%\begin{definition}[Hochberg's procedure]
%Hochberg's procedure 
which replaces the index $i^*$ in Holm's procedure with the largest index such that
$$
\overline{F}(x_{[i]}) \le \alpha / (p-i+1).
$$
%where  $\overline{F}(x)=1-F(x)$ is the survival function.
%\end{definition}

It can be shown that Bonferroni's procedure and Holm's procedure both control FWER at their nominal levels, regardless of dependence in the data.
In contrast, Sid\'ak's procedure and Hochberg's procedure control FWER at nominal levels when data are independent.

Last but not least, we review the famed Benjamini-Hochberg (BH) procedure \cite{benjamini1995controlling}, which aims at controlling FDR.
Recall the order statistics of our observations $x_{[1]} \ge x_{[2]}  \ge \ldots \ge x_{[p]}$.

\begin{definition}[Benjamini-Hochberg's procedure]
Let $i^*$ be the largest index such that
$$
\overline{F}(x_{[i]}) \le \alpha i/p.
$$
The Benjamini-Hochberg procedure with level $\alpha$ is the thresholding procedure with threshold
\begin{equation} \label{eq:BH-procedure}
    t_p(x) = x_{[i^*]},
\end{equation}
\end{definition}
The BH procedure is shown to control the FDR at level $\alpha$ when the statistics are independent.

\section{Proofs}
\label{sec:proofs}

\subsection{Proof of Lemma \ref{lemma:risk-exact-recovery-probability}}

\begin{proof}[Proof of Lemma \ref{lemma:risk-exact-recovery-probability}]
Notice that $\{\widehat{S}=S\}$ implies $\{\widehat{S}\subseteq S\} \cap \{\widehat{S}\supseteq S\}$, therefore we have for every fixed $p$,
\begin{equation} \label{eq:risk-exact-recovery-probability-proof-1}
    \mathrm{risk}^{\mathrm{E}} 
    = 2 - \P[\widehat{S} \subseteq S] - \P[S \subseteq \widehat{S}] \\
    \le 2 - 2\P[\widehat{S}=S].
\end{equation}
On the other hand, since $\{\widehat{S}\neq S\}$ implies $\{\widehat{S}\not\subseteq S\} \cup \{\widehat{S}\not\supseteq S\}$, we have for every fixed $p$,
\begin{equation} \label{eq:risk-exact-recovery-probability-proof-2}
    1 - \P[\widehat{S}=S]
    = \P[\widehat{S}\neq S]
    \le 2 - \P[\widehat{S} \subseteq S] - \P[S \subseteq \widehat{S}]
    = \mathrm{risk}^{\mathrm{E}}. 
\end{equation}
Relation \eqref{eq:exact-recovery-implies-risk-0} follows from \eqref{eq:risk-exact-recovery-probability-proof-1} and \eqref{eq:risk-exact-recovery-probability-proof-2}, and Relation \eqref{eq:failure-recovery-implies-risk-1} from \eqref{eq:risk-exact-recovery-probability-proof-2}.
\end{proof}

\subsection{Auxiliary facts about chi-square distributions}

We first establish some auxiliary facts about chi-square distributions, used in the proofs of Theorem \ref{thm:chi-squared-exact-boundary} and Theorem \ref{thm:chi-squared-approx-boundary}.

\begin{lemma}[Rapid variation of chi-square distribution tails] \label{lemma:rapid-variation-chisq}
The central chi-square distribution with $\nu$ degrees of freedom has rapidly varying tails.
That is, 
\begin{equation} \label{eq:rapid-variation-chisq}
    \lim_{x\to\infty}\frac{\P[\chi_\nu^2(0)>tx]}{\P[\chi_\nu^2(0)>x]} = 
    \begin{cases}
    0, & t > 1 \\
    1, & t = 1 \\
    \infty, & 0 < t < 1
\end{cases},
\end{equation}
where we overloaded the notation $\chi_\nu^2(0)$ to represent a random variable with the chi-square distribution.
\end{lemma}

\begin{proof}[Proof of Lemma \ref{lemma:rapid-variation-chisq}]
When $\nu=1$, the chi-square distribution reduces to a squared Normal, and \eqref{eq:rapid-variation-chisq} follows from the rapid variation of the standard Normal distribution.
For $\nu\ge2$, we recall the following bound on tail probabilities (see, e.g., \citep{inglot2010inequalities}),
$$
\frac{1}{2}\mathcal{E}_\nu(x) \le \P[\chi_\nu^2(0)>x] \le \frac{x}{(x-\nu+2)\sqrt{\pi}} \mathcal{E}_\nu(x), \quad \nu\ge2,\;x>\nu-2,
$$
where $\mathcal{E}_\nu(x) = \exp\left\{-\frac{1}{2}[(x-\nu-(\nu-2)\log(x/\nu) + \log\nu]\right\}$.
Therefore, we have 
$$
\frac{(x-\nu+2)\sqrt{\pi}}{2x}\frac{\mathcal{E}_\nu(tx)}{\mathcal{E}_\nu(x)} 
\le \frac{\P[\chi_\nu^2(0)>tx]}{\P[\chi_\nu^2(0)>x]}
\le \frac{2tx}{(tx-\nu+2)\sqrt{\pi}}\frac{\mathcal{E}_\nu(tx)}{\mathcal{E}_\nu(x)},
$$
where ${\mathcal{E}_\nu(tx)}/{\mathcal{E}_\nu(x)} = \exp\{-\frac{1}{2}[(t-1)x-(\nu-2)\log{t}]\}$ converges to $0$ or $\infty$ depending on whether $t>1$ or $0<t<1$.
The case where $t=1$ is trivial.
\end{proof}

\begin{corollary} \label{cor:relative-stability}
Maxima of independent observations from central chi-square distributions with $\nu$ degrees of freedom are relatively stable. 
Specifically, let $\epsilon_p = \left(\epsilon_p(j)\right)_{j=1}^p$ be independently and identically distributed (iid) $\chi_\nu^2(0)$ random variables. 
Define the sequence $(u_p)_{p=1}^\infty$ to be the $(1-1/p)$-th generalized quantiles, i.e., 
\begin{equation} \label{eq:quantiles}
    u_p = F^\leftarrow(1 - 1/p),
\end{equation}
where $F$ is the central chi-square distributions with $\nu$ degrees of freedom.
The triangular array ${\cal E} = \{\epsilon_p, p\in\N\}$ has relatively stable (RS) maxima, i.e.,
\begin{equation} \label{eq:RS-condition}
    \frac{1}{u_{p}} M_p := \frac{1}{u_{p}} \max_{j=1,\ldots,p} \epsilon_p(j) \xrightarrow{\P} 1,
\end{equation}
as $p\to\infty$.
\end{corollary}

\begin{proof}[Proof of Corollary \ref{cor:relative-stability}]
When $F(x)<1$ for all finite $x$, \citet{gnedenko1943distribution} showed that the distribution $F$ has rapidly varying tails if and only if the maxima of independent observations from $F$ are relatively stable.
Rapid variation follows from Lemma \ref{lemma:rapid-variation-chisq}.
\end{proof}

\begin{lemma}[Stochastic monotonicity] \label{lemma:stochastic-monotonicity}
The non-central chi-square distribution is stochastically monotone in its non-centrality parameter.
Specifically, for two non-central chi-square distributions both with $\nu$ degrees of freedom, and non-centrality parameters $\lambda_1 \le \lambda_2$, we have $\chi^2_\nu(\lambda_1) \stackrel{\mathrm{d}}{\le} \chi^2_\nu(\lambda_2)$. 
That is,
\begin{equation} \label{eq:stochastic-monotonicity}
    \P[\chi^2_\nu(\lambda_1) \le t] \ge \P[\chi^2_\nu(\lambda_2) \le t], \quad \text{for any}\quad t\ge0.
\end{equation}
where we overloaded the notation $\chi_\nu^2(\lambda)$ to represent a random variable with the chi-square distribution with non-centrality parameter $\lambda$ and degree-of-freedom parameter $\nu$.
\end{lemma}

\begin{proof}[Proof of Lemma \ref{lemma:stochastic-monotonicity}]
Recall that non-central chi-square distributions can be written as sums of $\nu-1$ standard normal random variables and a non-central normal random variable with mean $\sqrt{\lambda}$ and variance 1,
\begin{equation*}
    \chi_\nu^2(\lambda) 
    \stackrel{\mathrm{d}}{=} Z_1^2 + \ldots + Z_{\nu-1}^2 + (Z_\nu + \sqrt{\lambda})^2.
\end{equation*}
Therefore, it suffices to show that $\P[(Z+\sqrt{\lambda})^2 \le t]$ is non-increasing in $\lambda$ for any $t\ge0$, where $Z$ is a standard normal random variable.
We rewrite this expression in terms of standard normal probability function $\Phi$,
\begin{align}
    \P[(Z+\sqrt{\lambda})^2 \le t] 
    &= \P[-\sqrt{\lambda} - \sqrt{t} \le Z \le -\sqrt{\lambda} + \sqrt{t}] \nonumber \\
    &= \Phi(-\sqrt{\lambda} + \sqrt{t}) - \Phi(-\sqrt{\lambda} - \sqrt{t}). \label{eq:stochastic-monotonicity-proof-1}
\end{align}
The derivative of the last expression (with respect to $\lambda$) is 
\begin{equation} \label{eq:stochastic-monotonicity-proof-2}
    \frac{1}{2\sqrt{\lambda}} \left(\phi(\sqrt{\lambda} + \sqrt{t}) - \phi(\sqrt{\lambda} - \sqrt{t})\right) 
    = \frac{1}{2\sqrt{\lambda}} \left(\phi(\sqrt{\lambda} + \sqrt{t}) - \phi(\sqrt{t} - \sqrt{\lambda})\right),
\end{equation}
where $\phi$ is the density of the standard normal distribution.
Notice that we have used the symmetry of $\phi$ around 0 in the last expression.

Since $0 \le \max\{\sqrt{\lambda} - \sqrt{t}, \sqrt{t} - \sqrt{\lambda}\} < \sqrt{t} + \sqrt{\lambda}$ when $t>0$, by monotonicity of the normal density on $(0,\infty)$, we conclude that the derivative \eqref{eq:stochastic-monotonicity-proof-2} is indeed negative.
Therefore, \eqref{eq:stochastic-monotonicity-proof-1} is decreasing in $\lambda$, and \eqref{eq:stochastic-monotonicity} follows for $t>0$.
For $t = 0$, equality holds in \eqref{eq:stochastic-monotonicity} with both probabilities being 0.
\end{proof}

Finally, we derive asymptotic expressions for  chi-square quantiles.

\begin{lemma}[Chi-square quantiles] \label{lemma:chisq-quantiles}
Let $F$ be the central chi-square distributions with $\nu$ degrees of freedom, and let $u(y)$ be the $(1-y)$-th generalized quantile of $F$, i.e.,
\begin{equation} \label{eq:quantiles-generic}
    u(y) = F^\leftarrow(1 - y).
\end{equation}
Then 
\begin{equation}
    u(y) \sim 2\log(1/y), \quad \text{as }y\to0. 
\end{equation}
\end{lemma}

\begin{proof}[Proof of Lemma \ref{lemma:chisq-quantiles}]
The case where $\nu=1$ follows from the well-known formula for Normal quantiles (see, e.g., Proposition 1.1 in \cite{gao2020fundamental})
$$
F^\leftarrow(1 - y) = \Phi^\leftarrow(1-y/2)\sim\sqrt{2\log{(2/y)}}\sim\sqrt{2\log{(1/y)}}.
$$
The case where $\nu\ge2$ follows from the following estimates of high quantiles of chi-square distributions (see, e.g., \citep{inglot2010inequalities}),
$$
    \nu +  2\log(1/y) -5/2 \le u(y) \le \nu +  2\log(1/y) + 2\sqrt{\nu\log(1/y)}, \quad \text{for all }y\le0.17,
$$
where both the lower and upper bound are asymptotic to $2\log(1/y)$.
\end{proof}

\subsection{Proof of Theorem \ref{thm:chi-squared-exact-boundary}}
\label{subsec:proof-chi-squared-exact-boundary}

\begin{proof}[Proof of Theorem \ref{thm:chi-squared-exact-boundary}]
We first prove the sufficient condition.
The Bonferroni procedure sets the threshold at $t_p = F^\leftarrow(1-\alpha/p)$, which, by Lemma \ref{lemma:chisq-quantiles}, is asymptotic to $2\log{p} - 2\log{\alpha}$.
By the assumption on $\alpha$ in \eqref{eq:slowly-vanishing-error}, for any $\delta>0$, we have $p^{-\delta}=o(\alpha)$.
Therefore, we have $-\log\alpha\le\delta\log{p}$ for large $p$, and
\begin{equation} \label{eq:chi-square-sufficient-0}
    1 \le \limsup_{p\to\infty}\frac{2\log{p} - 2\log{\alpha}}{2\log{p}} \le 1+\delta,
\end{equation}
for any $\delta>0$.
Hence, $t_p\sim 2\log{p}$.

The condition $\underline{r} > {{g}}(\beta)$ implies, after some algebraic manipulation,
$\sqrt{\underline{r}} -\sqrt{1-\beta} > 1$.
Therefore, we can pick $q>1$ such that 
\begin{equation} \label{eq:choice-of-q}
    \sqrt{\underline{r}} -\sqrt{1-\beta} > \sqrt{q} > 1.
\end{equation}
Setting the $t^* = t^*_p = 2q\log{p}$, we have $t_p < t^*_p$ for large $p$.

On the one hand, $\text{FWER} = 1 - \P[\widehat{S}_p \subseteq S_p]$ vanishes under the Bonferroni procedure with $\alpha\to0$.
On the other hand, for large $p$, the probability of no missed detection is bounded from below by
\begin{equation} \label{eq:chi-square-sufficient-1}
    \P[\widehat{S}_p \supseteq S_p] 
    = \P[\min_{i\in S} x(i) \ge t_p] 
    \ge \P[\min_{i\in S} x(i) \ge t^*] 
    \ge 1 - p^{1-\beta}\P[\chi_\nu^2(\underline{\Delta}) < t^*],
\end{equation}
where we have used the fact that signal sizes are bounded below by $\underline{\Delta}$, and the stochastic monotonicity of chi-square distributions (Lemma \ref{lemma:stochastic-monotonicity}) in the last inequality.
Writing
$$
\chi_\nu^2(\underline{\Delta}) \stackrel{\mathrm{d}}{=} Z_1^2 + \ldots + Z_{\nu-1}^2 + (Z_\nu + \sqrt{\underline{\Delta}})^2
$$
where $Z_i$'s are iid standard normal variables, we have
\begin{align}
    \P[\chi_\nu^2({\underline{\Delta}}) < t^*]
    &\le \P[(Z_\nu+\sqrt{\underline{\Delta}})^2 < t^*] 
    = \P[|Z_\nu+\sqrt{\underline{\Delta}}| < \sqrt{t^*}]  \nonumber \\
    &\le \P\left[Z_\nu < - \sqrt{\underline{\Delta}} +  \sqrt{t^*}\right] \nonumber \\
    &= \P\left[Z_\nu < \sqrt{2\log{p}}\left(\sqrt{q} - \sqrt{\underline{r}}\right)\right]. \label{eq:chi-square-sufficient-2}
\end{align}
By our choice of $q$ in \eqref{eq:choice-of-q}, the last probability in \eqref{eq:chi-square-sufficient-2} can be bounded from above by 
\begin{align*}
    \P\Big[Z_\nu < -\sqrt{2(1-\beta)\log{p}}\Big]
    &\sim \frac{\phi\left(-\sqrt{2(1-\beta)\log{p}}\right)}{\sqrt{2(1-\beta)\log{p}}} \\
    &= \frac{1}{\sqrt{2(1-\beta)\log{p}}}p^{-(1-\beta)},
\end{align*}
where the first line uses Mill's ratio for Gaussian distributions.
This, combined with \eqref{eq:chi-square-sufficient-1}, completes the proof of the sufficient condition for the Bonferroni's procedure.

Under the assumption of independence, Sid\'ak's, Holm's, and Hochberg's procedures are strictly more powerful than Bonferroni's procedure, while controlling FWER at the nominal levels.
Therefore, the risks of exact support recovery for these procedures also vanishes.
This completes the proof for the first part of Theorem \ref{thm:chi-squared-exact-boundary}.

We now show the necessary condition. 
We first normalize the maxima by the chi-square quantiles $u_p = F^{\leftarrow}(1-1/p)$, where $F$ is the distribution of a (central) chi-square random variable,
\begin{equation} \label{eq:chi-square-necessary-0}
 \P[\widehat{S}_p = S_p] \le \P\left[M_{S^c} <  t_p \le m_{S} \right]
  % &= \P\left[\frac{\max_{i\in S^c}x(i)}{u_p} < \frac{\min_{i\in S}x(i)}{u_p}\right] \nonumber \\
  % &\le  \P\left[\frac{\max_{i\in S^c}\chi_\nu^2(\lambda(i))}{u_p} < \frac{\min_{i\in S}\chi_\nu^2(\lambda(i))}{u_p}\right] \nonumber \\
  \le \P\left[ \frac{M_{S^c}}{u_p} < \frac{m_S}{u_p} \right],
\end{equation}
where $M_{S^c} = \max_{i\in S^c}x(i)$ and $m_{S} = \min_{i\in S}x(i)$.
By the relative stability of chi-square random variables (Corollary \ref{cor:relative-stability}), we know that ${M_{S^c}}/{u_{|S^c|}}\to1$ in probability. 
Further, using the expression for $u_p$ (Lemma \ref{lemma:chisq-quantiles}), we obtain
$$
\frac{u_{p-p^{1-\beta}}}{u_{p}} \sim \frac{2\log{(p-p^{1-\beta})}}{2\log{p}} = \frac{\log{p}+\log{(1-p^{-\beta})}}{\log{p}} \sim 1.
$$
Therefore, the left-hand-side of the last probability in \eqref{eq:chi-square-necessary-0} converges to 1,
\begin{equation} \label{eq:chi-square-necessary-1}
    \frac{M_{S^c}}{u_{p}} = \frac{M_{S^c}}{u_{p-p^{1-\beta}}} \frac{u_{p-p^{1-\beta}}}{u_{p}} \stackrel{\P}{\longrightarrow} 1.
\end{equation}

Meanwhile, for any $i\in S$, by Lemma \ref{lemma:stochastic-monotonicity} and the fact that signal sizes are bounded above by $\overline{\Delta}$, we have,
\begin{equation*}
    {\chi_\nu^2(\lambda(i))} \stackrel{\mathrm{d}}{\le}
    {\chi_\nu^2(\overline{\Delta})} \stackrel{\mathrm{d}}{=} 
    {Z_1^2 + \ldots + Z_{\nu-1}^2 + \left(Z_\nu + \sqrt{\overline{\Delta}}\right)^2}.
\end{equation*}
Dividing through by $u_p$, and taking minimum over $S$, we obtain
\begin{equation} \label{eq:chi-square-necessary-3}
    \frac{m_S}{u_p} 
    = \min_{i\in S} \frac{\chi_\nu^2(\lambda(i))}{u_p} 
    % \stackrel{\mathrm{d}}{\le} \min \left\{\frac{\chi_\nu^2(\overline{\Delta})}{u_p}, s \text{ iid copies} \right\} \\
    \stackrel{\mathrm{d}}{\le} 
    \min_{i\in S}\left\{\frac{Z_1^2(i) + \ldots + Z_{\nu-1}^2(i)}{u_p} + \frac{(Z_\nu(i) + \sqrt{\overline{\Delta}})^2}{u_p}\right\}.
\end{equation}
Let $i^\dagger = i^\dagger_p$ be the index minimizing the second term in \eqref{eq:chi-square-necessary-3}, i.e.,
\begin{equation}
    i^\dagger := \argmin_{i\in S} \frac{(Z_\nu(i) + \sqrt{\overline{\Delta}})^2}{u_p}
    = \argmin_{i\in S} f_p\left(Z_\nu(i)\right),
    % \frac{(Z_\nu(i) + \sqrt{\overline{\Delta}})^2}{2\log{p}},
\end{equation}
where $f_p(x):=(x+\sqrt{\overline{\Delta}})^2/(2\log{p})$. 
We shall first show that 
\begin{equation} \label{eq:chi-square-necessary-4}
    %\min_{i\in S} \frac{(Z_\nu(i) + \sqrt{\overline{\Delta}})^2}{2\log{p}} 
    \P[ f_p(Z_\nu(i^\dagger)) < 1 -\delta ] \to 1,
\end{equation}
for some small $\delta>0$.
On the one hand, we know (by solving a quadratic inequality) that
\begin{equation} \label{eq:chi-square-necessary-5}
    f_p(x)<1-\delta \iff \frac{x}{\sqrt{2\log{p}}} \in (-(\sqrt{\overline{r}}+\sqrt{1-\delta}), -(\sqrt{\overline{r}}-\sqrt{1-\delta})).
\end{equation}
On the other hand, we know (by relative stability of iid Gaussians \cite{gao2020fundamental}) that 
\begin{equation} \label{eq:chi-square-necessary-6}
    % f_p(\min_{i\in S}z_\nu(i)) =
    \frac{\min_{i\in S} Z_\nu(i)}{\sqrt{2\log{p}}}
    \to -\sqrt{1-\beta} \quad\text{in probability}.
\end{equation}
Further, by the assumption on the signal sizes $\overline{r} < (1+\sqrt{1-\beta})^2$, we have,
\begin{equation*}
    -(\sqrt{\overline{r}}+1) < -1 <- \sqrt{1-\beta} < - (\sqrt{\overline{r}}-1).
\end{equation*}
Therefore we can picking a small $\delta>0$ such that 
\begin{equation} \label{eq:chi-square-necessary-7}
    -(\sqrt{\overline{r}}+1) < -(\sqrt{\overline{r}}+\sqrt{1-\delta})
    < - \sqrt{1-\beta}
    < - (\sqrt{\overline{r}}-\sqrt{1-\delta})
    < - (\sqrt{\overline{r}}-1).
\end{equation}
Combining \eqref{eq:chi-square-necessary-5}, \eqref{eq:chi-square-necessary-6}, and \eqref{eq:chi-square-necessary-7}, we obtain
\begin{align*}
    \P\left[\min_{i\in S} f_p(Z_\nu(i)) < 1-\delta\right]
    &= \P\left[ f_p(Z_\nu(i^\dagger)) < 1-\delta \right] \\
    &\ge \P\left[ f_p\left(\min_{i\in S}Z_\nu(i)\right) < 1-\delta \right] \to 1,
\end{align*}
and we arrive at \eqref{eq:chi-square-necessary-4}.
As a corollary, since $u_p\sim2\log{p}$, it follows that
\begin{equation} \label{eq:chi-square-necessary-8}
    \P\left[\min_{i\in S}\frac{(Z_\nu(i) + \sqrt{\overline{\Delta}})^2}{u_p} < 1-\delta\right]\to1.
\end{equation}

Finally, by independence between $Z_1^2(i)+\ldots+Z_{\nu-1}^2(i)$ and $(Z_\nu^2(i)+\sqrt{\overline{\Delta}})^2$, and the fact that $i^\dagger$ is a function of only the latter, we have
$$
Z_1^2(i^\dagger)+\ldots+Z_{\nu-1}^2(i^\dagger) 
\stackrel{\mathrm{d}}{=} Z_1^2(i)+\ldots+Z_{\nu-1}^2(i) 
\quad \text{for all} \;\; i\in S.
$$
Therefore, $Z_1^2(i^\dagger)+\ldots+Z_{\nu-1}^2(i^\dagger) = O_\P(1)$, and 
\begin{equation} \label{eq:chi-square-necessary-9}
    \frac{Z_1^2(i^\dagger)+\ldots+Z_{\nu-1}^2(i^\dagger)}{u_p} \to 0 \quad \text{in probability}. 
\end{equation}
Together, \eqref{eq:chi-square-necessary-8} and \eqref{eq:chi-square-necessary-9} imply that
\begin{align}
    \P\left[\frac{m_S}{u_p}<1-\delta\right]
    &\ge \P\left[\min_{i\in S}\left\{\frac{Z_1^2(i) + \ldots + Z_{\nu-1}^2(i)}{u_p} + \frac{(Z_\nu(i) + \sqrt{\overline{\Delta}})^2}{u_p}\right\} < 1-\delta\right] \nonumber \\
    &\ge \P\left[\frac{Z_1^2(i^\dagger) + \ldots + Z_{\nu-1}^2(i^\dagger)}{u_p} + \frac{(Z_\nu(i^\dagger) + \sqrt{\overline{\Delta}})^2}{u_p} < 1-\delta\right] \to 1. \label{eq:chi-square-necessary-10}
\end{align}
In view of \eqref{eq:chi-square-necessary-0}, \eqref{eq:chi-square-necessary-1}, and \eqref{eq:chi-square-necessary-10}, we conclude that exact recovery cannot succeed with any positive probability.
The proof of the necessary condition is complete.
\end{proof}

\subsection{Proof of Theorem \ref{thm:chi-squared-approx-boundary}}
\label{subsec:proof-chi-squared-approx-boundary}

We first show the necessary condition. 
That is, when $\overline{r} < \beta$, no thresholding procedure is able to achieve approximate support recovery.
The main structure of the proof follows that of Theorem 1 in \citet{arias2017distribution}. 
Our arguments, however, allow for unequal signal sizes; these arguments can in turn be used to generalize the results in \cite{arias2017distribution}.

\begin{proof}[Proof of necessary condition in Theorem \ref{thm:chi-squared-approx-boundary}]
Denote the distributions of $\chi^2_\nu(0)$, $\chi^2_\nu(\underline{\Delta})$ and $\chi^2_\nu(\overline{\Delta})$ as $F_0$, $F_{\underline{a}}$, and $F_{\overline{a}}$ respectively.

% We first show the necessary condition, i.e., when $\overline{r}<\beta$, approximate support recovery cannot be achieved with any thresholding procedure.
% In particular, we show that the liminf of the sum of FDP and NDP is at least 1.

Recall that thresholding procedures are of the form
$$
\widehat{S}_p = \left\{i\,|\,x(i) > t_p(x)\right\}.
$$
Denote $\widehat{S} := \left\{i\,|\,x(i) > t_p(x)\right\}$, and $\widehat{S}(u) := \left\{i\,|\,x(i) > u\right\}$.
For any threshold $u\ge t_p$ we must have $\widehat{S}(u)\subseteq\widehat{S}$, and hence
\begin{equation} \label{eq:approx-boundary-proof-FDP}
    \text{FDP} := \frac{|\widehat{S}\setminus{S}|}{|\widehat{S}|} \ge \frac{|\widehat{S}\setminus{S}|}{|\widehat{S}\cup{S}|} = \frac{|\widehat{S}\setminus{S}|}{|\widehat{S}\setminus{S}| + |S|} \ge
    \frac{|\widehat{S}(u)\setminus{S}|}{|\widehat{S}(u)\setminus{S}| + |S|}.
\end{equation}
On the other hand, for any threshold $u\le t_p$ we must have $\widehat{S}(u)\supseteq\widehat{S}$, and hence
\begin{equation} \label{eq:approx-boundary-proof-NDP}
    \text{NDP} := \frac{|{S}\setminus\widehat{S}|}{|{S}|} \ge 
    \frac{|{S}\setminus\widehat{S}(u)|}{|{S}|}.
\end{equation}
Since either $u\ge t_p$ or  $u\le t_p$ must take place, putting \eqref{eq:approx-boundary-proof-FDP} and \eqref{eq:approx-boundary-proof-NDP} together, we have
\begin{equation} \label{eq:approx-boundary-proof-converse-1}
    \text{FDP} + \text{NDP} 
    \ge \frac{|\widehat{S}(u)\setminus{S}|}{|\widehat{S}(u)\setminus{S}|+|{S}|} \wedge \frac{|{S}\setminus\widehat{S}(u)|}{|{S}|},
\end{equation}
for any $u$.
Therefore it suffices to show that for a suitable choice of $u$, the RHS of \eqref{eq:approx-boundary-proof-converse-1} converges to 1 in probability; the desired conclusion on FDR and FNR follows by the dominated convergence theorem.

Let $t^* = 2q\log{p}$ for some fixed $q$, we obtain an estimate of the tail probability
\begin{align}
    \overline{F_0}(t^*) 
    &= \P[\chi_\nu^2(0) > t^*] 
    = \frac{2^{-\nu/2}}{\Gamma(\nu/2)} \int_{2q\log{p}}^\infty x^{\nu/2-1}e^{-x/2} \mathrm{d}x \nonumber \\
    &\sim \frac{2^{-\nu/2}}{\Gamma(\nu/2)} 2\left(2q\log{p}\right)^{\nu/2-1}p^{-q}. \label{eq:approx-boundary-proof-null-tail-prob}
\end{align}
where $a_p\sim b_p$ is taken to mean $a_p/b_p\to 1$; this tail estimate was also obtained in \cite{donoho2004higher}.
Observe that $|\widehat{S}(t^*)\setminus{S}|$ has distribution $\text{Binom}(p-s, \overline{F_0}(t^*))$ where $s=|S|$, denote $X = X_p := {|\widehat{S}(t^*)\setminus{S}|}/{|S|}$, and we have 
$$
\mu := \E\left[X\right] = \frac{(p-s)\overline{F_0}(t^*)}{s},
\quad \text{and} \quad
\var\left(X\right) = \frac{(p-s)\overline{F_0}(t^*){F_0}(t^*)}{s^2} \le \mu/s.
$$
Therefore for any $M>0$, we have, by Chebyshev's inequality,
\begin{equation}
    \P\left[X < M\right] 
    \le \P\left[\left|X-\mu\right| > \mu - M\right]
    \le \frac{\mu/s}{(\mu-M)^2}
    = \frac{1/(\mu s)}{(1-M/\mu)^2}. \label{eq:approx-boundary-proof-converse-2}
\end{equation}
Now, from the expression of $\overline{F_0}(t^*)$ in \eqref{eq:approx-boundary-proof-null-tail-prob}, we obtain
$$
\mu = (p^\beta - 1)\overline{F_0}(t^*) \sim \frac{2^{1-\nu/2}}{\Gamma(\nu/2)} \left(2q\log{p}\right)^{\nu/2-1}p^{\beta-q}.
$$
Since $\overline{r}<\beta$, we can pick $q$ such that $\overline{r}<q<\beta$. 
In turn, we have $\mu \to\infty$, as $p\to\infty$.
Therefore the last expression in \eqref{eq:approx-boundary-proof-converse-2} converges to 0, and we conclude that $X\to\infty$ in probability, and hence
\begin{equation} \label{eq:approx-boundary-proof-converse-3}
\frac{|\widehat{S}(t^*)\setminus{S}|}{|\widehat{S}(t^*)\setminus{S}|+|{S}|} 
= \frac{X}{X+1} \to 1 \quad \text{in probability}.
\end{equation}

On the other hand, we show that with the same choice of $u = t^*$,
\begin{equation} \label{eq:approx-boundary-proof-converse-4}
    \frac{|{S}\setminus\widehat{S}(t^*)|}{|{S}|}\to 1 \quad \text{in probability}.
\end{equation}
By the stochastic monotonicity of chi-square distributions (Lemma \ref{lemma:stochastic-monotonicity}), the probability of missed detection for each signal is lower bounded by $\P[\chi^2_\nu(\lambda_i) \le t^*] \ge F_{\overline{a}}(t^*)$.
Therefore, $|{S}\setminus\widehat{S}(t^*)| \stackrel{\mathrm{d}}{\ge} \text{Binom}(s, {F_{\overline{a}}}(t^*))$, and it suffices to show that ${F_{\overline{a}}}(t^*)$ converges to 1.
This is indeed the case, since
\begin{align*}
    {F_{\overline{a}}}(t^*) 
    &= \P[Z_1^2 + \ldots + Z_\nu^2 + 2\sqrt{2\overline{r}\log{p}} Z_\nu + 2\overline{r}\log{p} \le 2q\log{p}] \\
    &\ge \P[Z_1^2 + \ldots + Z_\nu^2 \le (q-\overline{r})\log{p}, \; 2\sqrt{2\overline{r}\log{p}} Z_\nu \le (q-\overline{r})\log{p}],
\end{align*}
and both events in the last line have probability going to 1 as $p\to\infty$.
The necessary condition is shown.
\end{proof}

We now turn to the sufficient condition. 
That is, when $\underline{r} > \beta$, the Benjamini-Hochberg procedure with slowly vanishing FDR levels achieves asymptotic approximate support recovery.

The proof proceeds in two steps.
We first make the connection between power of the BH procedure and stochastic ordering of the alternatives.
It is formalized in the following lemma, which could be of independent interest.
This result, though natural, seems new. 

\begin{lemma}[Monotonicity of the BH procedure] \label{lemma:monotonicity-BH-procedure}
Compare Alternatives 1 and 2 where we have $p$ independent observations $x(i)$, $i\in\{1,\ldots,p\}$,
\begin{enumerate}
    \item[Alt.1] the $p-s$ coordinates in the null part have common distribution $F_0$, and the $s$ signals have alternative distributions $F^{i}_1$, $i\in S$, respectively.
    \item[Alt.2] the $p-s$ coordinates in the null part have common distribution $F_0$, and the $s$ signals have alternative distributions $F^{i}_2$, $i\in S$, respectively, where
    $$ F^{i}_2(t) \le F^{i}_1(t), \quad \text{for all} \;\; t\in\R, \; \text{and for all} \;\; i\in S.$$
\end{enumerate}
If we apply the BH procedure at the same nominal level of FDR $\alpha$, then the FNR under Alternative 2 is bounded above by the FNR under Alternative 1.
\end{lemma}

Loosely put, the power of the BH procedure is monotone increasing with respect to the stochastic ordering of the alternatives.

\begin{proof}[Proof of Lemma \ref{lemma:monotonicity-BH-procedure}]
We first re-express the BH procedure in a different form.
Recall that on observing $x(i)$, $i\in\{1,\ldots,p\}$, the BH procedure is the thresholding procedure with threshold set at $x_{[i^*]}$, where $i^* := \max\{i\,|\,\overline{F_0}(x_{[i]})\le \alpha i/p\}$, and $x_{[1]}\ge\ldots\ge x_{[p]}$ are the order statistics.

Let $\widehat{G}$ denote the empirical survival function
\begin{equation} \label{eq:empirical-tail-distribution}
    \widehat{G}(t) = \frac{1}{p}\sum_{i\in[p]}\mathbbm{1}\{x(i) \ge t\}.
\end{equation}
By the definition, we know that $\widehat{G}(x_{[i]}) = i/p$.
Therefore, by the definition of $i^*$, we have
\begin{equation*} 
    \overline{F_0}(x_{[i]}) > \alpha\widehat{G}(x_{[i]}) = \alpha i/p \quad \text{for all }i>i^*.
\end{equation*}
Since $\widehat{G}$ is constant on $(x_{[i^*+1]}, x_{[i^*]}]$, the fact that 
$\overline{F_0}(x_{[i^*]}) \le \alpha\widehat{G}(x_{[i^*]})$ and $\overline{F_0}(x_{[i^*+1]}) > \alpha\widehat{G}(x_{[i^*+1]})$ implies that $\alpha\widehat{G}$ and $\overline{F_0}$ must ``intersect'' on the interval by continuity of $F_0$.
We denote this ``intersection'' as
\begin{equation} \label{eq:approx-boundary-proof-tau}
    \tau = \inf\{t\,|\,\overline{F_0}(t)\le\alpha\widehat{G}(t)\}. 
    %= \min\{t\,|\,\overline{F_0}(t)=\alpha\widehat{G}(t)\}.
\end{equation}
Note that $\tau$ cannot be equal to $x_{[i^*+1]}$ since $\overline{F}_0$ is c\`adl\`ag.
Since there is no observation in $[\tau, x_{[i^*]})$, we can write the BH procedure as the thresholding procedure with threshold set at $\tau$.

Now, denote the observations under Alternatives 1 and 2 as $x_1(i)$ and $x_2(i)$.
Since $x_2(i)$ stochastically dominates $x_1(i)$ for all $i\in\{1,\ldots,p\}$, there exists a coupling $(\widetilde{x}_1, \widetilde{x}_2)$ of $x_1$ and $x_2$ such that 
% $\widetilde{x}_1(i) = \widetilde{x}_2(i)$ for $i\in S^c$, and 
$\widetilde{x}_1(i) \le \widetilde{x}_2(i)$ a.s. for all $i$.
We will replace $\widetilde{x}_1$ and $\widetilde{x}_2$ with $x_1$ and $x_2$ in what follows.
Since we will compare the FNR's, i.e., expectations with respect to the marginals of ${x}$'s in the last step, this replacement does not affect the conclusions.
To simplify notation, we still write $x_1$ and $x_2$ in place of $\widetilde{x}_1$ and $\widetilde{x}_2$.

Let $\widehat{G}_k$ be the empirical survival function under Alternative $k$, i.e.,
\begin{equation} \label{eq:empirical-survival}
    \widehat{G}_k(t) = \frac{1}{p}\sum_{i\in[p]}\mathbbm{1}\{x_k(i) \ge t\}, \quad k\in\{1,2\}.
\end{equation}
We define the BH thresholds $\tau_1$ and $\tau_2$ by replacing $\widehat{G}$ in \eqref{eq:approx-boundary-proof-tau} with $\widehat{G}_1$ and $\widehat{G}_2$, respectively.
Denote the set estimates of signal support $\widehat{S}_k = \{i\,|\,x_k(i)\ge\tau_k\}$ by the BH procedure.
We claim that
\begin{equation} \label{eq:monotonicity-BH-procedure-thresholds}
    \tau_2 \le \tau_1 \quad \text{with probability } 1.
\end{equation}

Indeed, by definition of the empirical survival function \eqref{eq:empirical-survival} and the fact that $x_1(i) \le x_2(i)$ almost surely for all $i$,  we have $\widehat{G}_1(t) \le \widehat{G}_2(t)$ for all $t$.
Hence, $\overline{F_0}(t)\le\alpha\widehat{G}_1(t)$ implies $\overline{F_0}(t)\le\alpha\widehat{G}_2(t)$, and Relation \eqref{eq:monotonicity-BH-procedure-thresholds} follows from the definition of $\tau$ in \eqref{eq:approx-boundary-proof-tau}.

Finally, when $\tau_2 \le \tau_1$, we have $\tau_2 \le \tau_1 \le x_1(i) \le x_2(i)$ with probability 1 for all $i\in\widehat{S}_1$.
Therefore, it follows that $\widehat{S}_1 \subseteq \widehat{S}_2$ and hence $|S\setminus\widehat{S}_2| \le |S\setminus\widehat{S}_1|$ almost surely. 
The conclusion in Lemma \ref{lemma:monotonicity-BH-procedure} follows from the last inequality.
\end{proof}

Lemma \ref{lemma:monotonicity-BH-procedure} allows us to reduce the analysis of alternatives with unequal signal sizes to alternatives with equal signal sizes. 
The structure for the rest of the proof for the sufficient condition follows that of Theorem 2 in \cite{arias2017distribution}. 

\begin{proof}[Proof of sufficient condition in Theorem \ref{thm:chi-squared-approx-boundary}]
The FDR vanishes by our choice of $\alpha$ and the FDR-controlling property of the BH procedure.
It only remains to show that FNR also vanishes.

To do so we compare the FNR under the alternative specified in Theorem \ref{thm:chi-squared-approx-boundary} to one with all of the signal sizes equal to $\underline{\Delta}$.
Let $x(i)$ be vectors of independent observations with $p-s$ nulls having $\chi^2_\nu(0)$ distributions, and $s$ signals having $\chi^2_\nu(\underline{\Delta})$ distributions.
By Lemma \ref{lemma:monotonicity-BH-procedure}, it suffices to show that the FNR under the BH procedure in this setting vanishes.

Let $\widehat{G}$ denote the empirical survival function as in \eqref{eq:empirical-tail-distribution}.
Define the empirical survival functions for the null part and signal part
\begin{equation} \label{eq:empirical-survival-null-signal}
    \widehat{W}_\text{null}(t) = \frac{1}{p-s}\sum_{i\not\in S}\mathbbm{1}\{x(i) \ge t\},
    \quad
    \widehat{W}_\text{signal}(t) = \frac{1}{s}\sum_{i\in S}\mathbbm{1}\{x(i) \ge t\},
\end{equation}
where $s=|S|$, so that
$$
\widehat{G}(t) = \frac{p-s}{p}\widehat{W}_\text{null}(t) + \frac{s}{p}\widehat{W}_\text{signal}(t).
$$

We need the following result to describe the deviations of the empirical distributions.
\begin{lemma}[Theorem 1 of \citet{eicker1979asymptotic}] \label{lemma:empirical-process}
Let $Z_1,\ldots,Z_k$ be iid with continuous survival function $Q$.
Let $\widehat{Q}_k$ denote their empirical survival function and define 
$\xi_k = \sqrt{2\log{\log{(k)}}/k}$ for $k \ge 3$. 
Then
$$
\frac{1}{\xi_k}\sup_z\frac{|\widehat{Q}_k(z) - Q(z)|}{\sqrt{Q(z)(1 - Q(z))}} \to 1,
$$
in probability as $k \to \infty$.
In particular,
$$
\widehat{Q}_k(z) = Q(z) + O_\P\left(\xi_k\sqrt{Q(z)(1 - Q(z))}\right),
$$
uniformly in z.
\end{lemma}

Apply Lemma \ref{lemma:empirical-process} to $\widehat{G}$, we obtain
$\widehat{G}(t) = G(t) + \widehat{R}(t)$.
where 
\begin{equation} \label{eq:empirical-process-mean}
    G(t) = \frac{p-s}{p}\overline{F_0}(t) + \frac{s}{p}\overline{F_a}(t),
\end{equation}
where $\overline{F_0}$ and $\overline{F_{a}}$ are the survival functions of $\chi_\nu^2(0)$ and $\chi_\nu^2(\underline{\Delta})$ respectively, and 
\begin{equation} \label{eq:empirical-process-residual}
    \widehat{R}(t) = O_\P\left(\xi_p\sqrt{\overline{F_0}(t)F_0(t)} + \frac{s}{p}\xi_s\sqrt{\overline{F_a}(t)F_a(t)}\right),
\end{equation}
uniformly in $t$.

Recall (see proof of Lemma \ref{lemma:monotonicity-BH-procedure}) that the BH procedure is the thresholding procedure with threshold set at $\tau$ (defined in \eqref{eq:approx-boundary-proof-tau}).
% \begin{equation} \label{eq:approx-boundary-proof-tau-recall}
%     \tau = \inf\{t\,|\,\overline{F_0}(t)\le\alpha\widehat{G}(t)\}. 
%     %= \min\{t\,|\,\overline{F_0}(t)=\alpha\widehat{G}(t)\}.
% \end{equation}
The NDP may also be re-written as 
$$
\text{NDP} = \frac{|{S}\setminus\widehat{S}|}{|{S}|} = \frac{1}{s}\sum_{i\in S}\mathbbm{1}\{x(i) < \tau\} = 1 - \widehat{W}_\text{signal}(\tau),
$$
so that it suffices to show that 
\begin{equation} \label{eq:approx-boundary-proof-sufficient-1}
    \widehat{W}_\text{signal}(\tau)\to 1
\end{equation} in probability.
Applying Lemma \ref{lemma:empirical-process} to $\widehat{W}_\text{signal}$, we know that 
$$
\widehat{W}_\text{signal}(\tau) = \overline{F_a}(\tau) + O_\P\left(\xi_s\sqrt{\overline{F_a}(\tau)F_a(\tau)}\right) = \overline{F_a}(\tau) + o_\P(1).
$$
So it suffices to show that $F_a(\tau)\to 0$ in probability.
Now let $t^* = 2q\log(p)$ for some $q$ such that $\beta<q<\underline{r}$.
We have 
\begin{align}
    F_a(t^*) 
    &= \P[\chi^2_\nu(\underline{\Delta}) \le t^*]
    \le \P\left[2\sqrt{\underline{\Delta}}Z_\nu \le t^* - \underline{\Delta}\right] \nonumber \\
    &= \P\left[Z_\nu \le \frac{t^*}{2\sqrt{\underline{\Delta}}} - \frac{\sqrt{\underline{\Delta}}}{2}\right] 
    = \P\left[Z_\nu \le \frac{q-\underline{r}}{2\sqrt{\underline{r}}}\sqrt{2\log{p}}\right] \to 0. \label{eq:approx-boundary-proof-sufficient-2}
\end{align} 
Hence in order to show \eqref{eq:approx-boundary-proof-sufficient-1}, it suffices to show 
\begin{equation} \label{eq:approx-boundary-proof-sufficient-3}
    \P\left[\tau \le t^*\right] \to 1.
\end{equation}
By \eqref{eq:empirical-process-mean}, the mean of the empirical process $\widehat{G}$ evaluated at $t^*$ is
\begin{equation} \label{eq:approx-boundary-proof-sufficient-4}
    G(t^*) = \frac{p-s}{p}\overline{F_0}(t^*) + \frac{s}{p}\overline{F_a}(t^*).
\end{equation}
The first term, using Relation \eqref{eq:approx-boundary-proof-null-tail-prob}, is asymptotic to $p^{-q}L(p)$, where $L(p)$ is the logarithmic term in $p$.
The second term, since $\overline{F_a}(t^*)\to 1$ by Relation \eqref{eq:approx-boundary-proof-sufficient-2}, is asymptotic to $p^{-\beta}$.
Therefore, $G(t^*) \sim p^{-q}L(p) + p^{-\beta} \sim p^{-\beta}$, since 
$p^{\beta-q}L(p)\to0$ where $q>\beta$.

The fluctuation of the empirical process at $t^*$, by Relation \eqref{eq:empirical-process-residual}, is 
\begin{align*}
    \widehat{R}(t^*) 
    &= O_\P\left(\xi_p\sqrt{\overline{F_0}(t^*)F_0(t^*)} + \frac{s}{p}\xi_s\sqrt{\overline{F_a}(t^*)F_a(t^*)}\right)\\
    &= O_\P\left(\xi_p\sqrt{\overline{F_0}(t^*)}\right) + o_\P\left(p^{-\beta}\right).
\end{align*}
By \eqref{eq:approx-boundary-proof-null-tail-prob} and the expression for $\xi_p$, the first term is $O_\P\left(p^{-(q+1)/2}L(p)\right)$ where $L(p)$ is a poly-logarithmic term in $p$.
Since $\beta<\min\{q,1\}$, we have $\beta<(q+1)/2$, and hence $\widehat{R}(t^*) = o_\P(p^{-\beta})$.

Putting the mean and the fluctuation of $\widehat{G}(t^*)$ together, we obtain
$$
\widehat{G}(t^*) = G(t^*) + \widehat{R}(t^*) \sim_\P G(t^*) \sim p^{-\beta},
$$
and therefore, together with \eqref{eq:approx-boundary-proof-null-tail-prob}, we have
$$
\overline{F_0}(t^*)/\widehat{G}(t^*) = p^{\beta-q}L(p)(1+o_{\P}(1)),
$$
which is eventually smaller than the FDR level $\alpha$ by the assumption \eqref{eq:slowly-vanishing-error} and the fact that $\beta<q$.
That is, 
$$
\P\left[\overline{F}_0(t^*) / \widehat{G}(t^*) < \alpha\right] \to 1.
$$
By definition of $\tau$ (recall \eqref{eq:approx-boundary-proof-tau}), this implies that $\tau \le t^*$ with probability tending to 1, and \eqref{eq:approx-boundary-proof-sufficient-3} is shown.
The proof for the sufficient condition is complete.
\end{proof}

\subsection{Proof of Theorems \ref{thm:chi-squared-exact-approx-boundary} and \ref{thm:chi-squared-approx-exact-boundary}}
\label{subsec:proof-chi-squared-mix-boundaries}

Proof of Theorem \ref{thm:chi-squared-exact-approx-boundary} uses ideas from both the proof of Theorem \ref{thm:chi-squared-exact-boundary} and that of Theorem \ref{thm:chi-squared-approx-boundary}, and is substantially shorter.

\begin{proof}[Proof of Theorem \ref{thm:chi-squared-exact-approx-boundary}]
% The reduction to equal signal sizes can be achieved 
We first show the sufficient condition.
Vanishing FWER is guaranteed by the properties of the procedures, as discussed in the proof  of Theorem \ref{thm:chi-squared-exact-boundary}, and we only need to show that FNR also goes to zero. 
As in the proof of Theorem \ref{thm:chi-squared-exact-approx-boundary}, it suffices to show that
\begin{equation} \label{eq:exact-approx-boundary-proof-sufficient-1}
    \text{NDP} = 1 - \widehat{W}_\text{signal}(t_p) \to 0,
\end{equation}
where $t_p$ is the threshold of Bonferroni's procedure.

Since $\underline{r}>\widetilde{g}(\beta)=1$, we can pick $q$ such that $1<q<\underline{r}$.
Let $t^* = 2q\log{p}$, we have $t_p<t_p^*$ for large $p$ as in the proof of Theorem \ref{thm:chi-squared-exact-boundary}.
Therefore for large $p$, we have
$$
\widehat{W}_\text{signal}(t_p) \ge \widehat{W}_\text{signal}(t^*) \ge \overline{F_a}(t^*) + o_\P(1),
$$
where the last inequality follows from the stochastic monotonicity of the chi-square family, and Lemma \ref{lemma:empirical-process}.
Indeed, $F_a(t^*)\to0$ by \eqref{eq:approx-boundary-proof-sufficient-2} and our choice of $q<\underline{r}$. 
The proof of the sufficient condition is complete.

Proof of the necessary condition follows a similar structure to that of Theorem \ref{thm:chi-squared-approx-boundary}.
That is, we show that $\mathrm{FWER} + \mathrm{FNR}$ has liminf at least 1 by working with the lower bound
\begin{equation} \label{eq:exact-approx-boundary-proof-necessary-1}
    \mathrm{FWER}(\mathcal{R}) + \mathrm{FNR}(\mathcal{R}) \ge \P\left[\max_{i\in S^c}x(i)>u\right] \wedge \E\left[\frac{|S\setminus \widehat{S}(u)|}{|S|}\right],
\end{equation}
which holds for any thresholding procedure $\mathcal{R}$ and for arbitrary $u\in\R$.
By the assumption that $\overline{r}<\widetilde{g}(\beta)=1$, we can pick $q$ such that $\overline{r}<q<1$ and let $u = t^*=2q\log{p}$.
By relative stability of chi-squared random variables (Lemma \ref{lemma:rapid-variation-chisq}), we have
\begin{equation} \label{eq:exact-approx-boundary-proof-necessary-2}
    \P\left[\frac{\max_{i\in S^c} x(i)}{2\log{p}} > \frac{t^*}{2\log{p}}\right] \to 1.
\end{equation}
where the first fraction in \eqref{eq:exact-approx-boundary-proof-necessary-2} converges to 1, while the second converges to $q<1$.
On the other hand, by our choice of $q>\overline{r}$, the second term in \eqref{eq:exact-approx-boundary-proof-necessary-1} also converges to 1 as in \eqref{eq:approx-boundary-proof-converse-4}.
This completes the proof of the necessary condition.
\end{proof}

\begin{proof}[Proof of Theorem \ref{thm:chi-squared-approx-exact-boundary}]
We first show the sufficient condition.
Since FDR control is guaranteed by the BH procedure, we only need to show that the FWNR also vanishes, that is,
\begin{equation} \label{eq:approx-exact-boundary-proof-sufficient-1}
    \P\left[\min_{i\in S}x(i) \ge \tau\right] \to 1,
\end{equation}
where $\tau$ is the threshold for the BH procedure.

By the assumption that $\underline{r}>\widetilde{h}(\beta)=(\sqrt{\beta}+\sqrt{1-\beta})^2$, we have $\sqrt{\underline{r}}-\sqrt{1-\beta}>\sqrt{\beta}$, so we can pick $q>0$, such that 
\begin{equation} \label{eq:approx-exact-boundary-proof-sufficient-2}
\sqrt{\underline{r}}-\sqrt{1-\beta}>\sqrt{q}>\sqrt{\beta}.
\end{equation}
Let $t^*=2q\log{p}$, we claim that 
\begin{equation} \label{eq:approx-exact-boundary-proof-sufficient-3}
\P\left[\tau\le t^*\right]\to 1.
\end{equation}
Indeed, by our choice of $q>\beta$, \eqref{eq:approx-exact-boundary-proof-sufficient-3} follows in the same way that \eqref{eq:approx-boundary-proof-sufficient-3} did.

With this $t^*$, we have
\begin{equation} \label{eq:approx-exact-boundary-proof-sufficient-4}
    \P\left[\min_{i\in S}x(i) \ge \tau\right] \ge 
    % \P\left[\min_{i\in S}x(i) \ge t^* \ge \tau\right] \ge
    \P\left[\min_{i\in S}x(i) \ge t^*,\; t^* \ge \tau\right].
\end{equation}
However, by our choice of $\sqrt{q} < \sqrt{\underline{r}}-\sqrt{1-\beta}$, the probability of the first event on the right-hand side of \eqref{eq:approx-exact-boundary-proof-sufficient-4} also goes to 1 according to \eqref{eq:chi-square-sufficient-1} and \eqref{eq:chi-square-sufficient-2}.
Together with \eqref{eq:approx-exact-boundary-proof-sufficient-3}, this proves \eqref{eq:approx-exact-boundary-proof-sufficient-1}, and completes proof of the sufficient condition.

The necessary condition follows from the lower bound
\begin{equation} \label{eq:approx-exact-boundary-proof-necessary-1}
    \mathrm{FDR}(\mathcal{R}) + \mathrm{FWNR}(\mathcal{R}) \ge \E\left[\frac{|\widehat{S}(u)\setminus S|}{|\widehat{S}(u)\setminus S| + |S|}\right] \wedge 
    \P\left[\min_{i\in S}x(i)<u\right],
\end{equation}
which holds for any thresholding procedure $\mathcal{R}$ and for arbitrary $u\in\R$.

By the assumption that $\overline{r}<\widetilde{h}(\beta)=(\sqrt{\beta}+\sqrt{1-\beta})^2$, we can pick a constant $q>0$, such that 
\begin{equation} \label{eq:approx-exact-boundary-proof-necessary-2}
    \sqrt{\overline{r}} - \sqrt{1-\beta} < \sqrt{q} < \sqrt{\beta}.
\end{equation}
Let also $u = t^*=2q\log{p}$.
By our choice of $q < \beta$, we know from \eqref{eq:approx-boundary-proof-converse-3} that the first term on the right-hand-side of \eqref{eq:approx-exact-boundary-proof-necessary-1} converges to 1.
It remains to show that the second term in \eqref{eq:approx-exact-boundary-proof-necessary-1} also converges to 1.

For the second term in \eqref{eq:approx-exact-boundary-proof-necessary-1}, dividing through by $2\log{p}$, we obtain
\begin{equation} \label{eq:approx-exact-boundary-proof-necessary-3}
    \P\left[\min_{i\in S}x(i)<t^*\right] = \P\left[ \frac{m_{S}}{2\log{p}} < q \right].
\end{equation}
Similar to \eqref{eq:chi-square-necessary-3}, we have
\begin{equation} \label{eq:approx-exact-boundary-proof-necessary-4}
    \frac{m_{S}}{2\log{p}} 
    \stackrel{\mathrm{d}}{\le} \min_{i\in S}\frac{Z_1^2(i) + \ldots + Z_{\nu-1}^2(i)}{2\log{p}} + \frac{(Z_\nu(i) + \sqrt{\overline{\Delta}})^2}{2\log{p}}.
\end{equation}
Define $i^\dagger = i^\dagger_p$ to be the index minimizing the second term in \eqref{eq:approx-exact-boundary-proof-necessary-4}, i.e.,
\begin{equation}
    i^\dagger := \argmin_{i\in S} f_p\left(Z_\nu(i)\right),
    % \frac{(Z_\nu(i) + \sqrt{\overline{\Delta}})^2}{2\log{p}},
\end{equation}
where $f_p(x):=(x+\sqrt{\overline{\Delta}})^2/(2\log{p})$. 

Since $\sqrt{q}>\sqrt{\overline{r}}-\sqrt{1-\beta}$ and $q>0$, we have
$\frac{\sqrt{\overline{r}}-\sqrt{q}}{\sqrt{1-\beta}}<1$.
Also, since
$$
    \frac{\sqrt{\overline{r}}+\sqrt{q}}{\sqrt{1-\beta}}>0,
    \quad\text{and}\quad
    \frac{\sqrt{\overline{r}}-\sqrt{q}}{\sqrt{1-\beta}} < \frac{\sqrt{\overline{r}}+\sqrt{q}}{\sqrt{1-\beta}},
$$
we can further pick a constant $\beta_0\in(0,1]$ such that
\begin{equation} \label{eq:approx-exact-boundary-proof-necessary-5}
    \frac{\sqrt{\overline{r}}-\sqrt{q}}{\sqrt{1-\beta}} 
    < \sqrt{\beta_0} < 
    \frac{\sqrt{\overline{r}}+\sqrt{q}}{\sqrt{1-\beta}}.
\end{equation}
Let $Z_{[1]}\le Z_{[2]}\le\ldots\le Z_{[s]}$ be the order statistics of 
$\{Z_\nu(i)\}_{i\in S}$ and define $k=\lfloor s^{1-\beta_0}\rfloor$.
Applying Lemma \ref{lemma:relative-stability-order-statistics} (stated below), we obtain
% \begin{equation*} 
% Z_{[k]} \sim -\sqrt{2\beta_0\log{s}} \sim -\sqrt{2\beta_0(1-\beta)\log{p}}.
% % , \quad\text{as}\;\;p\to\infty.
% \end{equation*}
% Therefore, we have,
\begin{equation} \label{eq:approx-exact-boundary-proof-necessary-6}
    % f_p(\min_{i\in S}z_\nu(i)) =
    \frac{Z_{[k]}}{\sqrt{2\log{p}}}
    = \frac{Z_{[k]}}{\sqrt{2\log{s}}} \frac{\sqrt{2\log{s}}}{\sqrt{2\log{p}}} 
    \to -\sqrt{\beta_0(1-\beta)}
    % \left(-(\sqrt{\overline{r}}+\sqrt{q}), -(\sqrt{\overline{r}}-\sqrt{q})\right) 
    \quad\text{in probability}.
\end{equation}
Since we know (by solving a quadratic inequality) that
\begin{equation} \label{eq:approx-exact-boundary-proof-necessary-7}
    f_p(x)<q \iff \frac{x}{\sqrt{2\log{p}}} \in \left(-(\sqrt{\overline{r}}+\sqrt{q}), -(\sqrt{\overline{r}}-\sqrt{q})\right),
\end{equation}
combining \eqref{eq:approx-exact-boundary-proof-necessary-5}, \eqref{eq:approx-exact-boundary-proof-necessary-6}, and 
\eqref{eq:approx-exact-boundary-proof-necessary-7}, it follows that
\begin{equation*}
    %\min_{i\in S} \frac{(Z_\nu(i) + \sqrt{\overline{\Delta}})^2}{2\log{p}} 
    \P\left[ f_p\left(Z_\nu(i^\dagger)\right) < q \right] \ge \P\left[ f_p\left(Z_{[k]}\right) < q \right] \to 1.
\end{equation*}
Finally, using \eqref{eq:chi-square-necessary-9}, we conclude that 
$$
\P\left[\min_{i\in S}x(i)<t^*\right] = 
\P\left[\frac{m_{S}}{2\log{p}} < q\right] \ge
\P\left[o_\P(1) + f_p\left(Z_\nu(i^\dagger)\right) < q \right] \to 1.
$$
Therefore, the two terms on the right-hand-side of \eqref{eq:approx-exact-boundary-proof-necessary-1} both converge 1. 
This completes the proof of the necessary condition.
\end{proof}

It only remains to justify \eqref{eq:approx-exact-boundary-proof-necessary-6}.

\begin{lemma}[Relative stability of order statistics]
\label{lemma:relative-stability-order-statistics}
Let $Z_{[1]} \le \ldots \le Z_{[s]}$ be the order statistics of $s$ iid standard Gaussian random variables.
Let $\beta_0\in(0,1]$ and define $k=\lfloor s^{1-\beta_0}\rfloor$, then we have
\begin{equation}
    \frac{Z_{[k]}}{\sqrt{2\log{s}}} \to -\sqrt{\beta_0} \quad\text{in probability}.
\end{equation}
\end{lemma}

\begin{proof}[Proof of Lemma \ref{lemma:relative-stability-order-statistics}]
Using the Renyi representation for order statistics, we write
\begin{equation} \label{eq:relative-stability-order-statistics-proof-1}
    Z_{[i]} = \Phi^{\leftarrow}(U_{[i]}),
\end{equation}
where $U_{[i]}$ is the $i^\mathrm{th}$ (smallest) order statistic of $s$ independent uniform random variables over $(0,1)$.
Since $U_{[i]}$ has a $\mathrm{Beta}(i, s+1-i)$ distribution, with mean and standard deviation,
$$
\E[U_{[k]}] = k/(s+1) \sim s^{-\beta_0}, 
\quad \text{and} \quad
{\mathrm{sd}(U_{[k]})} = \frac{1}{s+1}\sqrt{\frac{k(s+1-k)}{s+2}} \sim s^{-\frac{1+\beta_0}{2}},
$$
we obtain by Chebyshev's inequality 
% (since $\E[U_{[k]}]/\mathrm{sd}(U_{[k]})\sim p^{(1-\beta_0)/2}$) 
$$
\P\left[s^{-\beta_0}(1-\epsilon) < U_{[k]} < s^{-\beta_0}(1+\epsilon)\right] \to 1,
$$
where $\epsilon$ is an arbitrary positive constant.
This implies, by representation \eqref{eq:relative-stability-order-statistics-proof-1},
\begin{equation} \label{eq:relative-stability-order-statistics-proof-2}
    \P\left[\Phi^{\leftarrow}\left(s^{-\beta_0}(1-\epsilon)\right) < Z_{[k]} < \Phi^{\leftarrow}\left(s^{-\beta_0}(1+\epsilon)\right)\right] \to 1.
\end{equation}
Using the expression for standard Gaussian quantiles (see, e.g., Proposition 1.1. in \cite{gao2020fundamental}), we know that
\begin{align*}
    \Phi^{\leftarrow}\left(s^{-\beta_0}(1-\epsilon)\right) 
    &\sim -\sqrt{2\log{\left(s^{\beta_0}/(1-\epsilon)\right)}} \\
    &= -\sqrt{2(\beta_0\log{s} - \log{(1-\epsilon)})} \sim -\sqrt{2\beta_0\log{s}},
\end{align*}
and similarly $\Phi^{\leftarrow}\left(s^{-\beta_0}(1+\epsilon)\right)\sim -\sqrt{2\beta_0\log{s}}$.
Since both ends of the interval in \eqref{eq:relative-stability-order-statistics-proof-2} are asymptotic to $-\sqrt{2\beta_0\log{s}}$, 
% dividing through by $\sqrt{2\log{s}}$, and 
the desired conclusion follows.
\end{proof}

\subsection{Proof of Theorems \ref{thm:additive-error-exact-approx-boundary} and \ref{thm:additive-error-approx-exact-boundary}}
\label{subsec:proof-additive-error-mix-boundaries}

Proof of Theorem \ref{thm:additive-error-exact-approx-boundary} is analogous to that of Theorem \ref{thm:chi-squared-exact-approx-boundary}.

\begin{proof}[Proof of Theorem \ref{thm:additive-error-exact-approx-boundary}]
We first show the sufficient condition.
Similar to the proof of Theorem Theorem \ref{thm:chi-squared-exact-approx-boundary}, it suffices to show that
\begin{equation} \label{eq:additive-error-exact-approx-boundary-proof-sufficient-1}
    \text{NDP} = 1 - \widehat{W}_\text{signal}(t_p) \to 0,
\end{equation}
where $t_p$ is the threshold of Bonferroni's procedure.

By the expression for normal quantiles, we know that 
$$
t_p=F^\leftarrow(1-\alpha/p)\sim(2\log{p}-2\log{\alpha})^{1/2} \sim(2\log{p})^{1/2}.
$$
where the last asymptotic equivalence follows from our choice of $\alpha$ and \eqref{eq:chi-square-sufficient-0}.
Since $\underline{r}>\widetilde{g}(\beta)=1$, we can pick $q$ such that $1<q<\underline{r}$.
Let $t^* = \sqrt{2q\log{p}}$, we know that $t_p<t_p^*$ for large $p$.
Therefore for large $p$, we have
$$
\widehat{W}_\text{signal}(t_p) \ge \widehat{W}_\text{signal}(t^*) \ge \overline{F_a}(t^*) + o_\P(1),
$$
where $\overline{F_a}$ is the survival function of $\mathrm{N}(\sqrt{2\underline{r}\log{p}}, 1)$; the last inequality follows from the stochastic monotonicity of the Gaussian location family, and Lemma \ref{lemma:empirical-process}.
Indeed, by our choice of $q<\underline{r}$, we obtain
$$
F_a(t^*) = \Phi\left(\sqrt{2(q-\underline{r})\log{p}}\right)\to0,
$$
and \eqref{eq:additive-error-exact-approx-boundary-proof-sufficient-1} is shown. 
This completes the proof of the sufficient condition.

The proof of the necessary condition, again, follows similar structure as in the proof of Theorem \ref{thm:chi-squared-approx-boundary}, and uses the lower bound
\begin{equation} \label{eq:additive-error-exact-approx-boundary-proof-necessary-1}
    \mathrm{FWER}(\mathcal{R}) + \mathrm{FNR}(\mathcal{R}) \ge \P\left[\max_{i\in S^c}x(i)>u\right] \wedge \E\left[\frac{|S\setminus \widehat{S}(u)|}{|S|}\right],
\end{equation}
which holds for any arbitrary thresholding procedure $\mathcal{R}$ and arbitrary real $u\in\R$.
By the assumption that $\overline{r}<\widetilde{g}(\beta)=1$, we can pick $q$ such that $\overline{r}<q<1$ and let $u = t^*=\sqrt{2q\log{p}}$ in \eqref{eq:additive-error-exact-approx-boundary-proof-necessary-1}.
By relative stability of iid Gaussian random variables (see, e.g., \cite{gao2020fundamental}), we have
\begin{equation} \label{eq:additive-error-exact-approx-boundary-proof-necessary-2}
    \P\left[\frac{\max_{i\in S^c} x(i)}{\sqrt{2\log{p}}} > \frac{t^*}{\sqrt{2\log{p}}}\right] \to 1.
\end{equation}
since the first fraction in \eqref{eq:additive-error-exact-approx-boundary-proof-necessary-2} converges to 1, while the second converges to $q<1$.
Therefore, the first term on the right-hand side of \eqref{eq:additive-error-exact-approx-boundary-proof-necessary-1} converges to 1.

On the other hand, by the stochastic monotonicity of Gaussian location family, the probability of missed detection for each signal is lower bounded by $\P[Z+\mu(i) \le t^*] \ge F_{\overline{a}}(t^*)$, where $Z$ is a standard Gaussian r.v., and $F_{\overline{a}}$ is the cdf of $\mathrm{N}(\sqrt{2\overline{r}\log{p}}, 1)$.
Therefore, $|{S}\setminus\widehat{S}(t^*)| \stackrel{\mathrm{d}}{\ge} \text{Binom}(s, {F_{\overline{a}}}(t^*))$, and it suffices to show that ${F_{\overline{a}}}(t^*)$ converges to 1.
Indeed,
\begin{equation*}
    {F_{\overline{a}}}(t^*) = \Phi(\sqrt{2(q-\overline{r})\log{p}}) \to 1,
\end{equation*}
by our choice of $q>\overline{r}$.
Hence both quantities in the minimum on the right-hand side of \eqref{eq:additive-error-exact-approx-boundary-proof-necessary-1} converge to 1 in the limit, and the necessary condition is shown.
\end{proof}

\begin{proof}[Proof of Theorem \ref{thm:additive-error-approx-exact-boundary}]
We first show the sufficient condition.
As in the proof of Theorem \ref{thm:chi-squared-approx-exact-boundary}, we only need to show that with a specific choice of $t^*=\sqrt{2q\log{p}}$ where
\begin{equation} \label{eq:additive-error-approx-exact-boundary-proof-sufficient-1}
\sqrt{\underline{r}}-\sqrt{1-\beta}>\sqrt{q}>\sqrt{\beta},
\end{equation}
we have both
\begin{equation} \label{additive-error-eq:approx-exact-boundary-proof-sufficient-2}
\P\left[\tau\le t^*\right]\to 1,
\end{equation}
and 
\begin{equation} \label{eq:additive-error-approx-exact-boundary-proof-sufficient-3}
    \P\left[\min_{i\in S}x(i) \ge t^* \right] \to 1,
\end{equation}
so that 
\begin{equation*} 
    \P\left[\min_{i\in S}x(i) \ge \tau\right] \ge 
    % \P\left[\min_{i\in S}x(i) \ge t^* \ge \tau\right] \ge
    \P\left[\min_{i\in S}x(i) \ge t^*,\; t^* \ge \tau\right] \to 1.
\end{equation*}

Relation \eqref{additive-error-eq:approx-exact-boundary-proof-sufficient-2} follows in very much the same way \eqref{eq:approx-exact-boundary-proof-sufficient-3} did on page \pageref{eq:approx-exact-boundary-proof-sufficient-3}, the only difference being how the tail estimates for \eqref{eq:approx-boundary-proof-sufficient-4} is derived.
In the Gaussian model, we have
\begin{align*}
    G(t^*) 
    &= \frac{p-s}{p}\overline{\Phi}(t^*) + \frac{s}{p}\overline{\Phi}(t^*-\sqrt{2\underline{r}\log{p}}) \\
    % &\sim \frac{\phi(t^*)}{t^*} + p^{-\beta} \overline{\Phi}((\underline{q} - \underline{r})\sqrt{2\log{p}}) \\
    &\sim \frac{\phi(t^*)}{t^*} + p^{-\beta} 
    \sim L(p)p^{-q} + p^{-\beta} \sim p^{\beta},
\end{align*}
where the last asymptotic equivalence follows from our choice of $q>\beta$.

Relation \eqref{eq:additive-error-approx-exact-boundary-proof-sufficient-3} follows since 
$$
\frac{\min_{i\in S}x(i)}{\sqrt{2\log{p}}} 
= \frac{\min_{i\in S}x(i)}{\sqrt{2\log{p}}} 
\to -\sqrt{1-\beta} + \sqrt{\underline{r}},
$$
by relative stability of iid Gaussians. On the other hand, ${t^*}/{\sqrt{2\log{p}}}=\sqrt{q}<\sqrt{\underline{r}}-\sqrt{1-\beta}$, by our choice of ${q}$.

The necessary condition, again, follows from the lower bound
\begin{equation} \label{eq:additive-error-approx-exact-boundary-proof-necessary-1}
    \mathrm{FDR}(\mathcal{R}) + \mathrm{FWNR}(\mathcal{R}) \ge \E\left[\frac{|\widehat{S}(u)\setminus S|}{|\widehat{S}(u)\setminus S| + |S|}\right] \wedge 
    \P\left[\min_{i\in S}x(i)<u\right],
\end{equation}
which holds for any thresholding procedure $\mathcal{R}$ and for arbitrary $u\in\R$.
In particular, we show that both terms in the minimum in \eqref{eq:additive-error-approx-exact-boundary-proof-necessary-1} converge to 1 when we set $u=t^*=\sqrt{2q\log{p}}$ where 
\begin{equation}
\sqrt{\overline{r}}-\sqrt{1-\beta}<\sqrt{q}<\sqrt{\beta}.
\end{equation}

On the one hand, we have,
$$
\frac{\min_{i\in S}x(i)}{\sqrt{2\log{p}}} 
\stackrel{\mathrm{d}}{\le} \frac{\min_{i\in S}\epsilon(i)+\sqrt{2\overline{r}\log{p}}}{\sqrt{2\log{p}}} 
\to \sqrt{\overline{r}}-\sqrt{1-\beta},
$$
by relative stability of iid Gaussians. On the other hand, ${t^*}/{\sqrt{2\log{p}}}=\sqrt{q}>\sqrt{\underline{r}}-\sqrt{1-\beta}$ by our choice of ${q}$;
this shows that the second term on the right-hand side of \eqref{eq:additive-error-approx-exact-boundary-proof-necessary-1} converges to 1.

Observe that $|\widehat{S}(t^*)\setminus{S}|$ has distribution $\text{Binom}(p-s, \overline{\Phi}(t^*))$, and define $X = X_p := {|\widehat{S}(t^*)\setminus{S}|}/{|S|}$, we obtain,
% On the other hand, define $ = \E[|S\setminus\widehat{S}(t^*)|/|S|]$,
\begin{align*}
    \mu &:= \E[X] = (p^\beta-1)\overline{\Phi}(t^*) 
    \sim (p^\beta-1)\frac{\phi(t^*)}{t^*} \\
    &\sim \frac{1}{\sqrt{2\pi}}\left(2q\log{p}\right)^{-1/2}p^{\beta-q}\to\infty,
\end{align*}
where the divergence follows from our choice of $q<\beta$.
Using again Relations \eqref{eq:approx-boundary-proof-converse-2} and \eqref{eq:approx-boundary-proof-converse-3}, we conclude that the first term on the right-hand side of \eqref{eq:additive-error-approx-exact-boundary-proof-necessary-1} also converges to 1.
This completes the proof of the necessary condition.
\end{proof}

\subsection{Signal sizes and odds ratios in 2-by-2 contingency tables}
\label{subsec:proof-signal-size-odds-ratio}

\begin{proof}[Proof of Proposition \ref{prop:signal-size-odds-ratio} and Corollary \ref{cor:signal-limits-OR}]
We parametrize the 2-by-2 multinomial distribution with the parameter $\delta$, 
\begin{equation} \label{eq:reparametrize-2-by-2-table-1}
    \mu_{11} = \phi_1\theta_1+\delta,\quad 
    \mu_{12} = \phi_1\theta_2-\delta,\quad 
    \mu_{21} = \phi_2\theta_1-\delta,\quad 
    \mu_{22} = \phi_2\theta_2+\delta.
\end{equation}
By relabelling of categories, we may assume $0<\theta_1,\phi_1\le1/2$ without loss of generality.
Note that $\delta$ must lie within the range $[\delta_\mathrm{min}, \delta_\mathrm{max}]$, where
$$
\delta_\mathrm{min} := \max\{-\phi_1\theta_1, -\phi_2\theta_2, \phi_1\theta_2-1, \phi_2\theta_1-1\} 
= -\phi_1\theta_1,
$$
and
$$
\delta_\mathrm{max} := \min\{1-\phi_1\theta_1, 1-\phi_2\theta_2, \phi_1\theta_2, \phi_2\theta_1\}
= \min\{\phi_1\theta_2, \phi_2\theta_1\},
$$
in order for $\mu_{ij}\ge0$ for all $i,j\in \{1,2\}$.
Under this parametrization, Relation \eqref{eq:odds-ratio} then becomes
\begin{equation} \label{eq:odds-ratio-delta}
    \text{R} = \frac{\mu_{11}\mu_{22}}{\mu_{12}\mu_{21}}
    = \frac{\phi_1\theta_1\phi_2\theta_2 + \delta(\phi_1\theta_1+\phi_2\theta_2)+\delta^2}{\phi_1\theta_1\phi_2\theta_2 - \delta(\phi_1\theta_2+\phi_2\theta_1)+\delta^2},
\end{equation}
which is one-to-one and increasing in $\delta$ on $(\delta_\mathrm{min}, \delta_\mathrm{max})$.
Equation \eqref{eq:signal-size-chisq} becomes
\begin{equation} \label{eq:signal-size-chisq-delta}
w^2 = \sum_{i=1}^2 \sum_{j=1}^2 \frac{(\mu_{ij} - \phi_i\theta_j)^2}{\phi_i\theta_j}
= \delta^2\sum_i\sum_j \frac{1}{\phi_i\theta_j}
= \frac{\delta^2}{\phi_1\theta_1\phi_2\theta_2},
\end{equation}
Solving for $\delta$ in \eqref{eq:odds-ratio-delta}, and plugging into the expression for signal size \eqref{eq:signal-size-chisq-delta} yields Relation \eqref{eq:signal-size-odds-ratio}.
Corollary \ref{cor:signal-limits-OR} follows from the fact that $w^2(\delta)$ is decreasing on $[\delta_\mathrm{min},0)$, increasing on $(0,\delta_\mathrm{max}]$, with limits
$$
\lim_{d\to \delta_\mathrm{min}} w^2(\delta) = \frac{\phi_1\theta_1}{\phi_2\theta_2},
\quad
\text{and}
\quad
\lim_{d\to \delta_\mathrm{max}} w^2(\delta) = \min\left\{\frac{\phi_1\theta_2}{\phi_2\theta_1}, \frac{\phi_2\theta_1}{\phi_1\theta_2}\right\}.
$$
The other three cases ($1/2\le\theta_1,\phi_1\le1$; $0<\theta_1\le1/2\le\phi_1\le1$; and $0\le\phi_1\le1/2\le\theta_1\le1$) may be obtained similarly, or by appealing to the symmetry of the problem.
\end{proof}

\begin{proof}[Proof of Corollary \ref{cor:optimal-design}]
Using the parametrization in \eqref{eq:reparametrize-2-by-2-table-1} and in Corollary \ref{cor:signal-size-odds-ratio-conditional-frequency}, we solve for $\delta$ in \eqref{eq:odds-ratio-delta} to obtain
\begin{align}
    \delta &= \frac{\phi_1 fR}{fR+1-f} - \left(\frac{\phi_1 fR}{fR+1-f} + f(1-\phi_1)\right)\phi_1 \nonumber \\
    &= \frac{f(1-f)\phi_1(1-\phi_1)(R-1)}{fR+1-f}. \label{eq:reparametrize-2-by-2-table-2}
\end{align}
Substituting \eqref{eq:reparametrize-2-by-2-table-2} into the expression \eqref{eq:signal-size-chisq-delta}, after some simplification, yields
\begin{equation} \label{eq:reparametrize-2-by-2-table-3}
    w^2 = \frac{f(1-f)\phi_1(1-\phi_1)(R-1)^2}{\left[\phi_1 R + (1-\phi_1)D\right]\left[\phi_1 + (1-\phi_1)D\right]},
\end{equation}
where $D = fR+1-f > 0$.
Therefore, he derivative of \eqref{eq:reparametrize-2-by-2-table-3} with respect to $\phi_1$ is
\begin{equation} \label{eq:signal-size-first-derivative}
    \frac{\mathrm{d}w^2}{\mathrm{d}\phi_1} = 
    \frac{f(1-f)(R-1)^2}{\left[\phi_1 R+(1-\phi_1)D\right]^2 \left[\phi_1+(1-\phi_1)D\right]^2} \left[(D^2-R)\phi_1^2 - 2D^2\phi_1 + D^2\right].
\end{equation}
Further, we obtain the second derivative with respect to $\phi_1$,
\begin{equation} \label{eq:signal-size-second-derivative}
    \frac{\mathrm{d}^2w^2}{\mathrm{d}\phi_1^2} = 
    h(R,f) \left[(\phi_1-1)D^2 - \phi_1R\right],
\end{equation}
where $h$ is some function of $(R,f)$ taking on strictly positive values.

Since $\left[(\phi_1-1)D^2 - \phi_1R\right]<0$, the second derivative \eqref{eq:signal-size-second-derivative} must be strictly negative on $[0,1]$.
This implies that the first derivative \eqref{eq:signal-size-first-derivative} is strictly decreasing on $[0,1]$. 
Since the first derivative \eqref{eq:signal-size-first-derivative} is strictly positive at $\phi_1=0$, and strictly negative at $\phi_1=1$, it must have a unique zero between 0 and 1, and hence, the solution to $(D^2-R)\phi_1^2 - 2D^2\phi_1 + D^2 = 0$ in the interval of $[0,1]$ must be the maximizer of \eqref{eq:reparametrize-2-by-2-table-3} --- when $D^2-R>0$, the smaller of the two roots maximizes \eqref{eq:reparametrize-2-by-2-table-3}, and when $D^2-R<0$, it is the larger of the two.
They share the same expression ${D}/{(D+\sqrt{R})}$, which coincides with \eqref{eq:optimal-design}.
Finally, when $D^2=R$, the only root $\phi_1^*=1/2$, which also coincides with \eqref{eq:optimal-design}, is the maximizer of \eqref{eq:reparametrize-2-by-2-table-3}.
\end{proof}

\section*{Acknowledgements}
I thank Professor Stilian Stoev for encouraging me to pursue the current project, and for his most generous support.
I also thank Jingshen Wang and Yuanzhi Li for their helpful comments and thoughtful feedback, and Rafail Kartsioukas and Jinqi Shen for their careful proof-reading of the manuscript.
This research is supported by NSF Grant DMS-1830293, Algorithms for Threat Detection.
\bibliographystyle{plainnat}
\bibliography{references}

\begin{thebibliography}{32}
\providecommand{\natexlab}[1]{#1}
\providecommand{\url}[1]{\texttt{#1}}
\expandafter\ifx\csname urlstyle\endcsname\relax
  \providecommand{\doi}[1]{doi: #1}\else
  \providecommand{\doi}{doi: \begingroup \urlstyle{rm}\Url}\fi

\bibitem[Agresti(2018)]{agresti2018introduction}
Alan Agresti.
\newblock \emph{An introduction to categorical data analysis}.
\newblock Wiley, 2018.

\bibitem[Arias-Castro and Chen(2017)]{arias2017distribution}
Ery Arias-Castro and Shiyun Chen.
\newblock Distribution-free multiple testing.
\newblock \emph{Electronic Journal of Statistics}, 11\penalty0 (1):\penalty0
  1983--2001, 2017.

\bibitem[Barber and Cand{\`e}s(2015)]{barber2015controlling}
Rina~Foygel Barber and Emmanuel~J Cand{\`e}s.
\newblock Controlling the false discovery rate via knockoffs.
\newblock \emph{The Annals of Statistics}, 43\penalty0 (5):\penalty0
  2055--2085, 2015.

\bibitem[Benjamini and Hochberg(1995)]{benjamini1995controlling}
Yoav Benjamini and Yosef Hochberg.
\newblock Controlling the false discovery rate: a practical and powerful
  approach to multiple testing.
\newblock \emph{Journal of the royal statistical society. Series B
  (Methodological)}, pages 289--300, 1995.

\bibitem[Bogdan et~al.(2011)Bogdan, Chakrabarti, Frommlet, and
  Ghosh]{bogdan2011asymptotic}
Ma{\l}gorzata Bogdan, Arijit Chakrabarti, Florian Frommlet, and Jayanta~K
  Ghosh.
\newblock Asymptotic bayes-optimality under sparsity of some multiple testing
  procedures.
\newblock \emph{The Annals of Statistics}, 39\penalty0 (3):\penalty0
  1551--1579, 2011.

\bibitem[Bush and Moore(2012)]{bush2012genome}
William~S Bush and Jason~H Moore.
\newblock Genome-wide association studies.
\newblock \emph{PLoS computational biology}, 8\penalty0 (12):\penalty0
  e1002822, 2012.

\bibitem[Butucea et~al.(2018)Butucea, Ndaoud, Stepanova, and
  Tsybakov]{butucea2018variable}
Cristina Butucea, Mohamed Ndaoud, Natalia~A. Stepanova, and Alexandre~B.
  Tsybakov.
\newblock Variable selection with {H}amming loss.
\newblock \emph{The Annals of Statistics}, 46\penalty0 (5):\penalty0
  1837--1875, 2018.

\bibitem[Chen et~al.(2018)Chen, Ying, and Arias-Castro]{chen2018scan}
Shiyun Chen, Andrew Ying, and Ery Arias-Castro.
\newblock A scan procedure for multiple testing.
\newblock \emph{arXiv preprint arXiv:1808.00631}, 2018.

\bibitem[Donoho and Jin(2004)]{donoho2004higher}
David Donoho and Jiashun Jin.
\newblock Higher {C}riticism for detecting sparse heterogeneous mixtures.
\newblock \emph{The Annals of Statistics}, 32\penalty0 (3):\penalty0 962--994,
  2004.

\bibitem[Dunn(1961)]{dunn1961multiple}
Olive~Jean Dunn.
\newblock Multiple comparisons among means.
\newblock \emph{Journal of the American statistical association}, 56\penalty0
  (293):\penalty0 52--64, 1961.

\bibitem[Efron(2004)]{efron2004large}
Bradley Efron.
\newblock Large-scale simultaneous hypothesis testing: the choice of a null
  hypothesis.
\newblock \emph{Journal of the American Statistical Association}, 99\penalty0
  (465):\penalty0 96--104, 2004.

\bibitem[Eicker(1979)]{eicker1979asymptotic}
F~Eicker.
\newblock The asymptotic distribution of the suprema of the standardized
  empirical processes.
\newblock \emph{The Annals of Statistics}, pages 116--138, 1979.

\bibitem[Ferguson(2017)]{ferguson2017course}
Thomas~S Ferguson.
\newblock \emph{A course in large sample theory}.
\newblock Routledge, 2017.

\bibitem[Gao and Stoev(2020)]{gao2020fundamental}
Zheng Gao and Stilian Stoev.
\newblock Fundamental limits of exact support recovery in high dimensions.
\newblock \emph{Bernoulli}, 26\penalty0 (4):\penalty0 2605--2638, 2020.

\bibitem[Gao et~al.(2019)Gao, Terhorst, Van~Hout, and Stoev]{gao2019upass}
Zheng Gao, Jonathan Terhorst, Cristopher~V Van~Hout, and Stilian Stoev.
\newblock {U-PASS: unified power analysis and forensics for qualitative traits
  in genetic association studies}.
\newblock \emph{Bioinformatics}, 2019.
\newblock \doi{10.1093/bioinformatics/btz637}.
\newblock URL \url{https://doi.org/10.1093/bioinformatics/btz637}.

\bibitem[Genovese and Wasserman(2002)]{genovese2002operating}
Christopher Genovese and Larry Wasserman.
\newblock Operating characteristics and extensions of the false discovery rate
  procedure.
\newblock \emph{Journal of the Royal Statistical Society: Series B (Statistical
  Methodology)}, 64\penalty0 (3):\penalty0 499--517, 2002.

\bibitem[Genovese et~al.(2012)Genovese, Jin, Wasserman, and
  Yao]{genovese2012comparison}
Christopher~R Genovese, Jiashun Jin, Larry Wasserman, and Zhigang Yao.
\newblock A comparison of the lasso and marginal regression.
\newblock \emph{The Journal of Machine Learning Research}, 13\penalty0
  (1):\penalty0 2107--2143, 2012.

\bibitem[Gnedenko(1943)]{gnedenko1943distribution}
Boris Gnedenko.
\newblock Sur la distribution limite du terme maximum d'une serie aleatoire.
\newblock \emph{Annals of mathematics}, pages 423--453, 1943.

\bibitem[Hochberg(1988)]{hochberg1988sharper}
Yosef Hochberg.
\newblock A sharper bonferroni procedure for multiple tests of significance.
\newblock \emph{Biometrika}, 75\penalty0 (4):\penalty0 800--802, 1988.

\bibitem[Holm(1979)]{holm1979simple}
Sture Holm.
\newblock A simple sequentially rejective multiple test procedure.
\newblock \emph{Scandinavian journal of statistics}, pages 65--70, 1979.

\bibitem[Inglot(2010)]{inglot2010inequalities}
Tadeusz Inglot.
\newblock Inequalities for quantiles of the chi-square distribution.
\newblock \emph{Probability and Mathematical Statistics}, 30\penalty0
  (4):\penalty0 339--351, 2010.

\bibitem[Ji and Jin(2012)]{ji2012ups}
Pengsheng Ji and Jiashun Jin.
\newblock {UPS} delivers optimal phase diagram in high-dimensional variable
  selection.
\newblock \emph{The Annals of Statistics}, 40\penalty0 (1):\penalty0 73--103,
  2012.

\bibitem[Jin and Ke(2016)]{jin2016rare}
Jiashun Jin and Zheng~Tracy Ke.
\newblock Rare and weak effects in large-scale inference: methods and phase
  diagrams.
\newblock \emph{Statistica Sinica}, pages 1--34, 2016.

\bibitem[Jin et~al.(2014)Jin, Zhang, and Zhang]{jin2014optimality}
Jiashun Jin, Cun-Hui Zhang, and Qi~Zhang.
\newblock Optimality of graphlet screening in high dimensional variable
  selection.
\newblock \emph{The Journal of Machine Learning Research}, 15\penalty0
  (1):\penalty0 2723--2772, 2014.

\bibitem[Ke et~al.(2014)Ke, Jin, and Fan]{ke2014covariance}
Tracy Ke, Jiashun Jin, and Jianqing Fan.
\newblock Covariance assisted screening and estimation.
\newblock \emph{Annals of statistics}, 42\penalty0 (6):\penalty0 2202, 2014.

\bibitem[MacArthur et~al.(2016)MacArthur, Bowler, Cerezo, Gil, Hall, Hastings,
  Junkins, McMahon, Milano, Morales, et~al.]{macarthur2016new}
Jacqueline MacArthur, Emily Bowler, Maria Cerezo, Laurent Gil, Peggy Hall, Emma
  Hastings, Heather Junkins, Aoife McMahon, Annalisa Milano, Joannella Morales,
  et~al.
\newblock The new nhgri-ebi catalog of published genome-wide association
  studies (gwas catalog).
\newblock \emph{Nucleic acids research}, 45\penalty0 (D1):\penalty0 D896--D901,
  2016.

\bibitem[Michailidou et~al.(2017)Michailidou, Lindstr{\"o}m, Dennis, Beesley,
  Hui, Kar, Lema{\c{c}}on, Soucy, Glubb, Rostamianfar,
  et~al.]{michailidou2017association}
Kyriaki Michailidou, Sara Lindstr{\"o}m, Joe Dennis, Jonathan Beesley, Shirley
  Hui, Siddhartha Kar, Audrey Lema{\c{c}}on, Penny Soucy, Dylan Glubb, Asha
  Rostamianfar, et~al.
\newblock Association analysis identifies 65 new breast cancer risk loci.
\newblock \emph{Nature}, 551\penalty0 (7678):\penalty0 92, 2017.

\bibitem[Neuvial and Roquain(2012)]{neuvial2012false}
Pierre Neuvial and Etienne Roquain.
\newblock On false discovery rate thresholding for classification under
  sparsity.
\newblock \emph{The Annals of Statistics}, 40\penalty0 (5):\penalty0
  2572--2600, 2012.

\bibitem[{\v{S}}id{\'a}k(1967)]{vsidak1967rectangular}
Zbyn{\v{e}}k {\v{S}}id{\'a}k.
\newblock Rectangular confidence regions for the means of multivariate normal
  distributions.
\newblock \emph{Journal of the American Statistical Association}, 62\penalty0
  (318):\penalty0 626--633, 1967.

\bibitem[Storey(2007)]{storey2007optimal}
John~D Storey.
\newblock The optimal discovery procedure: a new approach to simultaneous
  significance testing.
\newblock \emph{Journal of the Royal Statistical Society: Series B (Statistical
  Methodology)}, 69\penalty0 (3):\penalty0 347--368, 2007.

\bibitem[Sun and Cai(2007)]{sun2007oracle}
Wenguang Sun and T~Tony Cai.
\newblock Oracle and adaptive compound decision rules for false discovery rate
  control.
\newblock \emph{Journal of the American Statistical Association}, 102\penalty0
  (479):\penalty0 901--912, 2007.

\bibitem[Wainwright(2019)]{wainwright2019high}
Martin~J Wainwright.
\newblock \emph{High-dimensional statistics: A non-asymptotic viewpoint}.
\newblock Cambridge University Press, Cambridge, {U}{K}, 2019.

\end{thebibliography}

% \begin{thebibliography}{9}
% 
% \bibitem{r1}
% \textsc{Billingsley, P.} (1999). \textit{Convergence of
% Probability Measures}, 2nd ed.
% Wiley, New York.
% \MR{1700749}
% 
% 
% \bibitem{r2}
% \textsc{Bourbaki, N.}  (1966). \textit{General Topology}  \textbf{1}.
% Addison--Wesley, Reading, MA.
% 
% \bibitem{r3}
% \textsc{Ethier, S. N.} and \textsc{Kurtz, T. G.} (1985).
% \textit{Markov Processes: Characterization and Convergence}.
% Wiley, New York.
% \MR{838085}
% 
% \bibitem{r4}
% \textsc{Prokhorov, Yu.} (1956).
% Convergence of random processes and limit theorems in probability
% theory. \textit{Theory  Probab.  Appl.}
% \textbf{1} 157--214.
% \MR{84896}
% 
% \end{thebibliography}

\end{document}